\documentclass[a4paper,10pt]{amsart}

\usepackage{amsmath}
\usepackage{amsthm}
\usepackage{amssymb}
\usepackage{amsfonts}
\usepackage{enumitem} 
\usepackage{mathrsfs} 
\usepackage{amscd}
\usepackage{xcolor}
\usepackage[bookmarks=true,hyperindex,pdftex,colorlinks,citecolor=blue]{hyperref}%
\usepackage[a4paper,lmargin=3cm,rmargin=3cm,tmargin=4cm,bmargin=4cm,marginparwidth=2.8cm,marginparsep=1mm]{geometry} 
\usepackage{marginnote} 
\usepackage{soul} 

\makeatletter
\@namedef{subjclassname@2020}{\textup{2020} Mathematics Subject Classification}
\makeatother

\makeatletter
\renewcommand{\tocsection}[3]{%
  \indentlabel{\@ifnotempty{#2}{\bfseries\ignorespaces#1 #2\quad}}\bfseries#3}
\renewcommand{\tocsubsection}[3]{%
  \indentlabel{\@ifnotempty{#2}{\ignorespaces#1 #2\quad}}#3}

\newcommand\@dotsep{4.5}
\def\@tocline#1#2#3#4#5#6#7{\relax
  \ifnum #1>\c@tocdepth 
  \else
    \par \addpenalty\@secpenalty\addvspace{#2}%
    \begingroup \hyphenpenalty\@M
    \@ifempty{#4}{%
      \@tempdima\csname r@tocindent\number#1\endcsname\relax
    }{%
      \@tempdima#4\relax
    }%
    \parindent\z@ \leftskip#3\relax \advance\leftskip\@tempdima\relax
    \rightskip\@pnumwidth plus1em \parfillskip-\@pnumwidth
    #5\leavevmode\hskip-\@tempdima{#6}\nobreak
    \leaders\hbox{$\m@th\mkern \@dotsep mu\hbox{.}\mkern \@dotsep mu$}\hfill
    \nobreak
    \hbox to\@pnumwidth{\@tocpagenum{\ifnum#1=1\bfseries\fi#7}}\par
    \nobreak
    \endgroup
  \fi}
\AtBeginDocument{%
\expandafter\renewcommand\csname r@tocindent0\endcsname{0pt}
}
\def\l@subsection{\@tocline{2}{0pt}{2.5pc}{5pc}{}}
\makeatother

\DeclareMathOperator{\lspan}{span}                          
\DeclareMathOperator{\supp}{supp}                           

\DeclareMathOperator{\ext}{ext}                             
\DeclareMathOperator{\Lip}{Lip}                             
\DeclareMathOperator{\lip}{lip}                             

\newcommand{\NN}{\mathbb{N}}                                
\newcommand{\RR}{\mathbb{R}}                                

\newcommand{\abs}[1]{\left|{#1}\right|}                     
\newcommand{\Glsc}{\mathfrak{G}^{\mathrm{lsc}}}
\newcommand{\esupp}[1]{\mathcal{S}(#1)}                     
\newcommand{\pare}[1]{\left({#1}\right)}                    
\newcommand{\set}[1]{\left\{{#1}\right\}}                   
\newcommand{\norm}[1]{\left\|{#1}\right\|}                  
\newcommand{\dual}[1]{{#1}^\ast}                            
\newcommand{\duality}[1]{\left<{#1}\right>}                 
\newcommand{\cl}[1]{\overline{#1}}                          
\newcommand{\restrict}{\mathord{\upharpoonright}}           
\newcommand{\lipfree}[1]{\mathcal{F}({#1})}                 
\newcommand{\lipnorm}[1]{\norm{#1}_L}                       
\newcommand{\meas}[1]{\mathcal{M}({#1})}                    
\newcommand{\wt}[1]{\widetilde{#1}}                         
\newcommand{\bwt}[1]{\beta\wt{#1}}                          
\newcommand{\pp}{\mathfrak{p}}                              
\newcommand{\rr}{\mathfrak{r}}                              
\newcommand{\opr}[1]{\mathcal{M}_{\mathrm{op}}(#1)}         
\newcommand{\rcomp}[1]{#1^\mathcal{R}}                    
\newcommand{\ucomp}[1]{#1^\mathcal{U}}                    

\newcommand{\ucl}[1]{\overline{#1}^{\mathcal{U}}}           
\newcommand{\rcl}[1]{\overline{#1}^{\mathcal{R}}}           

\newcommand{\bwtf}{\mathsf{R}}                              
\newcommand{\bwtfp}{\mathsf{R}^+}                           
\newcommand{\opt}{\mathsf{O}}                               

\newcommand{\uwt}[1]{\ucomp{\wt{#1}}}                          
\newcommand{\uwtf}{\mathsf{U}}                                

\newcommand{\uu}{\mathfrak{u}}                              
\newcommand{\vv}{\mathfrak{v}}                              
\newcommand{\gq}{\mathfrak{g}}                              

\renewcommand{\leq}{\leqslant}
\renewcommand{\geq}{\geqslant}
\newcommand{\ep}{\varepsilon}

\DeclareMathOperator{\cnv}{conv}                             
\DeclareMathOperator{\Mol}{Mol}                              

\makeatletter
\newcommand{\labeltext}[2]{%
  \@bsphack
  \csname phantomsection\endcsname 
  \def\@currentlabel{#1}{\label{#2}}%
  \@esphack
}
\makeatother

\allowdisplaybreaks

\theoremstyle{plain}
\newtheorem{theorem}{Theorem}[section]
\newtheorem{lemma}[theorem]{Lemma}
\newtheorem{corollary}[theorem]{Corollary}
\newtheorem{proposition}[theorem]{Proposition}
\newtheorem*{claim}{Claim}

\theoremstyle{definition}
\newtheorem*{definition*}{Definition}
\newtheorem{definition}[theorem]{Definition}
\newtheorem{example}[theorem]{Example}
\newtheorem{question}[theorem]{Question}

\theoremstyle{remark}
\newtheorem{remark}[theorem]{Remark}

\numberwithin{equation}{section}

\begin{document}

\title[A Choquet theory of Lipschitz-free spaces]{A Choquet theory of Lipschitz-free spaces}

\author[R. J. Smith]{Richard J. Smith}
\address[R. J. Smith]{School of Mathematics and Statistics, University College Dublin, Belfield, Dublin 4, Ireland}
\email{richard.smith@maths.ucd.ie}

\date{}


\begin{abstract}
Let $(M,d)$ be a complete metric space and let $\lipfree{M}$ denote the Lipschitz-free space over $M$. We develop a ``Choquet theory of Lipschitz-free spaces'' that draws from the classical Choquet theory and the De Leeuw representation of elements of $\lipfree{M}$ (and its bidual) by positive Radon measures on $\bwt{M}$, where $\wt{M}$ is the space of pairs $(x,y) \in M \times M$, $x \neq y$. We define a quasi-order $\preccurlyeq$ on the positive Radon measures on $\bwt{M}$ that is analogous to the classical Choquet order. Rather than in the classical case where the focus lies on maximal measures, we study the $\preccurlyeq$-minimal measures and show that they have a host of desirable properties. Among the applications of this theory is a solution (given elsewhere) to the extreme point problem for Lipschitz-free spaces.  
\end{abstract}



\maketitle

\tableofcontents

\section{Introduction}

\subsection{Background and preliminaries}\label{subsec:background}

Let $(M,d)$ be a \emph{pointed} metric space, that is, one with a fixed base point which we denote by $0$. We define the \emph{Lipschitz space} $\Lip_0(M)$ to be the vector space of all Lipschitz functions $f:M\to\RR$ satisfying $f(0)=0$. On this space the (optimal) Lipschitz constant of $f$ defines a norm:
\[
\lipnorm{f}=\Lip(f),\quad f \in \Lip_0(M),
\]
and with respect to this norm $\Lip_0(M)$ becomes a Banach space.

There is a natural isometric embedding $\delta:M\to \Lip_0(M)^*$ that is given by evaluation: $\duality{f,\delta(x)}=f(x)$, $f \in \Lip_0(M)$. We define the \emph{Lipschitz-free space} (or simply \emph{free space}) over $M$ to be the norm-closed linear span of the image of $M$ under $\delta$:
\[
 \lipfree{M} = \cl{\lspan}(\delta(M)) \subset \Lip_0(M)^*.
\]
It is a standard fact that $\lipfree{M}$ is a canonical isometric predual of $\Lip_0(M)$, and hereafter we write $\Lip_0(M)=\lipfree{M}^*$. Since $\lipfree{M}$ is easily seen to be isometric to the free space over the completion of $M$, from now on we will assume without loss of generality that $M$ is complete.

The elementary molecules are the basic building blocks of free spaces. We set
\[
 \wt{M} = \set{(x,y) \in M \times M \,:\, x \neq y}
\]
and, given $(x,y) \in \wt{M}$, we define the \emph{(normalised) elementary molecule}
\[
 m_{xy} = \frac{\delta(x)-\delta(y)}{d(x,y)}.
\]
Since $\delta$ is an isometric embedding, each elementary molecule belongs to the unit sphere $S_{\lipfree{M}}$ of $\lipfree{M}$. Let
\[
\Mol = \Mol(M) = \set{m_{xy} \,:\, (x,y) \in \wt{M}}
\]
denote the set of all elementary molecules. Since $\Mol$ is 1-norming for $\Lip_0(M)$, by a straightforward Hahn-Banach separation argument, the closed unit ball of $\lipfree{M}$ is the norm-closed convex hull of $\Mol$:
\begin{equation}\label{eq:norming}
B_{\lipfree{M}}=\cl{\cnv}(\Mol).
\end{equation}

A wealth of information about Lipschitz and Lipschitz-free spaces (where the latter are referred to as Arens-Eells spaces and are denoted using the symbol {\AE } instead of $\mathcal{F}$) can be found in the pioneering texts of Weaver \cite{Weaver,Weaver2}. While free spaces have been known since the work of Arens and Eells \cite{AE56}, they have been studied intensively by members of the functional analysis community since the publication of the seminal papers of Godefroy and Kalton \cite{GK,Kalton}, where they were shown to play a key role in the non-linear theory of Banach spaces. Free spaces can be found at the interface between functional analysis, metric geometry and optimal transport theory. Together with their duals the Lipschitz spaces, free spaces are arguably the canonical way to express metric geometry in functional analytic terms, analogously to how compact Hausdorff spaces and measure spaces can be expressed using $C(K)$-spaces and $L_p$-spaces (and their duals), respectively. 

Despite the ease with which they can be defined, the structure of free spaces is complicated and elusive, and many pertinent and simply-stated questions concerning their isomorphic and isometric theory remain unanswered. We refer the reader to \cite{Godefroy_survey,GLZ} for details on some of these questions, as well as an abundance of information on the class of free spaces and their role in the wider Banach space context that will not be repeated here.

This paper was motivated by questions concerning the isometric theory of free spaces, for example, the characterisation of the extreme points of their unit balls. It is a fundamental question in Banach space geometry to characterise the set $\ext B_X$ of extreme points of the (closed) unit ball $B_X$ of a given Banach space $X$. In the class of free spaces, such a characterisation has been hard to obtain; we shall refer to the search for this characterisation as the \emph{extreme point problem}.

In view of \eqref{eq:norming}, it is reasonable to suppose that extreme points of $B_{\lipfree{M}}$ are related to the elementary molecules. Motivated by the question of classifying isometries between Lipschitz spaces, Weaver proved the first such result in this direction \cite[Theorem A]{W95}, which can be viewed as a result about preserved extreme points ($x \in B_X$ is a \emph{preserved extreme point} of $B_X$ if $x \in \ext B_{X^{**}}$). A few years after \cite{W95} was published, Weaver provided the following generalisation. (For the remainder of this section, we will not mention explicitly the fact that the metric spaces under consideration are pointed.)

\begin{theorem}[{cf.~\cite[Corollary 2.5.4]{Weaver}}]\label{th:preserved_extreme}
Let $M$ be a complete metric space and let $m$ be a preserved extreme point of $B_{\lipfree{M}}$. Then $m \in \Mol$.
\end{theorem}

Central to the argument (and the one in \cite{W95}) is a fundamental construction due to K. de Leeuw \cite{DeLeeuw}, which we expand on in Section \ref{subsec:DeLeeuw}. Equation \eqref{eq:norming} and Theorem \ref{th:preserved_extreme} prompt the following question.

\begin{question}\label{qu:extreme}
Given a complete metric space $M$, is $\ext B_{\lipfree{M}} \subset \Mol$?
\end{question}

Some partial positive answers to Question \ref{qu:extreme} are known. For example, in the late 1990s it was shown that the answer is positive whenever $M$ is compact and the space $\lip_0(M) \subset \Lip_0(M)$ of ``little Lipschitz functions'' is an isometric predual of $\lipfree{M}$ \cite[Corollary 3.3.6]{Weaver}. Recently, for compact $M$, having $\lip_0(M)$ as an isometric predual was shown to be equivalent to saying simply that $\lipfree{M}$ is a dual space \cite[Theorem B]{AGPP} or, equally, that $M$ is \emph{purely 1-unrectifiable}, which is equivalent to saying that $M$ contains no bi-Lipschitz copy of a compact subset of $\RR$ of positive Lebesgue measure (see \cite[Corollary 1.12]{AGPP}). Question \ref{qu:extreme} is also known to have a positive answer, for example, if $M$ is a subset of an $\RR$-tree \cite[Corollary 4.5]{APP}, if $M$ is a \emph{proper metric space} (its closed bounded subsets are compact) \cite[Theorem 1]{Aliaga}, or if $M$ is \emph{uniformly discrete} (meaning that $\inf_{x \neq y} d(x,y)>0$) \cite[Corollary 6.3]{APS24a}.

Recent and significant progress has also been made on studying the extreme points of $B_{\lipfree{M}}$ that happen to be elementary molecules. Let $x,u,y \in M$ be distinct and satisfy $d(x,y)=d(x,u)+d(u,y)$. Then it is easy to see that $m_{xy} \notin \ext B_{\lipfree{M}}$ because one can write
\[
 m_{xy} = \frac{d(x,u)}{d(x,y)}m_{xu} + \frac{d(u,y)}{d(x,y)}m_{uy}.
\]
The converse, namely that $m_{xy} \in \ext B_{\lipfree{M}}$ whenever $x,y \in M$ are distinct and $d(x,y)<d(x,u)+d(u,y)$ for all $p \in M\setminus\{x,y\}$, is a consequence of the next result. 

\begin{theorem}[{\cite[Theorem 1.1]{AP_rmi}}]\label{th:AP_extreme}
Let $M$ be a complete metric space and let $m \in \lipfree{M}$ be ``finitely supported'', that is $m \in \lspan(\delta(M))$. Then $m \in \ext B_{\lipfree{M}}$ if and only if $m = m_{xy}$ for some distinct $x,y \in M$ satisfying $d(x,y)<d(x,u)+d(u,y)$ whenever $p \in M\setminus\{x,y\}$. 
\end{theorem}

Thus, when $m \in \Mol$, we have a complete understanding of when it is an extreme point of $B_{\lipfree{M}}$. Despite these efforts and advances, it was only very recently that Question \ref{qu:extreme} finally yielded to pressure. The answer in full generality is positive and, when combined with Theorem \ref{th:AP_extreme}, we obtain a complete solution to the extreme point problem.

\begin{theorem}[{\cite[Theorem 3.1]{APS24c}}]\label{th:extreme}
Let $M$ be a complete metric space. Then $\ext B_{\lipfree{M}} \subset \Mol$. Moreover, $m \in \ext B_{\lipfree{M}}$ if and only if $m = m_{xy}$ for some distinct $x,y \in M$ satisfying $d(x,y)<d(x,u)+d(u,y)$ whenever $p \in M\setminus\{x,y\}$.
\end{theorem}

The reason why we mention the extreme point problem and Theorem \ref{th:extreme} here is because the proof of the latter, specifically the inclusion $\ext B_{\lipfree{M}} \subset \Mol$, relies critically on the ``Choquet theory of Lipschitz-free spaces'' developed herein. This theory is based on the De Leeuw transform \cite{DeLeeuw}, a recent systematic study of its properties \cite{APS24a,APS24b}, and the classical Choquet theory. We cover the necessary elements of these ingredients in the coming sections.

We briefly mention several other applications of this paper that can be found in full in \cite{APS24c}. Starting from the final result of this paper, Corollary \ref{co:nice_rep}, we formulate a ``compact decomposition principle'' \cite[Theorem 2.1]{APS24c}, which shows that all free space elements can be decomposed into a convex series of ``compactly supported'' elements. Using this principle, one can quickly prove Theorem \ref{th:extreme}. The theorem can also be proved starting from Corollary \ref{co:nice_rep} (or Theorem \ref{th:opt_conc}) and adapting arguments from \cite{Aliaga} (see \cite[Remark 3.2]{APS24c}). Combined with \cite[Theorem 3.2]{APPP_2020}, Theorem \ref{th:extreme} implies that every extreme point of $B_{\lipfree{M}}$ is also an exposed point \cite[Corollary 3.3]{APS24c}. Theorem \ref{th:extreme} also yields consequences for surjective linear isometries between free spaces over geodesic metric spaces and concave metric spaces \cite[Corollary 3.5 and Theorem 3.6]{APS24c}.  The compact decomposition principle also has implications for convex integrals of molecules in free spaces (the definition of this is given after Proposition \ref{pr:Bochner}) over purely 1-unrectifiable metric spaces \cite[Theorem 4.1]{APS24c}. Finally, results from this paper can be used to prove that every free space element can be decomposed into a convex sum of a convex integral of molecules and a so-called ``diagonal element'' which (if non-zero) can be regarded as being as far removed from a convex integral of molecules as possible \cite[Theorem 5.4]{APS24c}. Thus every free space decomposes into such elements. As well as these applications, in this paper we provide a few others pertaining to the theory developed in \cite{APS24a,APS24b}. We are hopeful that this work will yield further applications in the future. 

Finally, we make some brief remarks on conventions and notation. All Banach spaces in the paper will be taken over the field of real numbers. We refer the reader to \cite{FHHSPZ11} for general information on Banach spaces and for the definitions of any concepts not defined here. Concerning topology, every topological space that features in this paper will be Hausdorff. For general topology we recommend \cite{Engelking}. Regarding measure theory, every measure under consideration will be a real-valued signed Radon measure (that is, a regular, tight, finite and countably additive Borel measure) on a compact or locally compact (Hausdorff) topological space, and hence will identify via the Riesz representation theorem with a bounded linear functional on the space of continuous functions on said topological space (that vanish at infinity in the locally compact case). For measure theory we refer the reader to \cite{Bogachev}. For Choquet theory, we recommend \cite{lmns10}, where a lot of information on the other fields mentioned above can also be found. The notation used throughout will be standard and will be consistent with that used in \cite{APS24a,APS24b}.

\subsection{The De Leeuw transform and optimal representations}\label{subsec:DeLeeuw}

The De Leeuw transform \cite{DeLeeuw} is the foundation stone on which everything in this paper rests. It is a tool that has been shown to have key importance in the theory of Lipschitz-free and Lipschitz spaces (and their duals) \cite{Aliaga,AGPP,AG,AP_rmi,APS24a,APS24b,APS24c,W95,Weaver,Weaver2}. The essential point is that the De Leeuw transform can be used to represent elements of $\lipfree{M}$ and its bidual $\Lip_0(M)^*$ by Radon measures on the Stone-\v Cech compactification of the (completely metrisable) space $\wt{M}$ introduced above. Because of this, we can leverage the mature, well-developed theory of Radon measures and duals of $C(K)$-spaces to help us to obtain results about said elements of $\lipfree{M}$ and $\Lip_0(M)^*$, which are less well understood.

A systematic study of this tool and its impact on $\lipfree{M}$ and its bidual was initiated in \cite{APS24a,APS24b}, after the authors recognised that one was missing from the literature. This study has yielded results about extremal structure of free spaces (mentioned above), connections with optimal transport theory and L-embeddability of free spaces in their biduals, to mention a few areas. These papers provide comprehensive introductions to the De Leeuw transform and associated notions. As such, we shall be content to provide only the essential details here, and refer the reader to \cite{APS24a,APS24b} for everything else.

The space $\wt{M}$ defined above is completely metrisable, so we can consider its Stone-\v Cech compactification $\bwt{M}$ together with the space $\meas{\bwt{M}}$ of signed Radon measures on $\bwt{M}$, which we identify with $C(\bwt{M})^*$. The \emph{De Leeuw transform} is the map $\Phi:\Lip_0(M)\to C(\bwt{M})$ defined as follows. Given $f \in \Lip_0(M)$, we define
\[
 (\Phi f)(x,y) = \duality{f,m_{xy}} = \frac{f(x)-f(y)}{d(x,y)}, \quad (x,y) \in \wt{M},
\]
to produce a bounded continuous function on $\wt{M}$, which we then extend continuously to $\bwt{M}$. Since $\Phi$ is a linear isometric embedding into $C(\bwt{M})$, the dual operator $\Phi^*:\meas{\bwt{M}} \to \Lip_0(M)^*$ is a quotient map, meaning that $\Phi^*B_{\meas{\bwt{M}}}=B_{\Lip_0(M)^*}$.

We will say that $\mu \in \meas{\bwt{M}}$ is a \emph{(De Leeuw) representation} of $\psi \in \Lip_0(M)^*$ simply if $\Phi^*\mu=\psi$. Because $\Phi^*$ is a quotient map, we can choose $\mu$ so that $\norm{\mu}=\norm{\psi}$. Moreover, by considering the reflection map $\rr$ (see Section \ref{subsec:coords}), it is not difficult to see that such measures $\mu$ can also be chosen to be positive \cite[Proposition 3]{Aliaga}. We call such measures \emph{optimal (De Leeuw) representations} of $\psi$, and write
\[
 \opr{\bwt{M}} = \set{\mu \in \meas{\bwt{M}}^+ \,:\, \norm{\mu}=\norm{\Phi^*\mu}}
\]
to denote the set of all optimal representations of some element of $\Lip_0(M)^*$ \cite{APS24a}. The canonical examples of optimal representations are the Dirac measures $\delta_{(x,y)}$, $(x,y) \in \wt{M}$, which evidently represent the corresponding elementary molecules $m_{xy}$.

Much of \cite{APS24a,APS24b} is devoted to establishing the favourable properties enjoyed by these optimal representations and what they can tell us about the structure of $\lipfree{M}$ and its bidual $\Lip_0(M)^*$. The following result shows that representations yield free space elements if they are concentrated on the ``nice'' part of $\bwt{M}$, namely $\wt{M}$. Given a Borel set $E \subset \bwt{M}$ and $\mu \in \meas{\bwt{M}}$, we denote by $\mu\restrict_{E}$ the restriction of $\mu$ to $E$ defined by $\mu\restrict_{E}(A)=\mu(E \cap A)$, $A \subset \bwt{M}$ Borel. Observe that as $\wt{M}$ is completely metrisable, it is a $G_\delta$ and thus Borel subset of $\bwt{M}$. 

\begin{proposition}[{\cite[Proposition 2.6]{APS24a}}]\label{pr:Bochner}
 If $\mu \in \meas{\bwt{M}}$ then 
 \[
  \Phi^*(\mu\restrict_{\wt{M}}) = \int_{\wt{M}} m_{xy}\,d\mu(x,y)
 \]
as a Bochner integral of elementary molecules in $\lipfree{M}$. 
\end{proposition}

We will say that $m \in \lipfree{M}$ is a \emph{convex integral of molecules} if $m=\Phi^*\mu$ for some $\mu \in \opr{\bwt{M}}$ that is concentrated on $\wt{M}$ \cite[Definition 2.7]{APS24a}. Because of the way they can be represented, convex integrals of molecules turn out to be well-behaved elements of $\lipfree{M}$; their study is the main focus of \cite{APS24a}. For example, if every element of $\lipfree{M}$ is a convex integral of molecules (for instance when $M$ is uniformly discrete \cite[Corollary 3.7]{APS24a}), then $\ext B_{\lipfree{M}} \subset \Mol$ \cite[Theorem 6.1]{APS24a}. Note that, in general, not every element of $\lipfree{M}$ can be a convex integral of molecules \cite[Theorem 4.1]{APS24a}.

We remark that, by iterating \eqref{eq:norming} (cf.~\cite[Lemma 3.100]{FHHSPZ11}), one can show that, given $m \in \lipfree{M}$ and $\ep>0$, there exist $(x_n,y_n) \in \wt{M}$ and constants $a_n \geq 0$, $n \in \NN$, such that 
\[
 m = \sum_{n=1}^\infty a_n m_{x_ny_n} \quad\text{and}\quad \sum_{n=1}^\infty a_n < \norm{m} + \ep.
\]
Hence $m$ admits the positive representation
\begin{equation}\label{eq:positive_rep}
 \mu :=\sum_{n=1}^\infty a_n\delta_{(x_n,y_n)},
\end{equation}
which is concentrated on $\wt{M}$ but is, in general, not optimal.

It must be stressed that optimal representations of a given element $\psi \in \Lip_0(M)^*$ are far from unique in general, and that ``some optimal representations are more optimal than others''. For example, given $M=[0,1]$ with the usual metric and base point $0$, it seems intuitively clear that $\delta_{(1,0)}$ is a ``good'' representation of $m_{10}$. On the other hand there exists $\mu \in \opr{\bwt{M}}$ that is supported entirely on the remainder $\bwt{M}\setminus\wt{M}$ and satisfies $\Phi^*\mu=m_{10}$ \cite[Example 5]{Aliaga}. Since Stone-{\v C}ech remainders are generally difficult to penetrate, $\mu$ would appear to be a rather convoluted representation of $m_{10}$, especially when a simple Dirac will do; for this reason such a representation can be said to be ``bad''.

Much of this paper is devoted to making this intuition about good and bad representations mathematically precise by defining and exploring a quasi-order on $\meas{\bwt{M}}^+$ that resembles the classical Choquet order on the space of probability measures on a compact Hausdorff space; essentially, optimal representations that are also \emph{minimal} (not maximal -- see Section \ref{subsec:Choquet} below) with respect to this quasi-order are, in our opinion, the most well-behaved and useful representations of free space elements and their cousins in the bidual known to date.

\subsection{Topological and metric preliminaries}\label{subsec:coords}

We make extensive use of results on the topological and metric structure of $\bwt{M}$ and related spaces given in \cite[Sections 2 and 3]{APS24b}, to which we refer the reader for a full account. We shall reproduce the essentials here; full justification of any unsubstantiated claims made below can be found in \cite{APS24b}. A suitable compactification of the metric space $M$ is required, but since (the continuous extensions of) bounded Lipschitz functions don't separate points of $\beta M$ in general, it is often more appropriate to consider the \emph{uniform (or Samuel) compactification} $\ucomp{M}$ of $M$, which was studied systematically in \cite{Woods} and exploited in the current context in e.g.~\cite{AP_measures,APS24b,Weaver2}. We state the defining properties of this compactification.

\begin{proposition}[{\cite[Corollary 2.4]{Woods}}]\label{pr:woodsuniform}
There is a compactification $\ucomp{M}$ of $M$, unique up to homeomorphism, such that
\begin{enumerate}[label={\upshape{(\roman*)}}]
\item\label{pr:woodsuniform_1} every bounded uniformly continuous function $f:M \to \RR$ can be extended uniquely and continuously to $\ucomp{f}:\ucomp{M} \to \RR$, and
\item\label{pr:woodsuniform_2} given non-empty $A,B \subset M$, their closures in $\ucomp{M}$ are disjoint if and only if $d(A,B)>0$.
\end{enumerate}
\end{proposition}
In particular, (the continuous extensions of) bounded Lipschitz functions separate points of $\ucomp{M}$. For clarity, we shall denote by $\ucl{A}$ the closure of $A \subset \ucomp{M}$ in $\ucomp{M}$. 

We set up a ``coordinate system'' on $\bwt{M}$ by letting $p_i:\bwt{M}\to\ucomp{M}$, $i=1,2$, be the continuous extensions of the functions $(x,y) \mapsto x$ and $(x,y) \mapsto y$ on $\wt{M}$, respectively. Then we define the ``coordinate map'' $\pp:\bwt{M}\to\ucomp{M}\times\ucomp{M}$ by setting $p(\zeta)=(\pp_1(\zeta),\pp_2(\zeta))$. Given $\zeta \in \bwt{M}$, we have $\pp(\zeta) \in \wt{M} \subset \ucomp{M}\times\ucomp{M}$ if and only if $\zeta=(x,y) \in \wt{M} \subset \bwt{M}$, but in general $\pp$ is not injective so we must be wary not to treat $p(\zeta)$ as the coordinates of $\zeta$ in the usual sense. Despite this caveat, we make considerable use of these functions in the paper.

We pause here to provide another example of ``good'' versus ``bad'' optimal representations. Before giving it, we establish an equality that holds when $M$ is a discrete metric space.

\begin{proposition}\label{pr:discrete}
The metric space $M$ is discrete if and only if $\wt{M}=\pp^{-1}(M \times M)$ $(\subset \bwt{M})$.
\end{proposition}

\begin{proof}
Suppose that there exists $\zeta \in \pp^{-1}(M\times M)\setminus\wt{M}$. Then $\pp(\zeta)=(x,x)$ for some $x \in M$. Let $(x_i,y_i)$ be a net of points in $\wt{M}$ converging to $\zeta$ in $\bwt{M}$. By continuity of $\pp$ we obtain $(x_i,y_i)=\pp(x_i,y_i) \to \pp(\zeta) = (x,x)$. As $x_i \neq y_i$ for all $i$, we conclude that $x$ must be an accumulation point of $M$, so $M$ is not discrete. Conversely, if $x$ is an accumulation point of $M$ then let $(x_n)$ be a sequence of distinct points of $M$ converging to $x$. Then any cluster point $\zeta \in \bwt{M}$ of the sequence $(x_n,x_{n+1})$ in $\wt{M}$ will belong to $\pp^{-1}(M\times M)\setminus\wt{M}$.
\end{proof}

\begin{example}\label{ex:bad_rep}
 Let $(e_n)$ denote the standard basis of $c_0$, set $x_n=\frac{1}{2}(e_1+e_n)$, $n \geq 2$, and define $M=\set{0,e_1} \cup \set{x_n \,:\, n \geq 2}$ with the metric inherited from the usual norm on $c_0$, and with base point $0$. Note that $M$ is closed in $c_0$ and hence complete. Then $\delta_{(e_1,0)} \in \opr{\bwt{M}}$ represents $m_{e_10}$. On the other hand, define $\mu_n \in \meas{\bwt{M}}^+$, $n \geq 2$, by $\mu_n=\frac{1}{2}(\delta_{(e_1,x_n)}+\delta_{(x_n,0)})$. A simple calculation shows that $\mu_n \in \opr{\bwt{M}}$, with $\Phi^*\mu_n = m_{e_10}$ for all $n \geq 2$. Now let $\mu$ be a $w^*$-cluster point of the $\mu_n$. By $w^*$-$w^*$-continuity of $\Phi^*$ and the fact that it is non-expansive, it follows that $\Phi^*\mu = m_{e_10}$ and $\mu \in \opr{\bwt{M}}$. However, as $M$ is (uniformly) discrete, every point of $\wt{M}$ is isolated in $\bwt{M}$, so every compact set $K \subset \wt{M}$ is finite and open in $\bwt{M}$. This means that $\mathbf{1}_K \in C(\bwt{M})$ and $\duality{\mathbf{1}_K,\mu_n}=\mu_n(K)=0$ for all large enough $n$. Therefore $\mu(K)=\duality{\mathbf{1}_K,\mu}=0$ for all such $K$. By the inner regularity of $\mu$ this yields $\mu(\wt{M})=0$. Finally, as $M$ is discrete, we conclude from Proposition \ref{pr:discrete} that $\mu(\pp^{-1}(M \times M))=\mu(\wt{M})=0$. Again, because of the convoluted way in which $\mu$ represents $m_{e_10}$, it can be said to be a ``bad'' representation.
\end{example}

This example will come up again later in Section \ref{subsec:opt_conc}, where the significance of the set $\pp^{-1}(M \times M)$ will become apparent.

We return to presenting our preliminaries. From time to time, we will also make use of the ``reflection map'' $\rr:\bwt{M}\to\bwt{M}$, defined by first setting $r(x,y)=(y,x)$, $(x,y) \in \wt{M}$, and then extending continuously to $\bwt{M}$. This is an involution on $\bwt{M}$ \cite[p.~4]{Aliaga}.

Concerning extensions of unbounded functions, every Lipschitz function $f:M\to\RR$ can be extended uniquely and continuously to $\ucomp{f}:\ucomp{M}\to[-\infty,\infty]$ \cite[Proposition 2.9]{AP_measures}. When such functions take non-real values we can run into problems, so more often than not we will be using the so-called \emph{Lipschitz realcompactification} $\rcomp{M}$ of $M$ \cite{GarridoMerono} instead of $\ucomp{M}$:
\[
 \rcomp{M} = \set{\xi \in \ucomp{M} \,:\, \ucomp{d}(\xi,0) < \infty},
\]
where $\ucomp{d}(\cdot,0)$ is the continuous extension of $d(\cdot,0)$ to $\ucomp{M}$. \label{d_ext} It is clear that $\rcomp{M}$ is an open subset of $\ucomp{M}$ and thus locally compact. We have $\ucomp{M}=\rcomp{M}$ if and only if $M$ is bounded, and $\rcomp{M}=M$ if and only if $M$ is proper. One of the advantages of using $\rcomp{M}$ is that any Lipschitz function $f:M\to\RR$ can be extended uniquely and continuously to $\rcomp{f}:=\ucomp{f}\restrict_{\rcomp{M}}:\rcomp{M} \to \RR$. Analogously to above, we denote by $\rcl{A}$ the closure of $A \subset \rcomp{M}$ in $\rcomp{M}$.

Another big advantage of $\rcomp{M}$ is that it can be endowed with a canonical lower semicontinuous metric $\bar{d}$ that extends $d$. Let $\tau$ denote the subspace topology on $\rcomp{M}$. The canonical embedding $\delta:M\to\Lip_0(M)^*$ can be extended naturally to $\rcomp{M}$ by setting $\duality{f,\delta(\xi)} = \rcomp{f}(\xi)$, $f \in \Lip_0(M)$, $\xi \in \rcomp{M}$. Then $\delta$ is a $\tau$-$w^*$ homeomorphism onto its image in $\Lip_0(M)^*$, with respect to which it is $w^*$-closed \cite[Proposition 2.3]{APS24b}. We define a $\tau$-lower semicontinuous metric extension $\bar{d}$ of $d$ to $\rcomp{M}$ by setting
\[
 \bar{d}(\xi,\eta) = \norm{\delta(\xi)-\delta(\eta)}_{\Lip_0(M)^*}, \quad \xi,\eta \in \rcomp{M}.
\]
This has the effect of turning $(\rcomp{M},\bar{d})$ into a ``metric bidual'' of $(M,d)$, as is evidenced by the following proposition.

\begin{proposition}[{\cite[Proposition 2.4]{APS24b}}]\label{pr:rcomp_observations}The following statements hold.
\begin{enumerate}[label={\upshape{(\roman*)}}]
\item\label{pr:rcomp_observations_1} The metric $\bar{d}$ is $\tau$-lower semicontinuous and finer than $\tau$;
\item\label{pr:rcomp_observations_2} the $\bar{d}$-closed balls $B_{\bar{d}}(\xi,r):= \set{\eta \in \rcomp{M}\,:\,\bar{d}(\xi,\eta) \leq r}$, $r>0$, are $\tau$-compact;
\item\label{pr:rcomp_observations_3} $\tau$-compact subsets of $\rcomp{M}$ are $\bar{d}$-bounded;
\item\label{pr:rcomp_observations_4} the metric space $(\rcomp{M},\bar{d})$ is complete;
\item\label{pr:rcomp_observations_5} $\rcomp{M}$ is $\tau$-normal;
\item\label{pr:rcomp_observations_6} $M$ is $\bar{d}$-closed in $\rcomp{M}$.
\end{enumerate}
\end{proposition}

The next result lists further important properties of $\bar{d}$.

\begin{proposition}[{\cite[Proposition 2.5]{APS24b}}]\label{pr:bar-d_cont}~
 \begin{enumerate}[label={\upshape{(\roman*)}}]
  \item\label{pr:bar-d_cont_1} Given a $\tau$-closed non-empty set $A \subset \rcomp{M}$, the function $\bar{d}(\cdot,A)$ is $\tau$-lower semicontinuous.
  \item\label{pr:bar-d_cont_2} Given a non-empty set $A \subset M$, $\bar{d}(\cdot,\rcomp{\cl{A}})=\rcomp{d}(\cdot,A)$, and is therefore $\tau$-continuous.
  \item\label{pr:bar-d_cont_3} Given $\xi \in \rcomp{M}$, $\bar{d}(\cdot,\xi)$ is $\tau$-continuous if and only if $\xi \in M$, i.e.~$\bar{d}$ is separately $\tau$-continuous if and only if the fixed coordinate belongs to $M$.
  \item\label{pr:bar-d_cont_4} Given $\tau$-closed non-empty sets $A,B \subset \rcomp{M}$, 
  \[
   \bar{d}(A,B) = \sup\set{d(U \cap M,V \cap M) \,:\, \text{$U \supset A$, $V \supset B$ are $\tau$-open subsets of $\rcomp{M}$}}.
  \]
 \end{enumerate}
\end{proposition}

As a consequence of Proposition \ref{pr:bar-d_cont} \ref{pr:bar-d_cont_3}, we can see that the distance function $\bar{d}(\cdot,M)$, being the infimum of a family of $\tau$-continuous functions, is $\tau$-upper semicontinuous.

We will have need of an extension result that will allow us to approximate certain $\bar{d}$-Lipschitz functions on $\rcomp{M}$ by (the extensions of) $d$-Lipschitz functions on $M$. 

\begin{proposition}[{\cite[Proposition 2.6]{APS24b}}]\label{pr:Mat_app}
Let $A \subset \rcomp{M}$ be $\tau$-closed and let $\psi:A \to \RR$ be $\tau$-continuous, bounded and $1$-$\bar{d}$-Lipschitz. Then there exists a bounded 1-Lipschitz function $f:M \to \RR$ such that $\rcomp{f}\restrict_A = \psi$ and $\inf\psi\leq f\leq\sup\psi$.
\end{proposition}

Hereafter, by default, all topological or measure-theoretic notions relating to $\rcomp{M}$ or $\ucomp{M}$ (e.g. open, closed, Borel subsets, continuity etc.) will taken with respect to $\tau$ or the compact topology on $\ucomp{M}$, respectively; whenever we refer to such notions in terms of $\bar{d}$ we will do so explicitly.

As well as $\bar{d}$, we can extend $d\restrict_{\wt{M}}$ continuously to a function on $\bwt{M}$, again called $d$, which takes values in $[0,\infty]$. It is necessary to consider how these two extensions relate to each other. This can be done via $\Phi^*$ and the coordinate maps. We shall fix the set 
\[
 \bwtf = \pp^{-1}(\rcomp{M}\times\rcomp{M})
\]
of points $\zeta \in \bwt{M}$ having ``real'' coordinates. Note that by continuity of $d$, Proposition \ref{pr:bar-d_cont} \ref{pr:bar-d_cont_3} and the triangle inequality
\begin{equation}\label{eq:R_finite_d}
d(\zeta)\leq \bar{d}(\pp_1(\zeta),0)+\bar{d}(\pp_2(\zeta),0), \quad \zeta \in \bwtf,
\end{equation}
so in particular $\bwtf \subset d^{-1}[0,\infty)$.

We also let $\delta:\bwt{M}\to\meas{\bwt{M}}$ denote the usual Dirac evaluation map. To avoid confusion with the other map denoted by $\delta$, we shall write $\delta_\zeta$ to denote the image of $\zeta \in \bwt{M}$ under this new map (as we did with $\delta_{(x,y)}$, $(x,y) \in \wt{M}$), to distinguish from $\delta(\xi)$, $\xi \in \rcomp{M}$, written above.

\begin{proposition}[{\cite[Proposition 2.8]{APS24b}}]\label{pr:De_Leeuw_relation}
Let $\zeta \in \bwtf$. Then
\begin{enumerate}[label={\upshape{(\roman*)}}]
 \item\label{pr:De_Leeuw_relation_1} $d(\zeta)\Phi^*\delta_\zeta = \delta(\pp_1(\zeta)) - \delta(\pp_2(\zeta))$;
 \item\label{pr:De_Leeuw_relation_2} $d(\zeta)\|\Phi^*\delta_\zeta\| = \bar{d}(\pp(\zeta))$;
 \item\label{pr:De_Leeuw_relation_3} $\bar{d}(\pp(\zeta)) \leq d(\zeta)$, and we have equality if and only if $\|\Phi^*\delta_\zeta\|=1$ or $d(\zeta)=0$.
\end{enumerate}
\end{proposition}

Finally, we define $\opt = \{\zeta \in \bwt{M} \,:\, \norm{\Phi^*\delta_\zeta}=1\}$ to be the set of \emph{optimal points} in $\bwt{M}$ \cite[Section 3.1]{APS24b}. This is a dense $G_\delta$ in $\bwt{M}$ that includes $\wt{M}$. Optimal points are relevant to optimal representations because the latter are always concentrated on $\opt$ \cite[Proposition 3.1]{APS24b}. A consequence of this fact, together with Proposition \ref{pr:De_Leeuw_relation}, is that if $\mu \in \opr{\bwt{M}}$ is concentrated on $\bwtf$ then it is also concentrated on $\opt \cap \bwtf$, on which we have $\bar{d} \circ \pp \equiv d$. Finally, we have ways of locating optimal points in some specific circumstances, as is witnessed by the next result.

\begin{proposition}[{cf.~\cite[Proposition 3.4]{APS24b}}]\label{pr:optimals_everywhere}
If $\xi, \eta \in \rcomp{M}$ are distinct then $\opt \cap \pp^{-1}(\xi,\eta)$ is compact and non-empty.
\end{proposition}

It is also known for example that $\zeta \in \opt$ whenever $\zeta \in \bwtf$ and only one of its coordinates belongs to $M$ \cite[p.~12]{APS24b}.

Finally we mention the \emph{vanishing points}, which are at the other end of the spectrum. These are the points $\zeta \in \bwt{M}$ satisfying $\Phi^*\delta_\zeta = 0$ \cite[Section 3.2]{APS24b}. By Proposition \ref{pr:De_Leeuw_relation} \ref{pr:De_Leeuw_relation_2}, if $d(\zeta)>\bar{d}(\pp(\zeta))=0$ then $\zeta$ is a vanishing point. As is stated at the start of \cite[Section 3.2]{APS24b}, points satisfying this condition can always be found when, for example, $M$ is infinite, bounded and uniformly discrete.

\subsection{Supports, shadows, extended supports and functionals that avoid infinity}\label{subsec:supp}

The idea of the support of a free space element was alluded to in the statement of Theorem \ref{th:AP_extreme} above. We make this idea precise in this section. Given a closed subset $A\subset M$ containing $0$, we can identify $\lipfree{A}$ with a subspace of $\lipfree{M}$, namely $\cl{\lspan}(\delta(A))$. The \emph{support} $\supp(m)$ of a functional $m\in\lipfree{M}$ is the smallest closed set $A \subset M$ such that $m\in\lipfree{A\cup\set{0}}$, the existence of which is non-trivial to establish \cite{AP_rmi,APPP_2020}. Equivalently, the support can be defined as
\[
\supp(m) = \bigcap\set{A\subset M: \text{$A$ is closed and $\duality{f,m}=0$ whenever $f\in\Lip_0(M)$ and $f\restrict_A=0$}} .
\]
Supports of free space elements are related to the supports of the measures in $\bwt{M}$ that represent them. Given $E \subset \bwt{M}$, we define the \emph{shadow} $\pp_s(E)$ of $E$ to be the set $\pp_1(E) \cup \pp_2(E)$. We will pay particular attention to the shadows $\pp_s(\supp(\mu))$ of supports of measures.

\begin{proposition}[{\cite[Lemma 8]{Aliaga}}]
\label{pr:aliaga_lemma_8}
If $m\in\lipfree{M}$, then $\supp(m)\subset\pp_s(\supp(\mu))$ whenever $\mu \in \meas{\bwt{M}}$ and $\Phi^*\mu=m$.
\end{proposition}

Thus the support of $m \in \lipfree{M}$ gives us a ``lower bound'' on $\supp(\mu)$ whenever $\Phi^*\mu=m$ or, put another way, $\pp_s(\supp(\mu))$ is an ``upper bound'' on $\supp(m)$. This fact is used in the proof of \cite[Theorem 1]{Aliaga} and will feature again (implicitly) in the proof of Theorem \ref{th:extreme}. An analogous result, Proposition \ref{pr:support_inclusion} below, can be given for measures that represent elements in the bidual; for this we need to introduce the concepts of functionals that avoid infinity and extended supports.

The concept of functionals that avoid infinity was first given in \cite[Section 3.1]{AP_measures}, using an annular decomposition of $\Lip_0(M)^*$ related to the one introduced by Kalton \cite[Section 4]{Kalton}. Rather than reproduce the original definition, we state an equivalent and more straightforward formulation expressed in terms of De Leeuw representations.

\begin{proposition}[{\cite[Corollary 2.12]{APS24b}}]\label{prop_avoid_infinity}
Let $\psi\in \Lip_0(M)^*$. Then the following are equivalent:
\begin{enumerate}[label={\upshape{(\roman*)}}]
\item $\psi$ avoids infinity;
\item every optimal De Leeuw representation of $\psi$ is concentrated on $\bwtf$;
\item $\psi$ has a De Leeuw representation that is concentrated on $\bwtf$.
\end{enumerate}
\end{proposition}

As every free space element avoids infinity \cite[p.~19]{AP_measures}, we get the following important corollary, foreshadowed in \cite{Aliaga}. Contrast this result with Example \ref{ex:bad_rep}, which shows that not every optimal representation of a free space element is concentrated on $\pp^{-1}(M \times M)$.

\begin{corollary}[{cf.~\cite[Proposition 7]{Aliaga}}]\label{co:free_bwtf}
Let $\mu \in \opr{\bwt{M}}$ satisfy $\Phi^*\mu \in \lipfree{M}$. Then $\mu$ is concentrated on $\bwtf$.
\end{corollary}

Now we turn to the definition of extended support.

\begin{definition}[{\cite[Definition 3.5]{AP_measures}}]\label{defn_extended_support}
Let $\psi\in\Lip_0(M)^*$. The \emph{extended support} $\esupp{\psi}$ is given by
\[
\esupp{\psi} = \bigcap\set{\ucl{A} \subset \ucomp{M}: \text{$A\subset M$ and $\duality{f,\psi}=0$ whenever $f\in\Lip_0(M)$ and $f\restrict_A=0$}}.
\]
\end{definition}
Equivalently, $\zeta\in\esupp{\psi}$ if and only if, for every neighbourhood $U$ of $\zeta$, there exists $f\in\Lip_0(M)$ satisfying $\duality{f,\psi}\neq 0$ and $\supp(f)\subset U\cap M$ \cite[Proposition 3.6]{AP_measures}. Extended supports are compatible with the supports of free space elements in the sense that
\[
 \supp(\psi)=\esupp{\psi}\cap M \quad\text{and}\quad \esupp{\psi}=\ucl{\supp(\psi)}
\]
whenever $\psi\in\lipfree{M}$ \cite[Corollary 3.7]{AP_measures}. Provided $\psi \in \Lip_0(M)^*$ avoids infinity, $\esupp{\psi}$ behaves well, meaning that $\duality{f,\psi}=0$ whenever $U\subset\ucomp{M}$ is a neighbourhood of $\esupp{\psi}$ and $f\in\Lip_0(M)$ vanishes on $U\cap M$ \cite[Theorem 3.13]{AP_measures}.

We finish the section by stating two results concerning extended supports and shadows. First, extended supports lead to a generalisation of Proposition \ref{pr:aliaga_lemma_8}.

\begin{proposition}[{\cite[Proposition 5.3]{APS24b}}]\label{pr:support_inclusion}
If $\mu\in\meas{\bwt{M}}$ then $\esupp{\dual{\Phi}\mu}\subset\pp_s(\supp(\mu))$.
\end{proposition}

Second, if $\psi \in \Lip_0(M)^*$ avoids infinity, then we can find $\mu \in \opr{\bwt{M}}$ that represents $\psi$ and has ``minimal shadow'' in the following sense.

\begin{proposition}[{\cite[Proposition 5.4]{APS24b}}]\label{pr:minimal_support_representation}
If $\psi\in\Lip_0(M)^*$ avoids infinity then there is $\mu\in\opr{\bwt{M}}$ such that $\dual{\Phi}\mu=\psi$ and $\pp_s(\supp(\mu)) \subset \esupp{\psi}\cup\set{0}$.
\end{proposition}

\subsection{Relationship with classical Choquet theory}\label{subsec:Choquet}

The proof of Theorem \ref{th:extreme} rests critically on our ``Choquet theory of Lipschitz-free spaces''. Our source for classical Choquet theory is the comprehensive text \cite{lmns10}. We will state almost none of the classical theory here; instead we will confine ourselves to pointing out the main differences between the classical theory as presented in \cite{lmns10}, and its incarnation in the current work. 

There are two differences of particular note. First, Choquet theory as presented in \cite[Chapter 3]{lmns10} requires a function space $\mathcal{H}$, which is a (not necessarily closed) linear subspace of some $C(K)$-space, to contain the constant functions and separate points of the underlying compact Hausdorff space $K$ \cite[Definition 3.1]{lmns10}. It is not clear to us whether there is a natural choice of function space $\mathcal{H}$ and compact space $K$ that satisfies these properties in the current context. One possible contender is the function space
\[
\Phi\Lip_0(M) \oplus \lspan(\mathbf{1}_{\bwt{M}}) \subset C(\bwt{M}).
\]
However, this space doesn't separate points of $\bwt{M}$. This prompts a different choice of underlying compact space, especially given our efforts to introduce $\ucomp{M}$ and $\rcomp{M}$ above, whose points can be separated by the functions we will be working with. However, a significant advantage of retaining $\bwt{M}$ is that we have a reasonable understanding of how it functions, particularly in light of \cite{APS24a,APS24b}, while the structure of other possibilities, e.g. $(\Phi^* \circ \delta)(\bwt{M}) \subset \Lip_0^*(M)$ (where $\delta:\bwt{M}\to\Lip_0(M)^*$ is the Dirac evaluation map), is less clear to us at the time of writing.

For this reason our choice of compact space remains $\bwt{M}$, and we simply dispense with the need to specify a function space. The price we pay is that points are not separated and, consequently, the quasi-order $\preccurlyeq$ that we define on $\meas{\bwt{M}}^+$ (see Section \ref{sec:quasi-order}), which is our analogue of the Choquet order \cite[Definition 3.57]{lmns10}, is not anti-symmetric. However, as we shall see, the lack of anti-symmetry ultimately doesn't cause any real trouble.

The other clear deviation from standard Choquet theory is that here we consider minimality with respect to a Choquet-type quasi-order, whereas in the standard theory it is maximal measures that receive the attention; see e.g. \cite[Section 3.6]{lmns10}. As far as we are aware, minimal measures in standard Choquet theory feature almost nowhere in the literature (the only exception we know of is \cite{Godefroy19}, where they are used to prove some classical results from several fields in analysis). Minimal measures seem to be more suitable in the current context because they are much more compatible with existing notions in free space theory such as convex integrals of molecules, supports of free space elements and extended supports of elements of the bidual. Roughly speaking, we want measures that have as much content as possible on the ``nice'' parts of $\bwt{M}$ such as $\wt{M}$ or, failing that, $\pp^{-1}(M \times M)$. As we will see, minimal measures serve this purpose, while maximal measures in the classical theory would try their best to have content on ``scary'' parts of the Stone-{\v C}ech remainder, and we would like to avoid this if possible.

\subsection{Plan of the paper}\label{subsec:plan}

In Section \ref{sec:G}, we define a natural function cone $G \subset C(\bwt{M})$ satisfying $G \cap (-G) = \Phi\Lip_0(M)$, and establish its fundamental properties. Then we define an associated quotient space $\wt{M}^G$, which is a Hausdorff compactification of $\wt{M}$, and explore its relationship with the uniform compactification $\uwt{M}$ of $\wt{M}$ and the coordinate map $\pp$. In Section \ref{sec:quasi-order} we introduce the quasi-order $\preccurlyeq$ on $\meas{\bwt{M}}^+$, defined in terms of the cone $G$, and explore its basic properties. This will be our version of the Choquet order in this context. Then $\preccurlyeq$-minimal (or just minimal) measures are introduced and investigated. Finally in Section \ref{sec:quasi-order}, we study the extent to which $\preccurlyeq$ preserves the sets on which measures in $\meas{\bwt{M}}^+$ are concentrated. In Section \ref{sec:min_charac} we state and prove a characterisation of minimal measures which will be key to all of the deeper results about minimality. Later in the section, the characterisation is used to prove results about the shadows and marginals of minimal measures. Finally, in Section \ref{sec:main_theorems}, we bring together several results from the previous sections to prove Theorem \ref{th:opt_conc}, on which the proof of Theorem \ref{th:extreme} crucially rests. 

\section{The function cone \texorpdfstring{$G$}{G} and the associated compact space \texorpdfstring{$\wt{M}^G$}{}}\label{sec:G}

\subsection{Definition and fundamental properties of \texorpdfstring{$G$}{G}}\label{subsec:G_defn}

In this section we define a function cone in $C(\bwt{M})$ and lay out some of its fundamental properties. In Section \ref{sec:quasi-order}, this cone will be used to define a quasi-order on $\meas{\bwt{M}}^+$ analogous to the classical Choquet order.

\begin{definition}
 Let $G$ be the set of all functions $g \in C(\bwt{M})$ satisfying
 \begin{equation}\label{G-function}
  d(x,y)g(x,y) \leq d(x,u)g(x,u) + d(u,y)g(u,y) \quad\text{whenever $x,u,y \in M$ are distinct.}
 \end{equation}
\end{definition}

Straightaway we see that $G$ is a closed convex cone. It includes $\Phi\Lip_0(M)$ because
\begin{align}\label{Phi_equality}
d(x,y)\Phi f(x,y) = f(x)-f(y) &= f(x)-f(u) + f(u)-f(y) \\
&= d(x,u)\Phi f(x,u) + d(u,y)\Phi f(u,y), \quad f \in \Lip_0(M). \nonumber
\end{align}
The cone $G$ is \emph{max-stable}, i.e.~$g \vee g' \in G$ whenever $g,g' \in G$. By the triangle inequality, $\mathbf{1}_{\bwt{M}} \in G$. It can also be seen to be invariant under the action of the reflection map $\rr$ in the sense that $g \circ \rr \in G$ whenever $g \in G$.

In order to generate some more straightforward yet important examples of elements of $G$, we shall appropriate one of the standard approaches taken to replace a metric by an equivalent bounded metric. The function $\frac{1}{d}$ on $\bwt{M}$ should be interpreted in the usual way, taking the values $\infty$ and $0$ on $d^{-1}(0)$ and $d^{-1}(\infty)$, respectively.

\begin{proposition}\label{pr:bounded_metric_trick}
Let $g \in G$ be non-negative and let $b\geq 0$. Then $g':=\min\set{g,\frac{b}{d}} \in G$.
\end{proposition}

\begin{proof}  Let $x,u,y \in M$ be distinct. If $g'(x,u)=b/d(x,u)$ or $g'(u,y)=b/d(u,y)$ then
\[
d(x,y)g'(x,y) \leq b \leq d(x,u)g'(x,u) + d(u,y)g'(u,y); 
\]
otherwise 
\begin{align*}
d(x,y)g'(x,y) \leq d(x,y)g(x,y) &\leq d(x,u)g(x,u)+d(u,y)g(u,y)\\
&= d(x,u)g'(x,u)+d(u,y)g'(u,y). \qedhere
\end{align*}
\end{proof}

\begin{corollary}\label{co:overd}
 Given $a,b\geq 0$, $\min\set{a,\frac{b}{d}}=\min\set{a \mathbf{1}_{\bwt{M}},\frac{b}{d}} \in G$.
\end{corollary}

We spend the rest of the section exploring the properties of $G$ a little more deeply. First we show that elements of $G$ identify with certain types of real-valued functions on $\wt{M}$. Let $g:\wt{M}\to\RR$ be a bounded map satisfying \eqref{G-function}. Define $h:M \times M \to \RR$ by
\[
 h(x,y) = \begin{cases} d(x,y)g(x,y) & \text{if }x \neq y \\ 0 & \text{otherwise.} \end{cases}
\]
It is clear that $|h(x,y)| \leq \norm{g}_\infty d(x,y)$ for all $x,y \in M$.

\begin{proposition}\label{pr:h_triangle}
 Let $M$ have at least three distinct points. Then $h$ satisfies the triangle inequality on $M \times M$, i.e. 
 \begin{equation}\label{tri_ineq}
  h(x,y) \leq h(x,u) + h(u,y) \quad\text{whenever $x,u,y \in M$.}
 \end{equation}
\end{proposition}

\begin{proof}
 The only non-trivial case to check is when $x=y$. Given distinct $x,u \in M$, by the hypothesis we can pick $z \in M\setminus\set{x,u}$. As $g$ satisfies \eqref{G-function}, we have
 \[
  h(x,z) \leq h(x,u) + h(u,z) \quad\text{and}\quad h(u,z) \leq h(u,x) + h(x,z).
 \]
Adding these two inequalities and cancelling yields $h(x,x) = 0 \leq h(x,u)+h(u,x)$.
\end{proof}

As \eqref{G-function} is satisfied vacuously whenever $M$ is a 2-point metric space, $G=C(\bwt{M})=C(\wt{M})$ and $h$ will fail the triangle inequality in this case. For this reason, hereafter and without further comment we will assume that $M$ always has at least three distinct points.

\begin{definition}\label{df:assoc_map}
We shall call $h$ the \emph{associated map}.
\end{definition} 

We endow $M \times M$ with a standard product metric:
\[
 \wt{d}((x,y),(u,v)) = d(x,u) + d(y,v), \quad (x,y),(u,v) \in M \times M.
\]
The metric $\wt{d}$ (and its restriction to $\wt{M}$) will be used frequently throughout this paper. Here is its first application.
 
\begin{lemma}\label{lm:auto_cont}
Let $g:\wt{M}\to\RR$ be a bounded map satisfying \eqref{G-function}. Then the associated map $h$ is $\norm{g}_\infty$-$\wt{d}$-Lipschitz. Hence $g$ is continuous on $\wt{M}$ and therefore extends continuously to a map on $\bwt{M}$, again labelled $g$, that belongs to $G$.
\end{lemma}

\begin{proof}
Given $(x,y),(u,v) \in M \times M$, by \eqref{tri_ineq},
\begin{align*}
h(x,y)-h(u,v) &\leq h(x,u)+h(u,v)+h(v,y) - h(u,v)\\
 &\leq \norm{g}_\infty(d(x,u)+d(y,v)) = \norm{g}_\infty \wt{d}((x,y),(u,v)).
\end{align*}
Thus $h$ is $\norm{g}_\infty$-$\wt{d}$-Lipschitz. The remaining assertions follow straightaway.
\end{proof}

According to this result, elements of $G$ identify with bounded real-valued maps on $\wt{M}$ satisfying \eqref{G-function}. A straightforward consequence of Lemma \ref{lm:auto_cont} is that $\norm{g}_\infty=\Lip_{\wt{d}}(h)$ whenever $g \in G$.

The next result allows us to express elements of $G$ in terms of bounded subsets of $\Lip_0(M)$, via the De Leeuw transform.

\begin{proposition}\label{pr:G-functions} Let $g:\wt{M}\to\RR$. Then $g$ is bounded and satisfies \eqref{G-function} if and only if there exists a non-empty bounded set $A \subset \Lip_0(M)$ such that
\begin{equation}\label{eq:g_sup}
g(x,y)=\sup_{f \in A} \Phi f(x,y), \quad (x,y) \in \wt{M}.
\end{equation}
Moreover, the continuous extension of $g$ to $\bwt{M}$ belongs to $G$, and $A$ can be chosen so that the supremum is always attained.
\end{proposition}

\begin{proof}
 Let $A \subset \Lip_0(M)$ be non-empty and bounded, such that \eqref{eq:g_sup} holds. Clearly $g$ is bounded. Given $f \in A$ and distinct $x,u,y \in M$, 
 \[
  d(x,y)\Phi f(x,y) = d(x,u)\Phi f(x,u) + d(u,y)\Phi f(u,y) \leq d(x,u)g(x,u) + d(u,y)g(u,y),
 \]
by \eqref{Phi_equality}. Therefore $g$ satisfies \eqref{G-function}. 

Conversely, let $g$ be bounded and satisfy \eqref{G-function}. By Lemma \ref{lm:auto_cont}, the associated map $h$ is $\norm{g}_\infty$-$\wt{d}$-Lipschitz.
Given $u \in M$, define $f_u:M \to \RR$ by $f_u(x)=h(x,u)-h(0,u)$. Evidently $f_u\in \norm{g}_\infty B_{\Lip_0(M)}$. Moreover, given $(x,y) \in \wt{M}$, by \eqref{tri_ineq},
\[
 d(x,y)\Phi f_u(x,y) = h(x,u) - h(y,u) \leq h(x,y) = d(x,y)g(x,y),
\]
and
\[
 d(x,y)\Phi f_y(x,y) = h(x,y) = d(x,y)g(x,y).
\]
Therefore if we set $A=\set{f_u \,:\, u \in M} \subset \norm{g}_\infty B_{\Lip_0(M)}$, we obtain \eqref{eq:g_sup} with the supremum attained. Finally, that $g$ can be extended continuously to $\bwt{M}$ follows by Lemma \ref{lm:auto_cont}. 
\end{proof}

When applying Proposition \ref{pr:G-functions} we will use the same symbol $g$ for the map $g:\wt{M}\to\RR$ and its continuous extension to $\bwt{M}$. We remark that, given $g \in G$, the set $\set{f \in \Lip_0(M) \,:\, \Phi f \leq g}$ is $w^*$-compact and convex, and is the largest set satisfying \eqref{eq:g_sup} above.

We end the section by observing that the intersection of $G$ and $-G$ is precisely the image of $\Phi$.

\begin{proposition}\label{pr:G_intersection}
 We have $G \cap (-G) = \Phi \Lip_0(M)$.
\end{proposition}

\begin{proof}
 Let $\pm g \in G$ and let $h$ be the map associated with $g$. Then \eqref{tri_ineq} becomes an equality, and if we set $f(x)=h(x,0)$, $x \in M$, we see quickly that $f \in \Lip_0(M)$ and $\Phi f = g$.
\end{proof}

\subsection{A compactification \texorpdfstring{$\wt{M}^G$ of $\wt{M}$}{} associated with \texorpdfstring{$G$}{G}}\label{subsec:wtMG}

Compactifications of $\wt{M}$, principally $\bwt{M}$ and $\pp(\bwt{M}) \subset \ucomp{M}\times \ucomp{M}$, have been an essential component of our study of the De Leeuw transform to date \cite{APS24a,APS24b}. In order to build a Choquet theory of free spaces, it appears necessary to consider some more compactifications of $\wt{M}$ that are associated with $G$.

As is pointed out in Section \ref{subsec:Choquet}, one of the key principles of Choquet theory is that the function spaces and cones to which the theory is applied must separate points of the underlying compact space \cite[Definition 3.1]{lmns10}. One reason for insisting on this is so that the Choquet order \cite[Definition 3.19]{lmns10} is anti-symmetric. We want to use $G$ to define a quasi-order on $\meas{\bwt{M}}$ in an analogous way. However, the cone $G$ does not separate points of $\bwt{M}$ even in quite simple cases. The next example is based on one given originally by E. Perneck\'a.

\begin{example}[E.~Perneck\'a]\label{ex:G_non-sep}
Let $M=\set{0} \cup \set{n^{-1} \,:\, n \in \NN}$ be endowed with the usual metric and have base point $0$. By compactness, the sequences $(n^{-1},0)_{n=2}^\infty$ and $(n^{-1},n^{-2})_{n=2}^\infty$ in $\wt{M}$ have subnets $(n_i^{-1},0)$ and $(n_i^{-1},n_i^{-2})$ (indexed by the same directed set) that converge to $\zeta,\omega \in \bwt{M}$, respectively; these points are distinct as the closures in $\wt{M}$ of the images of these sequences are disjoint. 

Now let $g \in G \cap B_{C(\bwt{M})}$, and let $h$ be the associated map. By Lemma \ref{lm:auto_cont},
\[
  \abs{h(n^{-1},0) - h(n^{-1},n^{-2})} \leq n^{-2} \quad\text{and}\quad \abs{h(n^{-1},n^{-2})} \leq 2n^{-1},
\]
thus
\begin{align*}
 \lim_n g(n^{-1},0) - g(n^{-1},n^{-2}) &= \lim_n nh(n^{-1},0)-\frac{n}{1-n^{-1}}h(n^{-1},n^{-2})\\
 &= \lim_n n\pare{h(n^{-1},0) - h(n^{-1},n^{-2})} - \frac{1}{1-n^{-1}}h(n^{-1},n^{-2}) = 0.
\end{align*}
Hence by continuity, $g(\zeta)=g(\omega)$. It follows that $G$ doesn't separate points of $\bwt{M}$.
\end{example} 

For this reason we shall define a compactification $\wt{M}^G$ of $\wt{M}$ in such a way that its points are separated by a function cone in $C(\wt{M}^G)$ that corresponds naturally to $G$. Define an equivalence relation $\sim$ on $\bwt{M}$ by $\zeta \sim \omega$ if and only if $g(\zeta) = g(\omega)$ for all $g \in G$, let $\gq$ be the map that sends $\zeta \in \bwt{M}$ to its equivalence class, and endow the set $\wt{M}^G:=\bwt{M}/\!\!\sim$ of equivalence classes with the quotient topology generated by $\gq$. Then $\wt{M}^G$ is compact and, because each $g \in G$ is continuous on $\bwt{M}$, it is easy to show that $\wt{M}^G$ is Hausdorff as well. 

For $\wt{M}^G$ to be a (Hausdorff) compactification of $\wt{M}$, it needs to admit a dense homeomorphic copy of $\wt{M}$. It is obvious that $\gq(\wt{M})$ is dense in $\wt{M}^G$. To show that $\gq\restrict_{\wt{M}}$ is a homeomorphism of $\wt{M}$ onto $\gq(\wt{M})$ involves the following lemma. We expect it is known but as we couldn't find it in the literature, we include a proof.

\begin{lemma}\label{lm:restrict_homeo}
Let $q:X \to Y$ be a continuous, closed and surjective mapping, and suppose that $E \subset X$ has the property that $q^{-1}(q(x))=\set{x}$ for all $x \in E$. Then $q\restrict_E:E \to q(E)$ is a homeomorphism.
\end{lemma}

\begin{proof}
 The condition on $E$ implies $q\restrict_E$ is injective, so it suffices to show that the restriction is open. Let $U \subset E$ be open. Then $U=E \cap V$ for some open $V \subset X$. Set $W = q^{-1}(Y\setminus q(X\setminus V)) \subset V$. As $q$ is closed and surjective, $q(W)=Y\setminus q(X\setminus V)$ is open in $Y$.
 
 We claim that $q(U)=q(E) \cap q(W)$, showing that $q(U)$ is open in $q(E)$ as required. Let $x \in U$. Then $q(x) \notin q(X\setminus V)$, lest $q^{-1}(q(x))\neq\set{x}$, hence $x \in W$, giving $q(x) \in q(E \cap W) \subset q(E) \cap q(W)$. Conversely, let $z \in q(E) \cap q(W)$, so $q(x)=z=q(y)$ for some $x \in E$ and $y \in W$. Then $y \in q^{-1}(q(x))=\set{x}$, giving $z \in q(E \cap W) \subset q(U)$.
\end{proof}

\begin{proposition}\label{pr:wt_homeo}
We have $\gq^{-1}(\gq(x,y))=\set{(x,y)}$ whenever $(x,y) \in \wt{M}$, and thus $\gq\restrict_{\wt{M}}:\wt{M}\to\gq(\wt{M})$ is a homeomorphism.
\end{proposition}

\begin{proof}
Fix $(x,y) \in \wt{M}$ and let $\zeta \in \bwt{M}\setminus\set{(x,y)}$. If $\zeta \in \wt{M}$ or $\zeta$ is a vanishing point (see the end of Section \ref{subsec:coords}), then it is easy to find $f \in \Lip_0(M)$ such that $\Phi f(\zeta) \neq \Phi f(x,y)$. Otherwise $\Phi^* \delta_\zeta \notin \lipfree{M}$ by \cite[Theorem 3.43]{Weaver2}, so in particular $\Phi^* \delta_\zeta \neq m_{xy}$, meaning that there exists $f \in \Lip_0(M)$ such that, again, $\Phi f(\zeta) \neq \Phi f(x,y)$. Whatever the case, there exists $g:=\Phi f \in G$ such that $g(\zeta) \neq g(x,y)$, hence $\gq^{-1}(\gq(x,y))=\set{(x,y)}$ as required. The second assertion of the proposition follows immediately by Lemma \ref{lm:restrict_homeo}.
\end{proof}

We have shown that $\wt{M}^G$ is a Hausdorff compactification of $\wt{M}$. Now define the set
\[
 \widehat{G} = \set{f \in C(\wt{M}^G) \,:\, f \circ \gq \in G}.
\]
Evidently this is a closed cone in $C(\wt{M}^G)$, and the linear isometric embedding $T:C(\wt{M}^G)\to C(\bwt{M})$, given by $Tk=k\circ \gq$, satisfies $T(\widehat{G})=G$ by the universal property of quotient spaces. It follows therefore that $\widehat{G}$ separates points of $\wt{M}^G$.

We know that $T^*:\meas{\bwt{M}}\to\meas{\wt{M}^G}$, given by $T^*\mu=\gq_\sharp\mu$ (the pushforward of $\mu$ with respect to $\gq$), is a natural extension of $\gq$. The fact that $\gq^{-1}(\gq(x,y))=\set{(x,y)}$ whenever $(x,y) \in \wt{M}$, shown above, extends via $T^*$ to $\mu \in \meas{\bwt{M}}$ concentrated on $\wt{M}$.

\begin{proposition}\label{pr:ext_g_pre-image}
Let $\mu \in \meas{\bwt{M}}$ be concentrated on $\wt{M}$. Then $(T^*)^{-1}(T^*\mu) = \set{\mu}$.
\end{proposition}

\begin{proof}
Let $\nu \in \meas{\wt{M}}$ satisfy $\gq_\sharp \nu = T^*\nu = T^*\mu = \gq_\sharp \mu$. Given compact $K \subset \bwt{M}$, we observe 
\[
 \nu(K) \leq \nu(\gq^{-1}(\gq(K))) = \mu(\gq^{-1}(\gq(K))) = \mu(\gq^{-1}(\gq(K)) \cap \wt{M}) = \mu(K \cap \wt{M}) = \mu(K),
\]
(the equality $\gq^{-1}(\gq(K)) \cap \wt{M}=K \cap \wt{M}$ holds as $\gq^{-1}(\gq(x,y))=\set{(x,y)}$ whenever $(x,y) \in \wt{M}$). By inner regularity of $\nu$ applied to $\bwt{M}\setminus \wt{M}$, $\nu$ is concentrated on $\wt{M}$ also; repeating the argument above with $\mu$ and $\nu$ swapped yields $\nu(K)=\mu(K)$ for all compact $K$, whence $\nu=\mu$ again by inner regularity.
\end{proof}

We continue by showing that the function $d$ factors naturally through $\wt{M}^G$. This fact will be of use later on.

\begin{proposition}\label{pr:d_factor}
The map $d$ is constant on every fibre of $\gq$, thus there is a continuous function $d^G:\wt{M}^G\to[0,\infty]$ such that $d=d^G \circ \gq$.
\end{proposition}

\begin{proof}
By Corollary \ref{co:overd}, $\min\set{n,\frac{1}{d}} \in G$ for all $n \in \NN$. Hence $d$ is constant on every fibre of $\gq$. The existence of $d^G$ follows by the universal property of quotient spaces.
\end{proof}

We end this section with a remark and an example. It would be nice if, like $d$, the coordinate map $\pp:\bwt{M}\to\ucomp{M}\times\ucomp{M}$ factored naturally through $\wt{M}^G$. By the universal property of quotient spaces, the existence of such a map $\mathfrak{q}$ is equivalent to the statement that $\pp(\zeta)=\pp(\omega)$ whenever $\gq(\zeta)=\gq(\omega)$. However, this is false. According to the results of the next section (specifically Lemma \ref{lm:G-functions_separate_points} and Proposition \ref{pr:uwt_basic} \ref{pr:uwt_basic_2}), $\pp(\zeta)=\pp(\omega)$ whenever $\gq(\zeta)=\gq(\omega)$ and $\min\set{d(\zeta),d(\omega)}<\infty$, but this can break down if $d(\zeta)=d(\omega)=\infty$, as the next example shows.

\begin{example}\label{ex:sep_at_infinity_fails}
Let $M=\NN$ with the metric inherited from $\RR$ and with base point $1$. Then $M$ is proper so $\rcomp{M}=M=\NN$. By compactness, the sequences $(1,n^2)_{n=2}^\infty$ and $(n,n^2)_{n=2}^\infty$ in $\wt{M}$ have subnets $(1,n_i^2)$ and $(n_i,n_i^2)$ (indexed by the same directed set) that converge to $\zeta,\omega \in \bwt{M}$, respectively. By continuity of $\pp_1$, we have $\pp_1(\zeta)=1$ and $\pp_1(\omega) \notin \NN$, so certainly $\pp(\zeta) \neq \pp(\omega)$. On the other hand, $g(\zeta)=g(\omega)$ for all $g \in G$. Indeed, let $g \in g \cap B_{C(\bwt{M})}$ and let $h$ be the associated map. By Lemma \ref{lm:auto_cont},
\[
 |h(1,n^2)-h(n,n^2)| \leq n \quad\text{and}\quad |h(n,n^2)| \leq 2n^2,
\]
thus
\begin{align*}
 \lim_n g(1,n^2)-g(n,n^2) &= \lim_n \frac{h(1,n^2)}{n^2-1} - \frac{h(n,n^2)}{n^2-n}\\ 
 &= \lim_n \frac{h(1,n^2)-h(n,n^2)}{n^2-1} + \left(\frac{1}{n^2-1} -\frac{1}{n^2-n} \right)h(n,n^2) = 0.
\end{align*}
Hence, by continuity, $g(\zeta)=g(\omega)$.
\end{example}

\subsection{The connection between \texorpdfstring{$\wt{M}^G$ and $\ucomp{\wt{M}}$}{the G- and uniform compactifications of tilde M}}\label{subsec:wtMG_and_uwt}

The compactification $\wt{M}^G$ was defined so that the corresponding function cone $\widehat{G}$ separates its points. However, as $\wt{M}^G$ was constructed in an abstract manner its structure might appear opaque. The point of this section is to reveal some of this structure by showing that $\wt{M}^G$ is closely related to the uniform compactification $\uwt{M}=\ucomp{(\wt{M})}$ of $\wt{M}$.

First note the distinction between $\uwt{M}$ and the locally compact space
\[
\wt{\ucomp{M}}:=\set{(\xi,\eta) \in \ucomp{M} \times \ucomp{M}\,:\, \xi \neq \eta};
\]
we will not consider the latter in this paper at all. In general the metric space $(\wt{M},\wt{d})$ is not complete, yet $\uwt{M}$ is homeomorphic to the uniform compactification of the completion of $\wt{M}$ \cite[Theorem 3.2]{Woods}. Assuming as we do that $M$ is complete, the completion of $\wt{M}$ is isometric to the closure
\[
\cl{\wt{M}} = \set{(x,y) \in M \times M \,:\, x \neq y \text{ or }x=y \text{ is an accumulation point of }M}
\]
of $\wt{M}$ in $M \times M$. This tells us that if $M$ is compact then $\uwt{M}=\ucomp{\cl{\wt{M}}}=\cl{\wt{M}}=\pp(\bwt{M})$ (the final equality follows from \cite[Proposition 2.2]{APS24b}). On the other hand, if $M$ is uniformly discrete then so is $\wt{M}$, meaning that $\uwt{M}=\bwt{M}$.

The next result explores some basic properties of $\uwt{M}$ in the context of $\bwt{M}$ and associated maps. The final statement starts to reveal the connection between $\uwt{M}$ and $\wt{M}^G$. The reader should note that the map $\ucomp{d}$ in part \ref{pr:uwt_basic_5} below is different from $\ucomp{d}(\cdot,0)$ defined on page \pageref{d_ext}.

\begin{proposition}\label{pr:uwt_basic}
 The following properties hold.
 \begin{enumerate}[label={\upshape{(\roman*)}}]
  \item\label{pr:uwt_basic_1} There is a unique continuous surjection $\uu:\bwt{M}\to\uwt{M}$ such that $\uu\restrict_{\wt{M}}$ is the identity.
  \item\label{pr:uwt_basic_2} There is a continuous coordinate map $\vv:\uwt{M}\to \ucomp{M} \times \ucomp{M}$ such that $\pp = \vv \circ \uu$; i.e.~$\pp$ factors continuously through $\uwt{M}$; in particular $\pp$ is constant on all fibres of $\uu$.
  \item\label{pr:uwt_basic_3} Every $\wt{d}$-Lipschitz map $k:\wt{M}\to\RR$ extends continuously to $\ucomp{k}:\uwt{M}\to[-\infty,\infty]$, and the continuous extension $k:\bwt{M}\to[-\infty,\infty]$ of $k$ to $\bwt{M}$ satisfies $k = \ucomp{k} \circ \uu$; in particular $k$ is constant on all fibres of $\uu$. 
\item\label{pr:uwt_basic_4} Given $g \in G$, \ref{pr:uwt_basic_3} applies to the restriction $h\restrict_{\wt{M}}:\wt{M}\to\RR$ of the associated map.
  \item\label{pr:uwt_basic_5} The map $d$ on $\wt{M}$ extends continuously to $\ucomp{d}:\uwt{M}\to [0,\infty]$ and, as a map on $\bwt{M}$, $d=\ucomp{d} \circ \uu$; in particular $d$ is constant on all fibres of $\uu$.
  \item\label{pr:uwt_basic_6} Every $g \in G$ is constant on every fibre of $\uu$ on which $d$ is finite and non-zero. Consequently, every such fibre of $\uu$ is included in a fibre of $\gq$.
 \end{enumerate}
\end{proposition}

\begin{proof}$\;$
 \begin{enumerate}[label={\upshape{(\roman*)}}]
  \item Define $\uu:\wt{M}\to\uwt{M}$ by $\uu(x,y)=(x,y)$. This map extends uniquely and continuously to $\bwt{M}$. As $\bwt{M}$ is compact and $\uu(\wt{M}) = \wt{M}$ is dense in $\uwt{M}$, we have $\uu(\bwt{M})=\uwt{M}$.
  \item Define $\vv_1,\vv_2:\wt{M}\to M$ by $\vv_1(x,y)=x$ and $\vv_2(x,y)=y$. Evidently these are $1$-$\wt{d}$-Lipschitz maps, so by \cite[Theorem 1.1]{Woods}, they extend uniquely and continuously to $\vv_1,\vv_2:\uwt{M}\to \ucomp{M}$, respectively; we define the continuous map $\vv:\uwt{M}\to \ucomp{M} \times \ucomp{M}$ by $\vv(\vartheta)=(\vv_1(\vartheta),\vv_2(\vartheta))$. Since $\pp(x,y)=(x,y)=\vv(\uu(x,y))$ for all $(x,y) \in \wt{M}$, by density and continuity $\pp=\vv \circ \uu$ on $\bwt{M}$.
\item By \cite[Proposition 2.9]{AP_measures}, $k$ extends uniquely and continuously to $\ucomp{k}:\uwt{M}\to[-\infty,\infty]$. Since $k(x,y)=\ucomp{k}(\uu(x,y))$ for all $(x,y) \in \wt{M}$, again by density and continuity, $k = \ucomp{k} \circ \uu$.
\item This is an immediate consequence of Lemma \ref{lm:auto_cont}.
\item This follows from \ref{pr:uwt_basic_3} (or \ref{pr:uwt_basic_4} by considering $\mathbf{1}_{\bwt{M}}\in G$).
\item Let $g \in G$, let $h$ be the associated map, let $k=h\restrict_{\wt{M}}$ and let $\vartheta \in \uwt{M}$ be such that $d$ is finite and non-zero on $\uu^{-1}(\vartheta)$. From \ref{pr:uwt_basic_4} and \ref{pr:uwt_basic_5}, (the continuous extension of) $k$ and $d$ are constant on all fibres of $\uu$. By assumption $d \equiv a \in (0,\infty)$ on $\uu^{-1}(\vartheta)$. It is clear by continuity that $|k| \leq \norm{g}_\infty d$ pointwise on $\bwt{M}$, hence $k \equiv b \in \RR$ on $\uu^{-1}(\vartheta)$. Now it is easy to see by continuity that $g \equiv b/a$ on $\uu^{-1}(\vartheta)$. The second statement follows immediately. \qedhere
\end{enumerate}
\end{proof}

Example \ref{ex:sep_at_infinity_fails} and Proposition \ref{pr:uwt_basic} \ref{pr:uwt_basic_2} tell us that $\wt{M}^G$ and $\uwt{M}$ differ. However, they agree on $d^{-1}(0,\infty)$, in the sense that the converse of Proposition \ref{pr:uwt_basic} \ref{pr:uwt_basic_6} also holds, which implies that the fibres of $\uu$ and $\gq$ that are included in $d^{-1}(0,\infty)$ are equal. In fact we show a little more than this; to do so we require two technical lemmas that will follow after the next one, which is straightforward.

\begin{lemma}\label{lm:separation}
 Let $\zeta,\omega \in \bwt{M}$ satisfy $\uu(\zeta) \neq \uu(\omega)$. There there exist sets $U \ni \zeta$ and $V \ni \omega$ open in $\bwt{M}$, such that $\wt{d}(U \cap \wt{M},V \cap \wt{M})>0$.
\end{lemma}

\begin{proof}
 By regularity of $\uwt{M}$, there exist sets $U_0 \ni \uu(\zeta)$ and $V_0 \ni \uu(\omega)$ that are open in $\uwt{M}$ and have disjoint closures. By the density of $\wt{M}$ in $\uwt{M}$ and Proposition \ref{pr:woodsuniform} \ref{pr:woodsuniform_2}, $\wt{d}(U_0 \cap \wt{M},V_0 \cap \wt{M})>0$. To finish, set $U=\uu^{-1}(U_0)$ and $V=\uu^{-1}(V_0)$, which are open in $\bwt{M}$ by Proposition \ref{pr:uwt_basic} \ref{pr:uwt_basic_1}. 
\end{proof}

\begin{lemma}\label{lm:u_and_p_agree_on_diagonal}
 Let $\zeta,\omega \in \bwt{M}$ satisfy $d(\zeta)=d(\omega)=0$. Then $\uu(\zeta)=\uu(\omega)$ if and only if $\pp(\zeta)=\pp(\omega)$.
\end{lemma}

\begin{proof}
If $\uu(\zeta)=\uu(\omega)$ then $\pp(\zeta)=\pp(\omega)$ by Proposition \ref{pr:uwt_basic} \ref{pr:uwt_basic_2}. Conversely, let $\uu(\zeta) \neq \uu(\omega)$. By Lemma \ref{lm:separation}, there exist sets $U \ni \zeta$ and $V \ni \omega$ open in $\bwt{M}$, such that $r:=\wt{d}(U \cap \wt{M},V \cap \wt{M})>0$. Since $d(\zeta)=0$, by considering $U \cap d^{-1}[0,\frac{r}{3})$ if necessary, we can assume that $d(U) \subset [0,\frac{r}{3})$; likewise we can assume $d(V) \subset [0,\frac{r}{3})$. We claim that $d(\pp_1(U \cap \wt{M}),\pp_1(V \cap \wt{M})) \geq \frac{r}{6}$. Given $x \in \pp_1(U \cap \wt{M})$, $u \in \pp_1(V \cap \wt{M})$, there exist $y,v \in M$ such that $(x,y) \in U \cap \wt{M}$ and $(u,v) \in V \cap \wt{M}$. Then we can estimate
\[
 r \leq \wt{d}((x,y),(u,v)) = d(x,u) + d(y,v) \leq d(x,u) + d(y,x) + d(x,u) + d(u,v) < 2d(x,u) + \tfrac{2r}{3},
\]
giving $d(x,u) > \frac{r}{6}$. This proves the claim. By Proposition \ref{pr:woodsuniform} \ref{pr:woodsuniform_2}, the closures of $\pp_1(U \cap \wt{M})$ and $\pp_1(V \cap \wt{M})$ in $\ucomp{M}$ are disjoint. By continuity of $\pp_1$ and density of $\wt{M}$,
\[
 \pp_1(\zeta) \in \pp_1(U) \subset \pp_1\pare{\cl{U \cap \wt{M}}} \subset \ucl{\pp_1(U \cap \wt{M})},
\]
and likewise $\pp_1(\omega) \in \ucl{\pp_1(V \cap \wt{M})}$, hence $\pp_1(\zeta) \neq \pp_1(\omega)$. This completes the proof.
\end{proof}

Note that $d(\zeta)=0$ implies $\pp_1(\zeta)=\pp_2(\zeta)$. If $\zeta \in \bwtf$ then this is an immediate consequence of Proposition \ref{pr:De_Leeuw_relation} \ref{pr:De_Leeuw_relation_3}, but it holds in general also. Indeed, given $\zeta \in \bwt{M}$ satisfying $\pp_1(\zeta)\neq\pp_2(\zeta)$, by an argument similar to the one in the proof of Lemma \ref{lm:separation}, there exist sets $U_k \ni \pp_k(\zeta)$ open in $\ucomp{M}$, $k=1,2$, such that $r:=d(U_1 \cap M,U_2 \cap M)>0$. Given a net $(x_i,y_i)$ of points in $\wt{M}$ converging to $\zeta$, by continuity of the $\pp_k$ we get $x_i \in U_1 \cap M$ and $y_i \in U_2 \cap M$ for all large enough $i$, rendering $d(\zeta) \geq r > 0$. (The converse of this statement fails, as is witnessed by vanishing points.) Hence if $d(\zeta)=d(\omega)=0$ and $\pp(\zeta) \neq \pp(\omega)$, then $\pp_1(\zeta) = \pp_2(\zeta) \neq \pp_1(\omega)=\pp_2(\omega)$.

\begin{lemma}\label{lm:G-functions_separate_points}
Let $\zeta,\omega$ in $\bwt{M}$ such that $\uu(\zeta) \neq \uu(\omega)$ and $\min\set{d(\zeta),d(\omega)}<\infty$. Then there exists $g \in G$ satisfying $g(\zeta) \neq g(\omega)$. 
\end{lemma}

\begin{proof}
By Proposition \ref{pr:d_factor}, if $d(\zeta) \neq d(\omega)$ then there exists $g \in G$ such that $g(\zeta) \neq g(\omega)$. Hence we can assume hereafter that $d(\zeta)=d(\omega)<\infty$. First, we consider the case $d(\zeta)=d(\omega)=0$. By Lemma \ref{lm:u_and_p_agree_on_diagonal}, $\uu(\zeta) \neq \uu(\omega)$ implies $\pp(\zeta) \neq \pp(\omega)$ and, by the remark immediately after that lemma, we have $\pp_2(\zeta) \neq \pp_2(\omega)$, so we can pick sets $U \ni \pp_2(\zeta)$, $V \ni \pp_2(\omega)$ open in $\ucomp{M}$, such that $s:=d(U \cap M,V \cap M)>0$. Given $u \in U \cap M$, define $h_u,f_u:M\to\RR$ by $h_u(x)=\min\set{d(x,u),s}$ and $f_u(x)=h_u(x)-h_u(0)$, $x \in M$. Then $f_u \in B_{\Lip_0(M)}$. We define $g \in G \cap B_{C(\bwt{M})}$ using Proposition \ref{pr:G-functions} by first setting
\[
g(x,y)=\sup_{u \in U \cap M} \Phi f_u(x,y),\quad (x,y) \in \wt{M},
\]
and then extending continuously to $\bwt{M}$. Define the sets $U'=\pp_2^{-1}(U) \cap d^{-1}[0,s) \ni \zeta$ and $V'=\pp_2^{-1}(V) \ni \omega$, which are open in $\bwt{M}$. Given $(x,y) \in U' \cap \wt{M}$, we have $y \in U \cap M$ and $d(x,y)<s$, so we can estimate
\[
g(x,y) \geq \Phi f_y (x,y) = \frac{h_y(x)-h_y(y)}{d(x,y)} = \frac{h_y(x)}{d(x,y)} = 1.
\]
Hence $g(\zeta)=1$ by continuity. On the other hand, if $(u,v) \in V' \cap \wt{M}$ then $v \in V \cap M$, so $d(x,v) \geq s$ whenever $x \in U \cap M$. Thus for such $x$, $h_x(u)\leq s$ and $h_x(v)=s$, giving
\[
\Phi f_x(u,v) = \frac{h_x(u)-h_x(v)}{d(u,v)} \leq \frac{s-s}{d(u,v)} \leq 0,
\]
giving $g(u,v)\leq 0$. Therefore $g(\omega) \leq 0$ by continuity.

Second, we assume $\alpha:=d(\zeta) = d(\omega) \in (0,\infty)$. Since $\uu(\zeta) \neq \uu(\omega)$, by Lemma \ref{lm:separation}, there exist sets $U \ni \zeta$, $V \ni \omega$ open in $\bwt{M}$ and satisfying $s:=\wt{d}(U \cap \wt{M},V \cap \wt{M}) > 0$. Let us define an element of $G$ satisfying $g(\zeta)\neq g(\omega)$. Given $(u,v) \in \wt{M}$, define $h_{u,v},f_{u,v}:M \to \RR$ by
 \[
  h_{u,v}(x) = \min\set{d(x,v),\max\set{\tfrac{1}{2}\alpha, \alpha-d(x,u)}} \quad\text{and}\quad f_{u,v}(x)=h_{u,v}(x)-h_{u,v}(0), \quad x \in M.
 \]
Evidently $h_{u,v}$ is $1$-Lipschitz, thus $f_{u,v} \in B_{\Lip_0(M)}$. As above, we define $g \in G \cap B_{C(\bwt{M})}$ using Proposition \ref{pr:G-functions} by first setting 
\[
g(x,y)=\sup_{(u,v) \in U \cap\wt{M}} \Phi f_{u,v}(x,y), \quad (x,y) \in \wt{M},
\]
and then extending it continuously to $\bwt{M}$.

Now define $\lambda=\max\set{\frac{1}{2},1-\frac{s}{2\alpha}} < 1$. We claim that $g(\zeta)=1 > \lambda \geq g(\omega)$, from which we obtain the conclusion. First, given $(x,y) \in U$, we estimate
\[
 g(x,y) \geq \Phi f_{x,y}(x,y) = \frac{h_{x,y}(x)-h_{x,y}(y)}{d(x,y)} = \frac{\min\set{d(x,y),\alpha}}{d(x,y)} = \min\set{1,\frac{\alpha}{d(x,y)}}.
\]
Because $\zeta \in U$ and $d(\zeta)=\alpha$, we have $g(\zeta) \geq 1$ by continuity (and equality follows as $\norm{g}_\infty \leq 1$). To demonstrate the other inequality, first suppose that $(x,y) \in \wt{M}$ satisfies $g(x,y) > \frac{\lambda\alpha}{d(x,y)}$. From the definition of $g$ we can pick $(u,v) \in U \cap\wt{M}$ such that
\begin{equation}\label{eqn:G-functions_separate_points_1}
 h_{u,v}(x) - h_{u,v}(y) = f_{u,v}(x)-f_{u,v}(y) > \lambda\alpha.
\end{equation}
Hence $h_{u,v}(y) < h_{u,v}(x)-\lambda\alpha \leq (1-\lambda)\alpha \leq \frac{1}{2}\alpha$, which by the definition of $h_{u,v}$ implies
\begin{equation}\label{eqn:G-functions_separate_points_2}
 d(y,v) = h_{u,v}(y) < (1-\lambda)\alpha.
\end{equation}
From \eqref{eqn:G-functions_separate_points_1} we also obtain $h_{u,v}(x) > \lambda\alpha \geq \frac{1}{2}\alpha$, thus
\[
\max\set{\tfrac{1}{2}\alpha,\alpha-d(x,u)} \geq h_{u,v}(x) > \tfrac{1}{2}\alpha,
\]
which in turn means that
\[
\alpha-d(x,u) = \max\set{\tfrac{1}{2}\alpha,\alpha-d(x,u)} \geq h_{u,v}(x) > \lambda\alpha,
\]
giving $d(x,u) < (1-\lambda)\alpha$. When combined with \eqref{eqn:G-functions_separate_points_2}, we obtain
\[
\wt{d}((x,y),(u,v)) = d(x,u)+d(y,v) < 2(1-\lambda)\alpha \leq s.
\]
Given that $s=\wt{d}(U \cap \wt{M},V \cap \wt{M})$, we conclude that $g(x,y) \leq \frac{\lambda\alpha}{d(x,y)}$ whenever $(x,y) \in V \cap \wt{M}$. By continuity of $g$ and because $\omega \in V$ and $d(\omega)=\alpha$, we deduce that $g(\omega) \leq \lambda$, as claimed. This completes the proof.
\end{proof}

Among other things, the next result gives an almost complete description of the relationship between $\wt{M}^G$ and $\uwt{M}$.

\begin{proposition}\label{co:g_fibres}$\;$
\begin{enumerate}[label={\upshape{(\roman*)}}]
  \item\label{co:g_fibres_1} The maps $\uu$ and $\pp$ are constant on every fibre of $\gq$ on which $d$ is finite.
  \item\label{co:g_fibres_2} The fibres of $\uu$ and $\gq$ agree on $d^{-1}(0,\infty)$.
\end{enumerate}
\end{proposition}

\begin{proof}$\;$
\begin{enumerate}[label={\upshape{(\roman*)}}]
  \item That $\uu$ is constant on every fibre of $\gq$ on which $d$ is finite follows straightaway from Lemma \ref{lm:G-functions_separate_points}. The corresponding assertion about $\pp$ follows from Proposition \ref{pr:uwt_basic} \ref{pr:uwt_basic_2}. 
  \item This follows from \ref{co:g_fibres_1} and Proposition \ref{pr:uwt_basic} \ref{pr:uwt_basic_6}. \qedhere
\end{enumerate}
\end{proof}

Example \ref{ex:sep_at_infinity_fails} and Proposition \ref{pr:uwt_basic} \ref{pr:uwt_basic_2} show that the fibres of $\uu$ and $\gq$ don't always agree on $d^{-1}(\infty)$. The final example of this section shows that the fibres of $\uu$ and $\gq$ don't always agree on $d^{-1}(0)$ either.

\begin{example}\label{ex:u_g_disagree_on_diagonal}
Let $M$ be as in Example \ref{ex:G_non-sep}. Then there exist $\zeta,\omega \in \bwt{M}$ such that $\uu(\zeta)=\uu(\omega)$ and $\gq(\zeta)\neq\gq(\omega)$. Indeed, let $\zeta \in \bwt{M}$ satisfy $\pp_1(\zeta)=\pp_2(\zeta)$. Then $d(\zeta)=0$, else by the compactness of $M$ we would obtain $\zeta=(x,y)$ for some $(x,y) \in \wt{M}$. Note that $\pp(\bwt{M})=M \times M$ in this case, so if $\zeta,\omega \in \pp^{-1}(0,0)$, then $\uu(\zeta)=\uu(\omega)$ by Lemma \ref{lm:u_and_p_agree_on_diagonal}. Now let $\zeta,\omega \in \pp^{-1}(0,0)$ be cluster points of the sequences $((2n-1)^{-1},(2n)^{-1})$ and $((2n)^{-1},(2n+1)^{-1})$, $n \in \NN$, in $\wt{M}$, respectively, and define $f \in \Lip_0(M)$ by $f(0)=0$ and $f((2n)^{-1})=f((2n-1)^{-1})=n^{-1}$, $n \in \NN$. Then $\Phi f((2n-1)^{-1},(2n)^{-1})=0$ for all $n \in \NN$ and $\lim_n \Phi f((2n)^{-1},(2n+1)^{-1})=4$, meaning that $\Phi f(\zeta)=0 \neq 4 = \Phi f(\omega)$ by continuity. In summary, $\uu(\zeta)=\uu(\omega)$ and $\gq(\zeta)\neq\gq(\omega)$.
\end{example}

We remark that the functionals $\Phi^*\delta_\zeta \neq \Phi^*\delta_\omega$, where $\zeta,\omega$ are as above, are examples of \emph{derivations} (see \cite[Section 2.5]{AP_measures} and \cite[Section 7.5]{Weaver2}). As $\Phi\Lip_0(M) \subset G$, the map $\Phi^* \circ \delta:\bwt{M}\to B_{\Lip_0(M)^*}$ factors naturally through $\wt{M}^G$ by the universal property of quotient spaces. On the other hand, Example \ref{ex:u_g_disagree_on_diagonal} shows that $\Phi^* \circ \delta$ doesn't factor through $\uwt{M}$ in this way.

\section{A Choquet-like quasi-order on \texorpdfstring{$\meas{\bwt{M}}^+$}{the Radon measures on the De Leeuw domain}}\label{sec:quasi-order}

\subsection{Definition and basic results}\label{subsec:defn_basic}

Following the concept of the Choquet order, we define a quasi-order $\preccurlyeq$ on $\meas{\bwt{M}}^+$ as follows.

\begin{definition}\label{df:order}
 Define a relation on $\preccurlyeq$ on $\meas{\bwt{M}}^+$ by $\mu \preccurlyeq \nu$ if and only if $\duality{g,\mu} \leq \duality{g,\nu}$ for all $g \in G$.
\end{definition}

It is obvious that $\preccurlyeq$ is reflexive and transitive. However, it is not anti-symmetric in general.

\begin{example}\label{ex:not_anti-sym}
 Let $M$ and $\zeta,\omega \in \bwt{M}$ be from Examples \ref{ex:G_non-sep} or \ref{ex:sep_at_infinity_fails}. Then $\delta_\zeta \preccurlyeq \delta_{\omega}$ and $\delta_{\omega} \preccurlyeq \delta_\zeta$, but $\delta_\zeta \neq \delta_{\omega}$.
\end{example}

It is however ``somewhat anti-symmetric'', as the next result shows.

\begin{proposition}\label{pr:somewhat_anti-sym}
 Let $\mu,\nu \in \meas{\bwt{M}}^+$. Then $\mu \preccurlyeq \nu$ and $\nu \preccurlyeq \mu$ if and only if $\gq_\sharp \mu = \gq_\sharp \nu$.
\end{proposition}

\begin{proof}
One implication is immediate. Conversely, let $\mu \in \meas{\bwt{M}}$ be non-positive in general, and assume that $\duality{g,\mu} = 0$ for all $g \in G$. We need to show that $\gq_\sharp \mu=0$, i.e.~$\duality{k \circ \gq,\mu} = 0$ whenever $k \in C(\wt{M}^G)$. Recall the norm-closed cone $\widehat{G}$ defined after Proposition \ref{pr:wt_homeo}. It was established there that $\widehat{G}$ separates points. It is also easy to see that $\mathbf{1}_{\wt{M}^G} \in \widehat{G}$ and that $\widehat{G}$ is max-stable. Using the lattice version of the Stone-Weierstrass theorem \cite[Proposition A.31]{lmns10}, the norm-closure of $\widehat{G}-\widehat{G}$ equals $C(\wt{M}^G)$ (strictly speaking, \cite[Proposition A.31]{lmns10} applies to min-stable convex cones containing the constant functions, but it can easily be adapted to apply in this case). Because $\duality{k \circ \gq,\mu} = 0$ whenever $k \in \widehat{G}$, we conclude that $\gq_\sharp \mu=0$ as required.
\end{proof}

We will see that this somewhat anti-symmetry will be sufficient for our purposes. We can use Proposition \ref{pr:somewhat_anti-sym} to see how the somewhat anti-symmetry affects pushforwards of measures under $\uu$ or $\pp$. To state the next result we define the open subset $\uwtf = (\ucomp{d})^{-1}[0,\infty)$ of $\uwt{M}$. We have $\vv^{-1}(\rcomp{M} \times \rcomp{M}) \subset \uwtf$. Indeed, let $\vartheta \in \uwt{M}\setminus\uwtf$ and, by Proposition \ref{pr:uwt_basic} \ref{pr:uwt_basic_1}, let $\zeta \in \bwt{M}$ satisfy $\uu(\zeta)=\vartheta$. By Proposition \ref{pr:uwt_basic} \ref{pr:uwt_basic_5} and \ref{pr:uwt_basic_2}, $d(\zeta)=\ucomp{d}(\vartheta)=\infty$ and thus $\vv(\vartheta)=\pp(\zeta) \notin \rcomp{M}\times\rcomp{M}$, respectively. Again by Proposition \ref{pr:uwt_basic} \ref{pr:uwt_basic_2} and \ref{pr:uwt_basic_5}, it follows that $\bwtf \subset \uu^{-1}(\uwtf)=d^{-1}[0,\infty)$, with equality not holding in general.

\begin{corollary}\label{co:somewhat_anti-sym}
Let $\mu,\nu \in \meas{\bwt{M}}^+$ such that $\mu \preccurlyeq \nu$ and $\nu \preccurlyeq \mu$. Then $(\uu_\sharp \mu)\restrict_\uwtf = (\uu_\sharp \nu)\restrict_\uwtf$. In turn this implies $(\pp_\sharp \mu)\restrict_{\rcomp{M} \times \rcomp{M}} = (\pp_\sharp \nu)\restrict_{\rcomp{M} \times \rcomp{M}}$.
\end{corollary}

\begin{proof}
We prove two contrapositive statements. Let $\mu,\nu \in \meas{\bwt{M}}^+$. First we suppose that $(\uu_\sharp \mu)\restrict_\uwtf \neq (\uu_\sharp \nu)\restrict_\uwtf$. As $\uwtf$ is open in $\uwt{M}$, it is locally compact, which implies $\meas{\uwtf} \equiv C_0(\uwtf)^*$. Thus there exists $f \in C_0(\uwtf)$ satisfying $\duality{f,(\uu_\sharp \mu)\restrict_\uwtf} \neq \duality{f,(\uu_\sharp \nu)\restrict_\uwtf}$. We can extend $f$ continuously to $\ucomp{\wt{M}}$ by ensuring that it vanishes on $\ucomp{\wt{M}}\setminus\uwtf = (\ucomp{d})^{-1}(\infty)$. It follows that $f \circ \uu \in C(\bwt{M})$ is constant on all fibres of $\gq$ by Proposition \ref{pr:d_factor} and Proposition \ref{co:g_fibres} \ref{co:g_fibres_1}, whence $f \circ \uu = k \circ \gq$ for some $k \in C(\wt{M}^G)$. We calculate
\[
\duality{k,\gq_\sharp \mu} = \duality{k \circ \gq,\mu} = \duality{f \circ \uu,\mu} = \duality{f,\uu_\sharp \mu} = \duality{f,(\uu_\sharp \mu)\restrict_\uwtf}
\]
and likewise $\duality{k,\gq_\sharp \nu} = \duality{f,(\uu_\sharp \nu)\restrict_\uwtf}$. Therefore either $\mu \not\preccurlyeq \nu$ or vice-versa, by Proposition \ref{pr:somewhat_anti-sym}.

Second, suppose that $(\pp_\sharp \mu)\restrict_{\rcomp{M} \times \rcomp{M}} \neq (\pp_\sharp \nu)\restrict_{\rcomp{M} \times \rcomp{M}}$. We argue similarly. The space $\rcomp{M} \times \rcomp{M}$ is open in $\ucomp{M} \times \ucomp{M}$ so is locally compact, giving $\meas{\rcomp{M} \times \rcomp{M}} \equiv C_0(\rcomp{M} \times \rcomp{M})^*$. Pick $f \in C_0(\rcomp{M} \times \rcomp{M})$ such that
\[
\duality{f,(\pp_\sharp \mu)\restrict_{\rcomp{M} \times \rcomp{M}}} \neq \duality{f,(\pp_\sharp \nu)\restrict_{\rcomp{M} \times \rcomp{M}}},
\]
and extend $f$ continuously to $\ucomp{M} \times \ucomp{M}$ by ensuring that it vanishes on the remainder. By Proposition \ref{pr:uwt_basic} \ref{pr:uwt_basic_2}, $f \circ \vv \in C(\uwt{M})$. Next, as $\vv^{-1}(\rcomp{M}\times\rcomp{M}) \subset \uwtf$, $f \circ \vv$ vanishes on $\uwt{M}\setminus\uwtf$, which implies, again by Proposition \ref{pr:uwt_basic} \ref{pr:uwt_basic_2},
\[
\duality{f \circ \vv,(\uu_\sharp \mu)\restrict_\uwtf} = \duality{f \circ \vv,\uu_\sharp \mu} = \duality{f,\pp_\sharp\mu} = \duality{f,(\pp_\sharp \mu)\restrict_{\rcomp{M} \times \rcomp{M}}}, 
\]
and likewise $\duality{f \circ \vv,(\uu_\sharp \nu)\restrict_\uwtf} = \duality{f,(\pp_\sharp \nu)\restrict_{\rcomp{M} \times \rcomp{M}}}$, giving $(\uu_\sharp \mu)\restrict_\uwtf \neq (\uu_\sharp \nu)\restrict_\uwtf$.
\end{proof}

Recall the discussion about ``good'' and ``bad'' De Leeuw representations in Section \ref{subsec:DeLeeuw}. We pause momentarily to explain in more detail, by way of a simple example, the intuition that motivates the introduction of this quasi-order in the context of these representations.

\begin{example}\label{ex:Choquet_motivation}
Let $M=\set{0,\frac{1}{2},1}$ have the usual metric (with base point $0$), and consider $\delta_{(1,0)}$ and
\[
 \nu:= \tfrac{1}{2}(\delta_{(1,\frac{1}{2})} + \delta_{(\frac{1}{2},0)}).
\]
It is easily seen that $\delta_{(1,0)},\nu \in \opr{\bwt{M}}$, with $\Phi^*\delta_{(1,0)}=\Phi^*\nu=m_{10}$. Moreover, $\delta_{(1,0)} \preccurlyeq \nu$ because $\duality{g,\delta_{(1,0)}} = g(1,0) \leq \frac{1}{2}(g(1,\frac{1}{2}) + g(\frac{1}{2},0)) = \duality{g,\nu}$ for all $g \in G$. Furthermore, $\nu \not\preccurlyeq \delta_{(1,0)}$ because if we set $g=\min\set{2,\frac{1}{d}}$, then $g \in G$ by Corollary \ref{co:overd} and $\duality{g,\delta_{(1,0)}} = 1 < 2 = \duality{g,\nu}$.
\end{example}

In the example above, it seems that $\delta_{(1,0)}$ is a ``better'' choice of representation of $m_{10}$, since the shared coordinate $\frac{1}{2}$ present in the Diracs that make up $\nu$ has been eliminated. More generally, the intuition is that if $\nu,\mu$ are optimal representations of some $\psi \in \Lip_0(M)^*$ and $\mu \preccurlyeq \nu$, then $\mu$ should be a ``better'' representation of $\psi$, and that $\preccurlyeq$-minimal elements (see Section \ref{subsec:min_elements} below) should be ``the best'' representations. The idea of coordinate elimination will be explored further in Section \ref{subsec:eliminating_coordinates}; coordinate elimination will form the basis of the results on mutually singular marginals in Section \ref{subsec:singular_marginals}. The simple example above is generalised considerably in Corollary \ref{co:minimal_points} below.

In the next proposition, we gather some basic properties of $\preccurlyeq$.

\begin{proposition}\label{pr:basic}
 Let $\mu,\nu \in \meas{\bwt{M}}^+$ satisfy $\mu \preccurlyeq \nu$. The following statements hold.
\begin{enumerate}[label={\upshape{(\roman*)}}]
 \item\label{pr:basic_1} $\mu + \mu' \preccurlyeq \nu + \nu'$ and $c\mu \preccurlyeq c\nu$ whenever $\mu',\nu' \in \meas{\bwt{M}}^+$ satisfy $\mu' \preccurlyeq \nu'$ and $c \geq 0$;
  \item\label{pr:basic_2} $\norm{\mu} \leq \norm{\nu}$;
  \item\label{pr:basic_3} $\Phi^*\mu = \Phi^* \nu$;
  \item\label{pr:basic_4} if $\nu$ is optimal then so is $\mu$.
\end{enumerate}
\end{proposition}

\begin{proof}$\;$
\begin{enumerate}[label={\upshape{(\roman*)}}]
\item This is trivial.
\item As $\mathbf{1}_{\bwt{M}} \in G$ and $\mu,\nu \geq 0$, we have $\norm{\mu} = \duality{\mathbf{1}_{\bwt{M}},\mu} \leq \duality{\mathbf{1}_{\bwt{M}},\nu} = \norm{\nu}$.
\item Let $f \in \Lip_0(M)$. As $\pm\Phi f \in G$, we have $\pm\duality{\Phi f,\mu}\leq \pm\duality{\Phi f, \nu}$, giving equality. Hence $\Phi^*\mu=\Phi^*\nu$.
\item By \ref{pr:basic_2} and \ref{pr:basic_3}, $\norm{\Phi^*\nu} = \norm{\Phi^*\mu} \leq \norm{\mu} \leq \norm{\nu} = \norm{\Phi^*\nu}$, so $\mu$ is optimal. \qedhere
\end{enumerate}
\end{proof}

The quasi-order $\preccurlyeq$ behaves well with respect to the $w^*$-topology. We omit the proof of the next easy fact.

\begin{proposition}\label{pr:w_star}
 Let $(\mu_i),(\nu_i)$ be nets in $\meas{\bwt{M}}^+$ that converge to $\mu,\nu$ in the $w^*$-topology, respectively. Suppose moreover that $\mu_i \preccurlyeq \nu_i$ for all $i$. Then $\mu \preccurlyeq \nu$.
\end{proposition}

The next basic result explores how non-negative scalar multiples of a given non-zero $\mu \in \meas{\bwt{M}}^+$ compare with respect to $\preccurlyeq$. The behaviour differs according to whether $\Phi^*\mu=0$ or not.

\begin{proposition}\label{pr:scalar_mult}
 Let $\mu \in \meas{\bwt{M}}^+$ be non-zero and let $a,b \geq 0$. The following statements hold.
 \begin{enumerate}[label={\upshape{(\roman*)}}]
 \item\label{pr:scalar_mult_1} If $a\mu \preccurlyeq b\mu$ and $\Phi^*\mu \neq 0$, then $a=b$;
 \item\label{pr:scalar_mult_2} $\Phi^*\mu=0$ if and only if $0 \preccurlyeq \mu$;
 \item\label{pr:scalar_mult_3} if $\Phi^*\mu=0$ then $a\mu \preccurlyeq b\mu$ if and only if $a \leq b$.
 \end{enumerate}
\end{proposition}

\begin{proof}$\;$
 \begin{enumerate}[label={\upshape{(\roman*)}}]
 \item By Proposition \ref{pr:basic} \ref{pr:basic_3}, $a\Phi^*\mu=b\Phi^*\mu$, whence $a=b$.
 \item If $\Phi^*\mu \neq 0$ then $0 \not\preccurlyeq \mu$ by \ref{pr:scalar_mult_1}. Now let $\Phi^*\mu = 0$. Given $g \in G$, let $A \subset \Lip_0(M)$ be the associated (non-empty) set from Proposition \ref{pr:G-functions}. In particular, there exists $f \in A$, whence $0=\duality{\Phi f,\mu} \leq \duality{g,\mu}$. It follows that $0 \preccurlyeq \mu$.
 \item Let $\Phi^*\mu = 0$. By \ref{pr:scalar_mult_2}, $\duality{g,\mu} \geq 0$ for all $g \in G$. If $a \leq b$ and $g \in G$ then $\duality{g,a\mu}=a\duality{g,\mu}\leq b\duality{g,\mu}=\duality{g,b\mu}$. Hence $a \leq b$ implies $a\mu \preccurlyeq b\mu$. Conversely, let $a\mu \preccurlyeq b\mu$. Then $a\norm{\mu} = a\duality{\mathbf{1}_{\bwt{M}},\mu} \leq b\duality{\mathbf{1}_{\bwt{M}},\mu} = b\norm{\mu}$. As $\mu$ is non-zero, $a \leq b$. \qedhere
 \end{enumerate}
\end{proof}

\subsection{Minimal measures}\label{subsec:min_elements}

Next we introduce the notion of minimality and present some of the basic results associated with it.

\begin{definition}
Let $\mu \in \meas{\bwt{M}}^+$. We say that $\mu$ is $\preccurlyeq$-minimal, or just minimal, if $\nu \in \meas{\bwt{M}}^+$ and $\nu \preccurlyeq \mu$ implies $\mu \preccurlyeq \nu$. 
\end{definition}

As was discussed in Section \ref{subsec:Choquet}, we deviate significantly from the classical Choquet theory because in that theory the emphasis is on maximal measures almost exclusively (see e.g.~\cite[Section 3.6]{lmns10}), while minimal measures are almost never considered. Despite this deviation, there are some parallels. For a start, minimal elements are plentiful.

\begin{proposition}\label{pr:minimal_exist}
Let $\nu \in \meas{\bwt{M}}^+$. There there exists minimal $\mu \in \meas{\bwt{M}}^+$ such that $\mu \preccurlyeq \nu$.
\end{proposition}

\begin{proof}
Let $W=\set{\mu \in \meas{\bwt{M}}^+ \,:\, \mu \preccurlyeq \nu}$. Evidently $W$ is $w^*$-closed and, by Proposition \ref{pr:basic} \ref{pr:basic_2}, we deduce that it is $w^*$-compact. We will say that $A \subset W$ is a chain if $\mu,\mu' \in A$ implies $\mu \preccurlyeq \mu'$ or vice-versa. By Zorn's Lemma, there exists a chain $A$ that is maximal with respect to inclusion. We claim that $A$ is $w^*$-closed; by the maximality of $A$, it is sufficient to show that $\cl{A}^{w^*}$ is also a chain, but this follows straightaway by Proposition \ref{pr:w_star}. Hence $A$ is $w^*$-closed as claimed. 

Now consider the sets $A_{\mu'}:=\set{\lambda \in A \,:\, \lambda \preccurlyeq \mu'}$, $\mu' \in A$. These sets are $w^*$-closed also and, as $A$ is a chain, they satisfy the finite intersection property. Hence by compactness, there exists $\mu \in \bigcap_{\mu' \in A} A_{\mu'}$. It follows that $\mu \preccurlyeq \mu'$ whenever $\mu' \in A$. Now let $\mu' \in \meas{\bwt{M}}^+$ satisfy $\mu' \preccurlyeq \mu$. Then it is easy to see that $A \cup \{\mu'\} \subset \Gamma$ is also a chain, so $\mu' \in A$ again by maximality. Consequently $\mu \preccurlyeq \mu'$ as well.
\end{proof}

Together with Proposition \ref{pr:minimal_exist}, the next proposition can be regarded as an analogue of \cite[Proposition 2.1]{APS24a}.

\begin{proposition}[{cf. \cite[Proposition 2.1]{APS24a}}]\label{pr:min_basic} Let $\mu \in \meas{\bwt{M}}^+$ be minimal. Then
 \begin{enumerate}[label={\upshape{(\roman*)}}]
  \item\label{pr:min_basic_1} $c\mu$ is minimal for every $c \geq 0$;
  \item\label{pr:min_basic_2} if $0 \leq \lambda \leq \mu$ then $\lambda$ is minimal;
  \item\label{pr:min_basic_3} if $E \subset \bwt{M}$ is Borel then $\mu\restrict_E$ is minimal.
 \end{enumerate}
\end{proposition}

\begin{proof}$\;$
 \begin{enumerate}[label={\upshape{(\roman*)}}]
  \item This is trivial if $c>0$. The fact that $0$ is minimal follows for example from Proposition \ref{pr:basic} \ref{pr:basic_2}.
  \item Let $\nu \in \meas{\bwt{M}}^+$ satisfy $\nu \preccurlyeq \lambda$. By Proposition \ref{pr:basic} \ref{pr:basic_1}, $\nu + (\mu-\lambda) \preccurlyeq \mu$, giving $\mu \preccurlyeq \nu + (\mu-\lambda)$ by minimality of $\mu$ and thus $\lambda \preccurlyeq \nu$. Hence $\lambda$ is minimal.
  \item This follows trivially from \ref{pr:min_basic_2}. \qedhere
  \end{enumerate}
\end{proof}

We turn now to the task of finding concrete examples of minimal measures. A couple of lemmas are needed first. To begin, we show that fibres of $\uu$ agree with those of $\pp$ whenever we have at least one coordinate in $M$. 

\begin{lemma}\label{lm:mixed_coords}
Let $\zeta,\omega \in \bwt{M}$ with $\pp(\zeta)=\pp(\omega)$ and either $\pp_1(\zeta) \in M$ or $\pp_2(\zeta) \in M$. Then $\uu(\zeta)=\uu(\omega)$.
\end{lemma}

\begin{proof}
 Without loss of generality, we show that if $\pp_2(\zeta)=\pp_2(\omega)=q \in M$ and $\uu(\zeta) \neq \uu(\omega)$, then $\pp_1(\zeta) \neq \pp_1(\omega)$. As $\uu(\zeta) \neq \uu(\omega)$, by Lemma \ref{lm:separation} there exist sets $U' \ni \zeta$, $V' \ni \omega$ open in $\bwt{M}$ such that $r:=\wt{d}(U' \cap \wt{M},V' \cap \wt{M})>0$. Let $W \ni q$ open in $\ucomp{M}$ such that $W \cap M = \set{p \in M \,:\, d(p,q) < \frac{1}{3}r}$, and set $U=U' \cap \pp_2^{-1}(W) \ni \zeta$ and $V=V' \cap \pp_2^{-1}(W) \ni \omega$. We claim that 
 \[
  d(\pp_1(U \cap \wt{M}),\pp_1(V \cap \wt{M})) \geq \tfrac{1}{3}r.
 \]
Indeed, given $x \in \pp_1(U \cap \wt{M})$ and $u \in \pp_1(V \cap \wt{M})$, there exist $y,v \in M$ such that $(x,y) \in U \cap \wt{M}$ and $(u,v) \in V \cap \wt{M}$, and we can estimate
\[
 d(x,u) = \wt{d}((x,y),(u,v)) - d(y,v) \geq r - d(y,q) - d(q,v) > \tfrac{1}{3}r,
\]
which establishes the claim. Hence the closures $\ucl{\pp_1(U \cap \wt{M})}$ and $\ucl{\pp_1(V \cap \wt{M})}$ are disjoint in $\ucomp{M}$. Using the same argument as the one at the end of the proof of Lemma \ref{lm:u_and_p_agree_on_diagonal}, we conclude that $\pp_1(\zeta) \neq \pp_1(\omega)$.
\end{proof}

The above will allow us to partially extend in the next result the supremum attainment witnessed in \eqref{eq:g_sup} to points $\zeta \in \bwtf$ possessing at least one coordinate in $M$. In turn, Proposition \ref{pr:least_measures} will yield our promised examples.

\begin{lemma}\label{lm:delta_minimal}
Let $\zeta \in \bwtf$ satisfy $\pp_1(\zeta) \neq \pp_2(\zeta)=y$, where $y \in M$. Given $g \in G$ and associated map $h$, define $f \in \Lip_0(M)$ by $f(x)=h(x,y)-h(0,y)$, $x \in M$. Then $\Phi f \leq g$ and $g(\zeta) = \Phi f(\zeta)$.
\end{lemma}

\begin{proof}
  Following the proof of Proposition \ref{pr:G-functions}, we obtain $\Phi f \leq g$ and $\Phi f(x,y)=g(x,y)$ for all $x \in M\setminus\set{y}$. To show that $g(\zeta)=\Phi f(\zeta)$, first we consider the special case that $\zeta$ is the limit of a net of points in $\wt{M}$ of the form $(x_i,y)$. In this case, by continuity of both $g$ and $\Phi f$, 
 \[
  g(\zeta) = \lim_i g(x_i,y) = \lim_i \Phi f(x_i,y) = \Phi f(\zeta).
 \]
 
Now we consider the general case. Let $(x_i,y_i)$ be a net of points in $\wt{M}$ converging to $\zeta$. By continuity of $\pp_2$, we have $y_i \to y$ and thus $d(y_i,y) \to 0$. Now consider the points $(x_i,y) \in M \times M$. Because $\pp_1(\zeta) \neq y$, we have $x_i \neq y$ for all $i$ large enough, giving $(x_i,y) \in \wt{M}$. Hence by compactness and by taking a subnet if we need to, we assume that $(x_i,y) \in \wt{M}$ for all $i$ and converges to some $\omega \in \bwt{M}$. From the first case, $g(\omega)=\Phi f(\omega)$. The proof will be finished if we can show that $\gq(\zeta)=\gq(\omega)$, for then $g(\zeta)=g(\omega)=\Phi f(\omega)=\Phi f(\zeta)$. By Proposition \ref{co:g_fibres} \ref{co:g_fibres_1}, this will follow if $\uu(\zeta)=\uu(\omega)$ and $\zeta,\omega \in d^{-1}(0,\infty)$. We know that $\uu(\zeta)=\uu(\omega)$ by Lemma \ref{lm:mixed_coords}, and $\zeta \in d^{-1}(0,\infty)$ by \eqref{eq:R_finite_d} and Proposition \ref{pr:De_Leeuw_relation} \ref{pr:De_Leeuw_relation_3}; likewise for $\omega$.
\end{proof}

\begin{proposition}\label{pr:least_measures}
 Let $\mu \in \meas{\bwt{M}}^+$ be concentrated on $A := \pp^{-1}(\rcomp{M}\setminus\{y\} \times \{y\})$ for some $y \in M$. Then $\mu \preccurlyeq \nu$ whenever $\nu \in \meas{\bwt{M}}^+$ and $\Phi^*\nu=\Phi^*\mu$. In particular, $\mu$ is minimal. The same holds if $\mu$ is concentrated on $\pp^{-1}(\{x\}\times\rcomp{M}\setminus\{x\})$ for some $x \in M$.
\end{proposition}

\begin{proof}
 Let $\nu \in \meas{\bwt{M}}^+$ satisfy $\Phi^*\nu=\Phi^*\mu$. Given $g \in G$, let $f \in \Lip_0(M)$ be as in Lemma \ref{lm:delta_minimal}. Then
 \[
  \duality{g,\mu} = \int_A g \,d\mu = \int_A \Phi f \,d\mu = \duality{\Phi f,\mu} = \duality{\Phi f,\nu} \leq \duality{g,\nu}.
 \]
 Hence $\mu \preccurlyeq \nu$. Now suppose $\nu \in \meas{\bwt{M}}^+$ and $\nu \preccurlyeq \mu$. Then $\Phi^*\nu=\Phi^*\mu$ by Proposition \ref{pr:basic} \ref{pr:basic_3}, so from above $\mu \preccurlyeq \nu$. It follows that $\mu$ is minimal. The last statement follows by considering $\rr_\sharp \mu$ and the invariance of $G$ under $\rr$.
\end{proof}

One observes that any measure $\mu$ satisfying the hypotheses of Proposition \ref{pr:least_measures} is unique, modulo the lack of anti-symmetry of $\preccurlyeq$, in the sense that if $\nu \in \meas{\bwt{M}}^+$ is minimal and $\Phi^*\nu=\Phi^*\mu$, then $\gq_\sharp\nu = \gq_\sharp\mu$.

\begin{corollary}\label{co:minimal_points}
The following statements hold.
  \begin{enumerate}[label={\upshape{(\roman*)}}]
  \item\label{co:minimal_points_1} If $\zeta \in \bwtf$ has distinct coordinates with at least one coordinate in $M$, then $\delta_\zeta$ is minimal and $\gq_\sharp\mu = \gq_\sharp\delta_\zeta$ whenever $\mu \in \meas{\bwt{M}}^+$ is minimal and $\Phi^*\mu = \Phi^*\delta_\zeta$.  
  \item\label{co:minimal_points_2} Given $(x,y) \in \wt{M}$, $\delta_{(x,y)}$ is the unique positive minimal representation of $m_{xy}$.
 \end{enumerate}
\end{corollary}

\begin{proof} Given the above observation only \ref{co:minimal_points_2} requires proof, and this follows immediately from \ref{co:minimal_points_1} and Proposition \ref{pr:ext_g_pre-image}. \qedhere
\end{proof}

\begin{remark}As we noted at the end of Section \ref{subsec:coords}, it is known that $\zeta \in \opt$ whenever $\zeta \in \bwtf$ and exactly one of the coordinates of $\zeta$ belongs to $M$. This follows also by applying Lemma \ref{lm:delta_minimal} to $\mathbf{1}_{\bwt{M}} \in G$ because then $\norm{\Phi^*\delta_\zeta} \geq \Phi f(\zeta)=1$. 
\end{remark}

Next, we show that representations of free space elements that are both optimal and minimal need not be unique. In doing so, we will see that if $\zeta \in \bwtf$ has neither coordinate in $M$, then $\delta_\zeta$ may be optimal and minimal but can fail spectacularly to be unique in the sense of Corollary \ref{co:minimal_points} \ref{co:minimal_points_1} (see Example \ref{ex:minimal_delta_not_unique}).

To see this, first we introduce a modulus
\[
 \gamma(M) = \inf\set{\frac{d(x,u)+d(u,y)}{d(x,y)} \,:\, x,u,y \in M \text{ are distinct}} \geq 1.
\]
We are interested in those spaces $M$ for which $\gamma(M)>1$. This condition is a form of ``uniform concavity'' strictly stronger than Weaver's uniform concavity \cite[Definition 3.33]{Weaver2}, and it is straightforward to check that it is satisfied only if (but not always if) $M$ is uniformly discrete. The main reason to introduce it is that it forces $G$ to be very large.

\begin{lemma}\label{lm:large_jack}
Let $\gamma=\gamma(M)>1$. Given $f \in C(\bwt{M})$, set $a=\big(\frac{\gamma+1}{\gamma-1}\big)\norm{f}_\infty$; then $f + a\mathbf{1}_{\bwt{M}} \in G$. 
\end{lemma}

\begin{proof}
 Let $x,u,y \in M$ be distinct. We estimate
 \begin{align*}
  d(x,y)(f(x,y)+a) \leq d(x,y)(\norm{f}_\infty+a) &= d(x,y)\gamma(a - \norm{f}_\infty)\\
  &\leq (d(x,u)+d(u,y))(a - \norm{f}_\infty)\\
  &\leq d(x,u)(f(x,u)+a) + d(u,y)(f(u,y)+a). \qedhere
 \end{align*}
\end{proof}

\begin{corollary}\label{co:super_concave}
 Let $\gamma(M)>1$. Then every element of $\opr{\bwt{M}}$ is minimal.
\end{corollary}

\begin{proof}
Let $\nu \in \opr{\bwt{M}}$ and suppose that $\mu\in\meas{\bwt{M}}^+$ satisfies $\mu \preccurlyeq \nu$. We claim that $\mu=\nu$, which shows that $\nu$ is minimal. By Proposition \ref{pr:basic} \ref{pr:basic_3} and \ref{pr:basic_4}, $\mu$ is an optimal representation of the same element of $\Lip_0(M)^*$; in particular, $\norm{\mu}=\norm{\nu}$. Let $f \in C(\bwt{M})$. By Lemma \ref{lm:large_jack}, $f + a\mathbf{1}_{\bwt{M}} \in G$ for some $a \geq 0$. Hence
\[
 \duality{f,\mu} = \langle f + a\mathbf{1}_{\bwt{M}},\mu\rangle - a\norm{\mu} \leq \langle f + a\mathbf{1}_{\bwt{M}},\nu\rangle - a\norm{\nu} = \duality{f,\nu}.
\]
Since this holds for all $f \in C(\bwt{M})$, we conclude that $\mu=\nu$.
\end{proof}

\begin{example}\label{ex:min_opt_non-unique}
Let $M=\set{0,1,2,3}$ have the discrete metric (and base point $0$) and define 
\[
 \nu_1 = \delta_{(0,2)} + \delta_{(1,3)} \qquad\text{and}\qquad \nu_2 = \delta_{(0,3)} + \delta_{(1,2)}.
\]
It is readily checked that $m:=\Phi^*\nu_1=\Phi^*\nu_2$ and (by considering $f \in B_{\Lip_0(M)}$ defined by $f(x)=0$ if $x=0$ or $x=1$, and $f(x)=-1$ otherwise) $\norm{m}=2$, so $\nu_1$ and $\nu_2$ are optimal. As $\gamma(M)=2$, by Corollary \ref{co:super_concave}, $\nu_1$ and $\nu_2$ are also minimal, as are all convex combinations thereof. Hence optimal and minimal elements need not be unique.
\end{example}

Another obvious consequence of Lemma \ref{lm:large_jack} is that if $\gamma(M)>1$ then $G$ separates points of $\bwt{M}$, meaning that $\gq$ is injective, $\wt{M}^G = \bwt{M}$ and $\preccurlyeq$ becomes anti-symmetric. This leads to the next example.

\begin{example}\label{ex:minimal_delta_not_unique}
Let $M=\NN$ have the discrete metric (and base point $1$), pick distinct $\xi,\eta \in \rcomp{M}\setminus M = \beta M\setminus M$ and let $\zeta \in \pp^{-1}(\xi,\eta)$. As $d(\zeta)=\bar{d}(\xi,\eta)=1$, $\zeta \in \opt$ by Proposition \ref{pr:De_Leeuw_relation} \ref{pr:De_Leeuw_relation_3}. As $\gamma(M)=2$, $\delta_\zeta$ is minimal by Corollary \ref{co:super_concave}. If $\omega \in \pp^{-1}(\xi,\eta)$ then $\Phi^*\delta_\omega = \Phi^*\delta_\zeta$ by Proposition \ref{pr:De_Leeuw_relation} \ref{pr:De_Leeuw_relation_1}. Given that $\preccurlyeq$ is anti-symmetric in this case, the conclusion of Corollary \ref{co:minimal_points} \ref{co:minimal_points_1} fails because $\pp^{-1}(\xi,\eta)$ is never a singleton. We employ a topological argument to show this (though it can also be demonstrated using tensor products of ultrafilters). Partition $\wt{M}$ into $\Gamma_1=\{(m,n) \in \wt{M} \,:\, m < n\}$ and $\Gamma_2=\{(m,n) \in \wt{M} \,:\, m > n\}$. If $U \ni \xi$ and $V \ni \eta$ are open in $\rcomp{M}$ then $U \cap M$ and $V \cap M$ are infinite (else $\xi \in M$ or $\eta \in M$), meaning that $(U \times V) \cap \Gamma_i \neq \varnothing$ for $i=1,2$. By compactness, there exist distinct
\[
 \zeta_i \in \bigcap_{U,V} \cl{((U \times V) \cap \Gamma_i)}^{\bwt{M}}, \quad i=1,2,
\]
which, by continuity of $\pp$, belong to $\pp^{-1}(\xi,\eta)$.

This non-uniqueness can be taken further by considering the points $\xi,\eta$ obtained in \cite[Example 3.11]{APS24b}. While the underlying set in that example is also $M=\NN$, the metric is different; however, we can replace it by the discrete metric without affecting the construction of $\xi$ or $\eta$. In this case $\pp^{-1}(\xi,\eta)$ is infinite, so in fact it has cardinality $2^{\mathfrak{c}}$ \cite[Theorem 3.6.14]{Engelking}. This yields a vast number of non-unique optimal and minimal Dirac measure representations of the same element of $\Lip_0(M)^*$.
\end{example}

Finally, it is natural to reflect on the relationship between optimality and minimality, and specifically whether one implies the other. Since some more tools are needed in order to address this question in a relatively succinct manner, we will do so later in Section \ref{subsec:optimality_minimality}.

\subsection{The quasi-order and sets on which measures are concentrated}\label{subsec:conc} Define the set $\mathfrak{C}$ of Borel subsets $C \subset \bwt{M}$ having the property that $\mu \in \meas{\bwt{M}}^+$ is concentrated on $C$ whenever $\nu \in \meas{\bwt{M}}^+$ is concentrated on $C$ and $\mu \preccurlyeq \nu$. Observe that $\mathfrak{C}$ is closed under taking intersections of countably many elements. This set is of interest to us because it will yield information about the sets on which minimal measures are concentrated.

We spend this section identifying some important elements of $\mathfrak{C}$. To this end, we define the set of extended Borel functions
\[
 \mathfrak{G} = \set{g:\bwt{M}\to[0,\infty] \,:\, g\text{ is Borel and }\int g\,d\mu \leq \int g\,d\nu \text{ whenever }\mu,\nu \in \meas{\bwt{M}}^+, \mu \preccurlyeq \nu}.
\]
This set can be regarded as a cone, though its elements can take infinite values so it can't be included in a vector space of functions in the usual sense. For the most part, the functions in $\mathfrak{G}$ that we consider will be bounded and thus can be regarded as elements of $C(\bwt{M})^{**}$ via integration. Obviously all positive elements of $G$ belong to $\mathfrak{G}$. By the monotone convergence theorem, $\mathfrak{G}$ is closed under taking pointwise limits of increasing sequences, and pointwise limits of bounded decreasing sequences.

The following result will be our principal tool for identifying elements of $\mathfrak{C}$; this is the main motivation for introducing $\mathfrak{G}$, though this set is used also in Proposition \ref{pr:imp_4.20}.

\begin{proposition}\label{pr:zero_set}
Let $C \subset \bwt{M}$ be Borel, such that given any $K_\sigma$ set $F \subset C$, there exists $g \in \mathfrak{G}$ satisfying $F \subset g^{-1}(0) \subset C$. Then $C \in \mathfrak{C}$.
\end{proposition}

\begin{proof}
 Let $\mu,\nu \in \meas{\bwt{M}}^+$, $\mu \preccurlyeq \nu$, and suppose $\nu$ is concentrated on $C$. By inner regularity, $\nu$ is concentrated on some $K_\sigma$ set $F \subset C$. Let $g \in \mathfrak{G}$ such that $F \subset g^{-1}(0) \subset C$. Then 
 \[
  0 \leq \int_{\bwt{M}\setminus C} g \,d\mu \leq \int_{\bwt{M}\setminus g^{-1}(0)} g \,d\mu = \int g \,d\mu \leq \int g \,d\nu = \int_{g^{-1}(0)} g \,d\nu = 0.
 \]
As $g$ is strictly positive on $\bwt{M}\setminus C$, this implies $\mu(\bwt{M}\setminus C)=0$. Hence $\mu$ is concentrated on $C$.
\end{proof}

The next result, which gives us our first non-trivial element of $\mathfrak{C}$, will be generalised in Proposition \ref{pr:delta} below.

\begin{proposition}\label{pr:diag}
We have $\mathbf{1}_{d^{-1}(0)} \in \mathfrak{G}$ and hence $d^{-1}(0,\infty] \in \mathfrak{C}$.
\end{proposition}

\begin{proof}
By Corollary \ref{co:overd}, $g_n:=\min\set{1,\frac{1}{nd}} \in G$ for all $n \in \NN$. If $d(\zeta)=0$ then $g_n(\zeta)=1$ for all $n \in \NN$, and if $d(\zeta)>0$ then $g_n(\zeta)$ decreases to $0$. Hence $\mathbf{1}_{d^{-1}(0)} \in \mathfrak{G}$, being the pointwise limit of the bounded decreasing sequence $(g_n)$. We apply Proposition \ref{pr:zero_set} to finish.
\end{proof}

We use the next key lemma to find more complicated examples of elements of $\mathfrak{C}$. Let $E^\circ$ denote the interior of a set $E \subset \rcomp{M}$.

\begin{lemma}\label{lm:G_distance_functions}
Let $f:M\to[0,\infty)$ be Lipschitz and let $a \geq \Lip(f)$. Then there exists $g \in G \cap aB_{C(\bwt{M})}$ such that
\begin{equation}\label{eqn:G_distance_functions_1}
  g(x,y) = \min\set{a,\frac{f(x)}{d(x,y)}}, \quad (x,y) \in \wt{M}.
 \end{equation}
Therefore
\begin{equation}\label{eqn:G_distance_functions_2}
 g(\zeta) = \begin{cases}
             \displaystyle \min\set{a,\frac{\rcomp{f}(\pp_1(\zeta))}{d(\zeta)}} & \text{if }\pp_1(\zeta) \in \rcomp{M}\text{ and either }\rcomp{f}(\pp_1(\zeta))>0\text{ or }d(\zeta)>0,\\
             0 & \text{if }\pp_1(\zeta) \in (\rcomp{f})^{-1}(0)^\circ.
        \end{cases}
\end{equation}
Statements \eqref{eqn:G_distance_functions_1} and \eqref{eqn:G_distance_functions_2} also hold when $f(x)$ and $\pp_1$ are replaced by $f(y)$ and $\pp_2$, respectively.
\end{lemma}

\begin{proof}
Define the bounded map $g:\wt{M}\to\RR$ using \eqref{eqn:G_distance_functions_1}. We verify that $g$ satisfies \eqref{G-function} and thus its continuous extension to $\bwt{M}$, also denoted $g$, belongs to $G$ by Lemma \ref{lm:auto_cont}. Let $x,u,y \in M$ be distinct. There are three cases to check. First suppose $ad(x,u) \leq f(x)$ and $ad(u,y) \leq f(u)$. Then
\[
d(x,y)g(x,y) \leq ad(x,y) \leq a(d(x,u)+d(u,y)) = d(x,u)g(x,u)+d(u,y)g(u,y).
\]
Second, if $f(x) < ad(x,u)$, then as $f$ is positive,
\[
d(x,y)g(x,y) \leq f(x) = d(x,u)g(x,u) \leq d(x,u)g(x,u)+d(u,y)g(u,y).
\] 
The final case is $ad(x,u) \leq f(x)$ and $f(u) < ad(u,y)$; given that $a \geq \Lip(f)$,
\[
d(x,y)g(x,y) \leq f(x) \leq ad(x,u) + f(u) = d(x,u)g(x,u) + d(u,y)g(u,y).
\]
Therefore $g$ satisfies \eqref{G-function}.

Now we establish \eqref{eqn:G_distance_functions_2}. Define $U=\pp_1^{-1}((\rcomp{f})^{-1}(0,\infty)) \cup \pp_1^{-1}(\rcomp{M})\setminus d^{-1}(0)$, which is open in $\bwt{M}$. The top line of the right hand side of \eqref{eqn:G_distance_functions_2} defines a continuous function on $U$. Using \eqref{eqn:G_distance_functions_1}, the density of $\wt{M}$ and continuity, we conclude that \eqref{eqn:G_distance_functions_2} holds whenever $\zeta \in U$. Now let $\pp_1(\zeta) \in (\rcomp{f})^{-1}(0)^\circ$ and let $(x_i,y_i)$ be a net in $\wt{M}$ converging to $\zeta$. As $\pp_1$ is continuous and $(\rcomp{f})^{-1}(0)^\circ$ is open, $f(x_i)=0$ for all $i$ large enough. This implies $g(x_i,y_i)=0$ for all such $i$ and, by continuity, $g(\zeta)=0$. This establishes \eqref{eqn:G_distance_functions_2}. The last statement about replacing $f(x)$ and $\pp_1$ follows by considering $g \circ \rr \in G$.
\end{proof} 

This lemma has a number of applications. 

\begin{corollary}\label{co:bwtf_preserved}
There exists a bounded function $g \in \mathfrak{G}$ such that $\bwtf=g^{-1}(0)$, and thus $\bwtf \in \mathfrak{C}$. 
\end{corollary}

\begin{proof}
Given $k \in \NN$, define the compact set $B_k=\set{\xi \in \rcomp{M} \,:\, \bar{d}(0,\xi) \leq k}$. According to Lemma \ref{lm:G_distance_functions} and using the fact that $G$ is a cone, for each $k \in \NN$ there exists $g_k \in G \cap 2B_{C(\bwt{M})}$ such that
\[
 g_k(x,y) = \min\set{1,\frac{\bar{d}(x,B_k)}{d(x,y)}} + \min\set{1,\frac{\bar{d}(y,B_k)}{d(x,y)}}, \quad (x,y) \in \wt{M}.
\]
Then $(g_k)$ is a bounded and decreasing sequence of positive elements of $G$, thus its pointwise limit $g$ exists and belongs to $\mathfrak{G}$. 

We show that $\bwtf = g^{-1}(0)$. Let $\zeta \in \bwt{M}\setminus\bwtf$. We claim that $g(\zeta) \geq 1$. Assume without loss of generality that $\zeta \notin \pp_1^{-1}(\rcomp{M})$, and let $(x_i,y_i)$ be a net of points in $\wt{M}$ converging to $\zeta$. We note
\begin{equation}\label{eqn:bwtf_preserved_1}
 |\bar{d}(x,B_k)-d(x,0)| \leq k, \quad x \in M,
\end{equation}
hence $\lim_i \bar{d}(x_i,B_k)=\infty$. There are two cases: $d(\zeta)<\infty$ and $d(\zeta)=\infty$. If $d(\zeta)<\infty$ then
\[
 g_k(\zeta) = \lim_i g_k(x_i,y_i) \geq \lim_i \min\set{1,\frac{\bar{d}(x_i,B_k)}{d(x_i,y_i)}} = 1.
\]
If $d(\zeta)=\infty$ then, again by \eqref{eqn:bwtf_preserved_1}, 
\[
\lim_i \frac{\bar{d}(x_i,B_k) + \bar{d}(y_i,B_k)}{d(x_i,y_i)} = \lim_i \frac{d(x_i,0) + d(y_i,0)}{d(x_i,y_i)} \geq 1,
\]
thus
\[
 g_k(\zeta) = \lim_i g_k(x_i,y_i) \geq \lim_i \min\set{1, \frac{\bar{d}(x_i,B_k) + \bar{d}(y_i,B_k)}{d(x_i,y_i)}} \geq 1.
\]
This completes the proof of the claim. On the other hand, $\bwtf = \bigcup_{k=1}^\infty L_k$, where $L_k := \pp^{-1}(B_k \times B_k)$. If $\zeta \in L_k$ then $x_i,y_i \in B_{k+1}$ for all $i$ large enough by Proposition \ref{pr:bar-d_cont} \ref{pr:bar-d_cont_3}. Hence $\lim_i \bar{d}(x_i,B_{k+1})=\lim_i \bar{d}(y_i,B_{k+1})=0$, giving $g_{k+1}(\zeta)=0$. We deduce that $g(\zeta)=0$ whenever $\zeta \in \bwtf$. The second assertion of the proposition follows immediately by Proposition \ref{pr:zero_set}. 
\end{proof}

The next example shows that the preimages $\pp_1^{-1}(\rcomp{M}), \pp_2^{-1}(\rcomp{M})$ (and their complements) don't belong to $\mathfrak{C}$ in general.

\begin{example}\label{ex:preimages_not_preserved}
 Let $M$, $\zeta$ and $\omega$ be as in Example \ref{ex:sep_at_infinity_fails}. As $g(\zeta)=g(\omega)$ for all $g \in G$, we have $\delta_{\omega} \preccurlyeq \delta_\zeta$ and vice-versa. As $M$ is proper we have $\rcomp{M}=M$. In this case $\delta_\zeta$ and $\delta_{\omega}$ are concentrated on $\pp_1^{-1}(M)$ and $\pp_1^{-1}(\ucomp{M}\setminus M)$, respectively, because $\pp_1(\zeta)=1 \in M$ and $\pp_1(\omega) \notin M$.
\end{example}

Our next (indirect) application of Lemma \ref{lm:G_distance_functions} is Theorem \ref{th:C_conc}. Example \ref{ex:preimages_not_preserved} shows that preimages of nice subsets of $\rcomp{M}$ don't belong to $\mathfrak{C}$ in general. However, if one restricts one's attention to the better behaved realm of $\bwtf$, we get some positive results. Theorem \ref{th:C_conc} below shows that $\bwtf \cap \pp_i^{-1}(E) \in \mathfrak{C}$, $i=1,2$, provided $E \subset \rcomp{M}$ is sufficiently regular.  Before presenting the result we introduce two more lemmas, together with a variant of $\mathfrak{G}$ that will be of use to us again later.

\begin{lemma}\label{lm:dist_pointwise_sup}
Let $A \subset \rcomp{M}$ be closed and non-empty, and let $\Lambda$ be a directed set of subsets of $\rcomp{M}$, ordered by reverse inclusion, such that $A \subset L$ for all $L \in \Lambda$, and given $U \supset A$ open in $\rcomp{M}$, there exists $L \in \Lambda$ satisfying $L \subset U$. Then 
 \[
  \bar{d}(\xi,A) = \sup_{L \in \Lambda} \bar{d}(\xi,L), \quad \xi \in \rcomp{M}.
 \]
\end{lemma}

\begin{proof}
Fix $\xi \in \rcomp{M}$. Clearly $\bar{d}(\xi,A) \geq \bar{d}(\xi,L)$ for all $L \in \Lambda$. Now let $\ep>0$. Given $L \in \Lambda$, pick $\eta_L \in L$ such that $\bar{d}(\xi,L) > \bar{d}(\xi,\eta_L) - \ep$. Set $C=\bar{d}(\xi,L)+\ep$ and let $K=A \cap B_{\bar{d}}(\xi,C)$, which is compact by Proposition \ref{pr:rcomp_observations} \ref{pr:rcomp_observations_2}.

We claim that the net $(\eta_L)$ clusters at a point $\eta \in K$. Indeed, if not, given $\eta \in K$, there exists open $U_\eta \ni \eta$ and $L_\eta \in \Lambda$ such that $\eta_L \notin U_\eta$ whenever $L \in \Lambda$, $L \subset L_\eta$. Let $\eta_1,\ldots,\eta_n \in K$ such that $K \subset U:=\bigcup_{k=1}^n U_{\eta_k}$, and $L' \in \Lambda$ such that $L' \subset \bigcap_{k=1}^n L_{\eta_k}$. Define the open set
\[
V = U \cup (\rcomp{M}\setminus B_{\bar{d}}(\xi,C)),
\]
which includes $A$, and then pick $L \in \Lambda$ satisfying $L \subset V \cap L'$. We see that
\[
 \eta_L \in L \cap B_{\bar{d}}(\xi,C) \subset V \cap B_{\bar{d}}(\xi,C) \subset U,
\]
thus $\eta_L \in U_{\eta_k}$ for some $k$, which is a contradiction.

Given our cluster point $\eta \in K \subset A$, by the lower semicontinuity of $\bar{d}$ (Proposition \ref{pr:rcomp_observations} \ref{pr:rcomp_observations_1}) fix open $W \ni \eta$ such that $\bar{d}(\xi,\eta') > \bar{d}(\xi,\eta)-\ep$ whenever $\eta' \in W$. Then pick $\eta_L \in W$ to get 
 \[
  \bar{d}(\xi,A) \leq \bar{d}(\xi,\eta) < \bar{d}(\xi,\eta_L) + \ep < \bar{d}(\xi,L) + 2\ep. \qedhere
 \]
\end{proof}

Now we define our variant of $\mathfrak{G}$. Let $\bigvee$ denote the pointwise supremum of a set of real-valued functions. Set
\[
 \Glsc = \set{\bigvee E : E \subset G \text{ is non-empty and bounded}} \subset C(\bwt{M})^{**}.
\]
Clearly $G \subset \Glsc$ and every element of $\Glsc$ is a bounded lower semicontinuous real-valued function on $\bwt{M}$ that is constant on fibres of $\gq$. As $G$ is a max-stable cone in $C(\bwt{M})$ that is invariant under the action of the reflection map $\rr$, $\Glsc$ is a max-stable cone in $C(\bwt{M})^{**}$ also invariant under $\rr$. Let $E \subset G$ be non-empty and bounded and let $g=\bigvee E \in \Glsc$. As $G$ is max-stable, the max-stable closure
\[
E':=\set{g_1 \vee \ldots \vee g_n \,:\, g_1,\ldots, g_n \in E,\, n \in \NN},
\]
of $E$ is again a bounded subset of $G$, and it is evident that $g=\bigvee E'$. Thus we can always assume that $E \subset G$ is max-stable to begin with. This makes $E$ an upwards-directed family, so by Lebesgue's monotone convergence theorem for nets \cite[Theorem A.84]{lmns10}, we conclude that every positive element of $\Glsc$ belongs to $\mathfrak{G}$.

\begin{lemma}\label{lm:G_distance_functions_closed}
Let $A \subset \rcomp{M}$ be non-empty and closed, $i \in \set{1,2}$ and $n \in \NN$. Then there exists $g_{n,A} \in \Glsc \cap \mathfrak{G}$, taking values in $[0,n]$, such that
\begin{equation}\label{eqn:G_distance_functions_compact_1}
g_{n,A}(\zeta) = \begin{cases}\displaystyle \min\set{n,\frac{\bar{d}(\pp_i(\zeta),A)}{d(\zeta)}} & \text{if } \zeta \in \pp_i^{-1}(\rcomp{M}\setminus A),\\
0 & \text{if } \zeta \in \pp_i^{-1}(A).
\end{cases}
\end{equation}
Moreover, $g_{m,B} \leq g_{n,A}$ whenever $m \leq n$, $A \subset B \subset \rcomp{M}$ and $B$ is closed.
\end{lemma}

\begin{proof}
Without loss of generality we assume $i=1$; the case $i=2$ follows by considering $g_{n,A} \circ \rr \in \Glsc$. Consider any set $U \supset A$ open in $\rcomp{M}$. Since $U \subset \rcomp{M}$ is open and $M$ is dense in $\rcomp{M}$, by Proposition \ref{pr:bar-d_cont} \ref{pr:bar-d_cont_2} we see that $\bar{d}(\cdot,\rcl{U}) = \bar{d}(\cdot,\rcl{U \cap M}) = \rcomp{d}(\cdot,U \cap M)$. Again because $U$ is open, $U \subset (\rcomp{U})^\circ$. By applying Lemma \ref{lm:G_distance_functions}, there exists positive $g_{n,U} \in G \cap nB_{C(\bwt{M})}$ satisfying
\begin{equation}\label{eqn:G_distance_functions_compact_2}
 g_{n,U}(x,y) = \min\set{n,\frac{d(x,U \cap M)}{d(x,y)}}, \quad (x,y) \in \wt{M},
\end{equation}
and
\begin{equation}\label{eqn:G_distance_functions_compact_3}
g_{n,U}(\zeta) = \begin{cases}
                  \displaystyle \min\set{n,\frac{\bar{d}(\pp_1(\zeta),\rcl{U})}{d(\zeta)}} & \text{if } \zeta \in \pp_1^{-1}(\rcomp{M}\setminus \rcl{U})\\
                  0 & \text{if } \zeta \in \pp_1^{-1}(U).
                 \end{cases}
\end{equation}
Define $g_{n,k} \in \Glsc$ by $g_{n,A}(\zeta)=\sup_U g_{n,U}(\zeta)$, where the $U$ range over all supersets of $A$ that are open in $\rcomp{M}$, as above. Because $g_{n,A}$ takes values in $[0,n]$, we see from above that $g_{n,A} \in \mathfrak{G}$.

We claim that $g_{n,A}$ satisfies \eqref{eqn:G_distance_functions_compact_1}. Indeed, because $\rcomp{M}$ is normal by Proposition \ref{pr:rcomp_observations} \ref{pr:rcomp_observations_5}, we can apply Lemma \ref{lm:dist_pointwise_sup} to conclude that
\[
 \bar{d}(\pp_1(\zeta),A) = \sup_{U \supset A} \bar{d}(\pp_1(\zeta),\rcl{U}).
\]
Hence, if $\pp_1(\zeta) \in \rcomp{M}\setminus A$ then \eqref{eqn:G_distance_functions_compact_3} will hold for all small enough open sets $U \supset A$, and taking the supremum of both sides of \eqref{eqn:G_distance_functions_compact_3} yields the stated formula for $g_{n,A}(\zeta)$ in this case. If $\pp_1(\zeta) \in A$ then $g_{n,U}(\zeta)=0$ for all open supersets $U$ of $A$ by \eqref{eqn:G_distance_functions_compact_3}, giving $g_{n,A}(\zeta)=0$ as claimed.

Finally, let $m \leq n$ and $A \subset B$ be non-empty and closed. Given an open set $U \supset B$, it is clear from \eqref{eqn:G_distance_functions_compact_2} that $g_{m,U} \leq g_{n,U}$. Consequently,
\[
 g_{m,B} = \sup_{U \supset B} g_{m,U} \leq \sup_{U \supset A} g_{n,U} = g_{n,A}. \qedhere
\]
\end{proof}

\begin{theorem}\label{th:C_conc}
Let $E \subset \rcomp{M}$ be $\bar{d}$-closed and Borel. Then $\bwtf \cap \pp_i^{-1}(E) \in \mathfrak{C}$, $i=1,2$.
\end{theorem}

\begin{proof}
Again we use Proposition \ref{pr:zero_set}. Without loss of generality we assume $i=1$. Let $F \subset \bwtf \cap \pp_1^{-1}(E)$ be a non-empty $K_\sigma$ set. Then $\pp_1(F) \subset E$ is a $K_\sigma$ set; write $\pp_1(F)=\bigcup_{n=1}^\infty K_n$, where $(K_n)$ is an increasing sequence of non-empty compact subsets of $\rcomp{M}$. According to Lemma \ref{lm:G_distance_functions_closed}, there exists a decreasing sequence of bounded functions $g_n \in \mathfrak{G}$ such that
\begin{equation}\label{eqn:C_conc_1}
g_n(\zeta) = \begin{cases}\displaystyle \min\set{1,\frac{\bar{d}(\pp_1(\zeta),K_n)}{d(\zeta)}} & \text{if } \zeta \in \pp_1^{-1}(\rcomp{M}\setminus K_n),\\
0 & \text{if } \zeta \in \pp_1^{-1}(K_n).
\end{cases}
\end{equation}
Consequently, by the monotone convergence theorem, the pointwise limit $g'$ of the $g_n$ also belongs to $\mathfrak{G}$. By Corollary \ref{co:bwtf_preserved}, there exists $g'' \in \mathfrak{G}$ such that $\bwtf=g''^{-1}(0)$. Define $g=g'+g'' \in \mathfrak{G}$.

We verify that $F \subset g^{-1}(0) \subset \bwtf \cap \pp_1^{-1}(E)$. Given $\zeta \in F$, we have $\zeta \in \bwtf$, meaning that $g''(\zeta)=0$. We also have $\pp_1(\zeta) \in \pp_1(F)$, so $\pp_1(\zeta) \in K_n$ for all large enough $n$, giving $g'(\zeta)=0$ by \eqref{eqn:C_conc_1}. Hence $g(\zeta)=0$.

If $g(\zeta)=0$ then $g'(\zeta)=g''(\zeta)=0$. Now $g''(\zeta)=0$ implies $\zeta \in \bwtf$, which means in particular that $\zeta \in \pp_1^{-1}(\rcomp{M})$ and $d(\zeta)<\infty$. Either $\pp_1(\zeta) \in \pp_1(F) \subset E$, giving $\zeta \in \pp_1^{-1}(E)$, or $\pp_1(\zeta) \notin \pp_1(F)$. In the second case $\zeta \in \pp_1^{-1}(\rcomp{M}\setminus K_n)$ for all $n \in \NN$, so $g'(\zeta)=0$ and \eqref{eqn:C_conc_1} imply
\[
 \lim_n \min\set{1,\frac{\bar{d}(\pp_1(\zeta),K_n)}{d(\zeta)}} = \lim_n g_n(\zeta) = 0.
\]
which in turn implies $d(\zeta) \in (0,\infty)$ and so $\lim_n \bar{d}(\pp_1(\zeta),K_n)=0$, giving $\pp_1(\zeta) \in \cl{\pp_1(F)}^{\bar{d}} \subset E$. This means that $\zeta \in \pp_1^{-1}(E)$, as required.
\end{proof}

It is not possible to remove the hypothesis that $E$ is $\bar{d}$-closed from Theorem \ref{th:C_conc}.

\begin{example}\label{ex:C_conc}
Let $M=[0,1]$ have the usual metric and base point $0$, and let $C=[0,1]\setminus\bigcup_{n=1}^\infty (a_n,b_n)$ denote the usual middle-thirds Cantor set, where the intervals $(a_n,b_n)$ are disjoint. Set
 \[
  \nu = \sum_{n=1}^\infty (b_n-a_n)\delta_{(b_n,a_n)}.
 \]
We verify that $\Phi^*\nu = m_{10}$. Indeed, given $f \in \Lip_0(M)$, by absolute continuity,
\[
\duality{f,\Phi^*\nu} = \duality{\Phi f,\nu} = \sum_{n=1}^\infty f(b_n)-f(a_n) = \sum_{n=1}^\infty \int_{a_n}^{b_n} f'(t) \,dt = \int_0^1 f'(t) \,dt = f(1)-f(0) = \duality{f,m_{10}}.
\]
Hence $\delta_{(1,0)} \preccurlyeq \nu$ by Proposition \ref{pr:least_measures}. However, $\nu$ is concentrated on $\bwtf \cap \pp_1^{-1}(E)$, where $E:=\set{b_n \,:\, n \in \NN}$, while on the other hand $(1,0) \notin \pp_1^{-1}(E)$.
\end{example}

Recall from Section \ref{subsec:supp} the notion of the shadow of the support of a measure. Theorem \ref{th:C_conc} allows us to show that shadows of the supports of measures concentrated on $\bwtf$ behave well with respect to $\preccurlyeq$. 

\begin{corollary}\label{cr:support}
Let $\mu,\nu \in \meas{\bwt{M}}^+$ such that $\mu \preccurlyeq \nu$ and $\nu$ is concentrated on $\bwtf$. Then $\pp_i(\supp(\mu)) \subset \pp_i(\supp(\nu))$, $i=1,2$.
\end{corollary}

\begin{proof}
Without loss of generality assume $i=1$. Consider $E=\rcomp{M} \cap \pp_1(\supp(\nu))$, which is closed in $\rcomp{M}$ and thus $\bar{d}$-closed as well. We see that $\nu$ is concentrated on $\bwtf \cap \supp(\nu) \subset \bwtf \cap \pp_1^{-1}(E)$, hence $\mu$ is concentrated on $\bwtf \cap \pp_1^{-1}(E)$ by Theorem \ref{th:C_conc}. Now let $U=\ucomp{M}\setminus\pp_1(\supp(\nu))$, which is open in $\ucomp{M}$. Then
\[
\mu(\pp_1^{-1}(U)) = \mu(\pp_1^{-1}(U) \cap \bwtf \cap \pp_1^{-1}(E)) \leq \mu(\pp_1^{-1}(U \cap \pp_1(\supp(\nu)))) = 0.
\]
Given that $\pp_1^{-1}(U)$ is open, this implies $\supp(\mu) \cap \pp_1^{-1}(U) = \varnothing$, which in turn implies $\pp_1(\supp(\mu)) \cap U = \varnothing$, i.e. $\pp_1(\supp(\mu)) \subset \pp_1(\supp(\nu))$. 
\end{proof}

We remark that supports themselves do not behave well with respect to $\preccurlyeq$. Indeed, consider the measure $\nu$ in Example \ref{co:somewhat_anti-sym}. In this case $\delta_{(1,0)} \preccurlyeq \nu$, but $\supp(\delta_{(1,0)}) = \set{(1,0)} \not\subset \set{(1,\frac{1}{2}),(\frac{1}{2},0)} = \supp(\nu)$.

In Section \ref{subsec:shadow} we find an upper bound for the shadow of minimal measures concentrated on $d^{-1}[0,\infty)$.

We continue with some more examples of sets that belong to $\mathfrak{C}$.

\begin{corollary}\label{co:wt}
 We have $\bwtf\setminus d^{-1}(0), \pp^{-1}(M \times M),\wt{M} \in \mathfrak{C}$.
\end{corollary}

\begin{proof}
By Corollary \ref{co:bwtf_preserved} and Proposition \ref{pr:diag}, $\bwtf\setminus d^{-1}(0) \in \mathfrak{C}$. Since $M$ is $\bar{d}$-closed and Borel, $\bwtf \cap \pp_i^{-1}(M) \in \mathfrak{C}$, $i=1,2$, by Theorem \ref{th:C_conc}. Thus $\pp^{-1}(M \times M) = \bwtf \cap \pp_1^{-1}(M) \cap \pp_2^{-1}(M) \in \mathfrak{C}$ also. Finally, $\wt{M}=\pp^{-1}(M \times M)\setminus d^{-1}(0) \in \mathfrak{C}$ by this and Proposition \ref{pr:diag}.
\end{proof}

We finish the section with another example of a family of sets that belong to $\mathfrak{C}$. This is our promised generalisation of Proposition \ref{pr:diag}. Hereafter, given a Borel set $A \subset \rcomp{M}$, define $\Delta(A) = d^{-1}(0)\cap \pp^{-1}(A \times A)$. 

\begin{proposition}\label{pr:delta} If $A \subset \rcomp{M}$ is Borel then $\bwtf\setminus\Delta(A) \in \mathfrak{C}$. 
\end{proposition}

\begin{proof}
We will use Proposition \ref{pr:zero_set} once more. Fix a $K_\sigma$ set $F \subset \bwtf\setminus\Delta(A)$. We consider the $K_\sigma$ set $H:=F \cap d^{-1}(0)$. First we assume $H$ is non-empty. Notice that $\pp_1(H) \subset \rcomp{M}\setminus A$, because if $\zeta \in H$ then $\zeta \in \bwtf$, $d(\zeta)=0$ and $\zeta \notin \Delta(A)$, which means by Proposition \ref{pr:De_Leeuw_relation} \ref{pr:De_Leeuw_relation_3} that $\pp_1(\zeta) = \pp_2(\zeta) \in \rcomp{M}\setminus A$. Thus we can write $\pp_1(H) = \bigcup_{m=1}^\infty K_m$, where $(K_m)$ is an increasing sequence of non-empty compact subsets of $\rcomp{M}\setminus A$.

Use Lemma \ref{lm:G_distance_functions_closed} to define the functions $g_{n,K_m} \in \mathfrak{G}$, and set $g_{n,m} = \frac{1}{n}g_{n,K_m} \in \mathfrak{G}$. According to \eqref{eqn:G_distance_functions_compact_1}, for all $m,n \in \NN$ we have
 \begin{equation}\label{eqn:delta_1}
g_{n,m}(\zeta) = \begin{cases}\displaystyle \min\set{1,\frac{\bar{d}(\pp_1(\zeta),K_m)}{nd(\zeta)}} & \text{if } \zeta \in \pp_1^{-1}(\rcomp{M}\setminus K_m),\\
0 & \text{if } \zeta \in \pp_1^{-1}(K_m).
\end{cases}
\end{equation}
For a fixed $n\in \NN$, the bounded sequence $(g_{n,m})_m$ is decreasing. So if we define $g_n:\bwt{M}\to[0,1]$ by $g_n(\zeta) = \lim_m g_{n,m}(\zeta)$, we have $g_n \in \mathfrak{G}$ by the monotone convergence theorem. Similarly, for a fixed $m \in \NN$, the sequence $(g_{n,m})_n$, is decreasing, meaning that $(g_n)$ is also decreasing. So if we define $g':\bwt{M}\to[0,1]$ by $g'(\zeta)=\lim_n g_n(\zeta)$, then $g' \in \mathfrak{G}$ as well. Using Corollary \ref{co:bwtf_preserved}, let $g'' \in \mathfrak{G}$ such that $\bwtf=g''^{-1}(0)$, and define $g=g'+g'' \in \mathfrak{G}$.

To finish, we show that $F \subset g^{-1}(0) \subset \bwtf\setminus\Delta(A)$. Let $\zeta \in F$. There are two cases: $\zeta \in H$ or $\zeta \notin H$. If $\zeta \in H$ then $\pp_1(\zeta) \in K_m$ for all large enough $m$, meaning that $g_n(\zeta)=\lim_m g_{n,m}(\zeta)=0$ for all $n$, and thus $g(\zeta)=0$. If $\zeta \notin H$ then $d(\zeta)>0$, so
\[
g(\zeta) = \lim_n g_n(\zeta) \leq \lim_n g_{n,1}(\zeta) \leq \lim_n \frac{\bar{d}(\pp_1(\zeta),K_1)}{nd(\zeta)} = 0.
\]
Now we show that if $g(\zeta)=0$ then $\zeta \in \bwtf\setminus\Delta(A)$. Given such $\zeta$, as $g''(\zeta)=0$ we have $\zeta \in \bwtf$. If $d(\zeta)>0$ then $\zeta \in \bwtf\setminus\Delta(A)$, so we assume from now on that $d(\zeta)=0$. If $\zeta \in \pp_1^{-1}(\rcomp{M}\setminus K_m)$ for all $m \in \NN$, then according to \eqref{eqn:delta_1}, $g_{n,m}(\zeta)=1$ for all $m,n \in \NN$, which implies $g(\zeta) = 1 \neq 0$. This implies that for some $m \in \NN$ we have $\pp_1(\zeta) \in K_m \subset \rcomp{M}\setminus A$, giving $\zeta \in \bwtf\setminus\Delta(A)$. This completes the proof when $H$ is non-empty.

If $H$ is empty then define instead $g'=\mathbf{1}_{d^{-1}(0)}$, which we know belongs to $\mathfrak{G}$ from Proposition \ref{pr:diag}. Then it is straightforward to check that $g:=g'+g''$ satisfies $F \subset g^{-1}(0) \subset \bwtf\setminus d^{-1}(0) \subset \bwtf\setminus\Delta(A)$.
\end{proof}

\section{A characterisation of minimality and consequences}\label{sec:min_charac}

\subsection{A characterisation of minimality}\label{subsec:min_charac}

The main result of this section is a characterisation of minimality. It is the key to all of the deeper results about minimality. It borrows from \cite[Lemma 3.21]{lmns10}, which is an essential component of Mokobodzki's test for measures that are maximal with respect to the classical Choquet order \cite[Theorem 3.58]{lmns10}. Theorem \ref{th:minimality_test} is a test for minimality, so the results are different. In Section \ref{subsec:shadow} we use it to show that the shadows of minimal elements concentrated on $d^{-1}[0,\infty)$ are themselves minimal, modulo the base point $0$. In Section \ref{subsec:optimality_minimality}, it is used to show that minimality does not imply optimality. Finally, in the remaining sections, it is used to prove a series of results concerning the marginals of minimal elements which, together with the characterisation, will be needed in Section \ref{subsec:opt_conc}.

Before presenting it, we note by the universal property of quotient spaces that $f \in C(\bwt{M})$ is constant on all fibres of $\gq$ if and only if $f = k \circ \gq$ for some $k \in C(\wt{M}^G)$.

\begin{theorem}[{cf. \cite[Lemma 3.21]{lmns10}}]\label{th:minimality_test}
Let $\mu \in \meas{\bwt{M}}^+$. Then the following are equivalent.
\begin{enumerate}[label={\upshape{(\roman*)}}]
\item\label{pr:minimality_test_1} $\mu$ is minimal; 
\item\label{pr:minimality_test_2} given $f \in C(\bwt{M})$ constant on all fibres of $\gq$,
 \begin{equation}\label{eqn:minimality_test}
  \duality{f,\mu} = \inf\set{\duality{g,\mu} \,:\ g \in G,\, f \leq g};
 \end{equation}
\item\label{pr:minimality_test_3} given an upper semicontinuous function $f:\bwt{M}\to\RR$ that is constant on all fibres of $\gq$, 
 \[
  \int_{\bwt{M}} f \,d\mu = \inf\set{\duality{g,\mu} \,:\ g \in G,\, f \leq g};
 \]
 \item\label{pr:minimality_test_4} given a Borel set $A \subset \wt{M}^G$ such that $\mu$ is concentrated on $\gq^{-1}(A)$, and an upper semicontinuous function $f:\bwt{M}\to\RR$ that is constant on all fibres of $\gq$ included in $\gq^{-1}(A)$, 
 \[
  \int_{\bwt{M}} f \,d\mu = \inf\set{\duality{g,\mu} \,:\ g \in G,\, f \leq g}.
 \]
 \end{enumerate}
\end{theorem}

\begin{proof}
\ref{pr:minimality_test_1} $\Rightarrow$ \ref{pr:minimality_test_2}: let $\mu$ be minimal and suppose that $f \in C(\bwt{M})$ is constant on all fibres of $\gq$. To obtain \eqref{eqn:minimality_test}, we use a Hahn-Banach extension argument. Define $p:C(\bwt{M}) \to \RR$ by
\[
p(k)= \inf\set{\duality{g,\mu} \,:\ g \in G,\, k \leq g}
\]
Then $p$ is subadditive and positively homogeneous. Indeed, given $k,k' \in C(\bwt{M})$ with $g,g' \in G$, $k \leq g$, $k' \leq g'$, we have $g+g' \in G$, $k+k' \leq g+g'$, and thus
\[
 p(k+k') \leq \duality{g+g',\mu} = \duality{g,\mu} + \duality{g',\mu},
\]
whence $p(k+k') \leq p(k) + p(k')$. As $G$ is a cone, positive homogeneity is also easy to establish. Since $k \leq \norm{k}_\infty\cdot\mathbf{1}_{\bwt{M}}$ and the latter function belongs to $G$, we also have $p(k) \leq \norm{k}\norm{\mu}$. Now define the linear functional $\psi:\lspan(f)\to\RR$ by $\duality{tf,\psi}=tp(f)$; because $0=p(0)=p(f-f) \leq p(f) + p(-f)$, we have $\psi \leq p\restrict_{\lspan(f)}$. By the Hahn-Banach and Riesz representation theorems, there exists $\nu \in \meas{\bwt{M}}$ such that $p(f)=\duality{f,\nu}$ and $\duality{k,\nu} \leq p(k)$ for $k \in C(\bwt{M})$ (so in particular $\norm{\nu} \leq \norm{\mu}$). Furthermore, given $k \geq 0$, we have $-k \leq 0$ and $0 \in G$, so $-\duality{k,v}=\duality{-k,v} \leq p(-k) \leq 0$, meaning that $\nu \in \meas{\bwt{M}}^+$. As $\duality{g,\nu} \leq p(g) = \duality{g,\mu}$ for all $g \in G$, we have $\nu \preccurlyeq \mu$. By minimality of $\mu$ and Proposition \ref{pr:somewhat_anti-sym}, it follows that $\gq_\sharp\nu = \gq_\sharp\mu$. Since $f = k \circ \gq$ for some $k \in C(\wt{M}^G)$, we obtain $\duality{f,\nu}=\duality{f,\mu}$. Finally, we have $\duality{f,\mu} = \duality{f,\nu} = p(f)$, which yields the conclusion.

\ref{pr:minimality_test_2} $\Rightarrow$ \ref{pr:minimality_test_3}: let $f:\bwt{M}\to\RR$ be upper semicontinuous and constant on all fibres of $\gq$. Then $f=k \circ \gq$ for some $k:\wt{M}^G\to\RR$, and from the definition of quotient topology we see that $k$ is upper semicontinuous also. According to e.g.~\cite[Proposition A.50]{lmns10}, $k$ is the pointwise infimum of all elements $\ell \in C(\wt{M}^G)$ satisfying $k \leq \ell$. As the set of all such $\ell$ is min-stable and thus a downwards-directed family, we have
\[
 \int_{\bwt{M}} f \,d\mu = \inf\set{\duality{h,\mu} \,:\, h \in C(\bwt{M}) \text{ is constant on all fibres of $\gq$ and $f \leq h$}}
\]
by Lebesgue's monotone convergence theorem for nets \cite[Theorem A.84]{lmns10}. The conclusion follows.

\ref{pr:minimality_test_3} $\Rightarrow$ \ref{pr:minimality_test_4}: let $A \subset \wt{M}^G$ be a Borel set such that $\mu$ is concentrated on $\gq^{-1}(A)$, and let $f:\bwt{M}\to\RR$ be an upper semicontinuous function that is constant on all fibres of $\gq$ included in $\gq^{-1}(A)$. Define a new function $\bar{f}:\bwt{M}\to\RR$ by
\[
\bar{f}(\zeta) = \max\set{f(\omega) \,:\, \omega \in \gq^{-1}(\gq(\zeta))}.
\]
This is well-defined because, on compact subsets, upper semicontinuous functions are bounded above and attain their suprema. Clearly $\bar{f}$ is constant on fibres of $\gq$ and agrees with $f$ on $\gq^{-1}(A)$. We claim that $\bar{f}$ is also upper semicontinuous. Let $\zeta \in \bwt{M}$ and $\alpha>\bar{f}(\zeta)$. Define the set $U=f^{-1}(-\infty,\alpha)$, which is open by the upper semicontinuity of $f$, and a second open set $V=\bwt{M}\setminus \gq^{-1}(\gq(\bwt{M}\setminus U))$. We claim that $\zeta \in V$ and $\bar{f}(\zeta') < \alpha$ whenever $\zeta' \in V$. Indeed, if $\zeta \notin V$ then $\gq(\zeta)=\gq(\omega)$ for some $\omega \in \bwt{M}\setminus U$, yielding $\bar{f}(\zeta) \geq f(\omega) \geq \alpha > \bar{f}(\zeta)$ which is impossible. Given $\zeta' \in V$, let $\omega' \in \gq^{-1}(\gq(\zeta'))$ such that $\bar{f}(\zeta')=f(\omega')$. Then $\omega' \in U$ because $\zeta' \notin \gq^{-1}(\gq(\bwt{M}\setminus U))$. Hence $\bar{f}(\zeta') < \alpha$. This completes the proof of the claim.

Because $\mu$ is concentrated on $\gq^{-1}(A)$ and $f$ and $\bar{f}$ agree on this set, we have
\[
\int_{\bwt{M}} f \,d\mu = \int_{\bwt{M}} \bar{f} \,d\mu.
\]
Moreover, given $g \in G$, we have $f \leq g$ if and only if $\bar{f} \leq g$. Indeed, one implication is immediate and the other follows because $g$ is constant on fibres of $\gq$. Hence we can apply \ref{pr:minimality_test_3} to $\bar{f}$ to obtain the desired conclusion. 

\ref{pr:minimality_test_4} $\Rightarrow$ \ref{pr:minimality_test_3} and \ref{pr:minimality_test_3} $\Rightarrow$ \ref{pr:minimality_test_2} are immediate.

\ref{pr:minimality_test_2} $\Rightarrow$ \ref{pr:minimality_test_1}: let $\mu \in \meas{\bwt{M}}^+$ and suppose
\[
  \duality{f,\mu} = \inf\set{\duality{g,\mu} \,:\ g \in G,\, f \leq g}.
\]
whenever $f \in C(\bwt{M})$ is constant on all fibres of $\gq$. Let $\nu \geq 0$ satisfy $\nu \preccurlyeq \mu$. We show that $\mu \preccurlyeq \nu$ and thus that $\mu$ is minimal. Let $h \in G$ and $\ep>0$. Since $-h$ is constant on all fibres of $\gq$, we can pick $g \in G$ satisfying $-h \leq g$ and $\duality{g,\mu} < \duality{-h,\mu} + \ep$. As $\nu \geq 0$ and $\nu \preccurlyeq \mu$
\[
 -\duality{h,\nu} = \duality{-h,\nu} \leq \duality{g,\nu} \leq \duality{g,\mu} < \duality{-h,\mu} + \ep = -\duality{h,\mu} + \ep,
\]
giving $\duality{h,\mu} < \duality{h,\nu}+\ep$.
Because this holds for all $h \in G$ and $\ep>0$, we obtain $\duality{h,\nu}=\duality{h,\mu}$ for all $h \in G$, as required.
\end{proof}

The essential idea behind most of the following applications of Theorem \ref{th:minimality_test} is to choose functions $f$ that are ``far away'' from $G$ in some sense (certainly, picking $f \in G$ yields no information). In this way, the fact that we can approximate $\int_{\bwt{M}} f \,d\mu$ by $\duality{g,\mu}$, $g \in G$, $f \leq g$, despite $f$ being far away from $G$, should provide information about $\mu$.

\subsection{The shadows of minimal measures}\label{subsec:shadow}

In this section we show that, under an assumption milder than the one in Proposition \ref{pr:minimal_support_representation}, the shadows of minimal measures satisfy the same conclusion (and in so doing, we slightly generalise this result). This shows that minimal measures are also ``minimally supported'', which lends weight to the argument that it is the minimal measures that fit most naturally into the existing theory of supports and extended supports of elements of $\lipfree{M}$ and $\Lip_0(M)^*$.

We start by noting that, by Proposition \ref{pr:least_measures}, if $(x,y) \in \wt{M}$, $g \in G$ and $\mu \in \meas{\bwt{M}}^+$ are such that $g(x,y) \geq 0$ and $\Phi^*\mu=m_{xy}=\Phi^*\delta_{(x,y)}$, then $\duality{g,\mu} \geq \duality{g,\delta_{(x,y)}} = g(x,y) \geq 0$. It turns out that the inequality $\duality{g,\mu} \geq 0$ is important in some applications of Theorem \ref{th:minimality_test}. 

\begin{proposition}\label{pr:G-support}
Let $g \in G$ and $A \subset M$ such that $0 \in A$ and $g(x,y) \geq 0$ for all $(x,y) \in \wt{A}$. Then there exists $f \in \Lip_0(M)$ such that $f$ vanishes on $A$ and $\Phi f \leq g$. Consequently, if $\mu \in \meas{\bwt{M}}^+$ has the property that $\duality{h,\Phi^*\mu}=0$ whenever $h \in \Lip_0(M)$ vanishes on $A$, then $\duality{g,\mu} \geq 0$.
\end{proposition}

\begin{proof} Given $g \in G$, let $h$ be the associated map from Definition \ref{df:assoc_map}. Define $f:M \to \RR$ by $f(x)=\inf_{a \in A} h(x,a)$. First we make sure that $f$ is well-defined. Given $x \in M$ and $a \in A$,
\[
h(x,a) \geq h(0,a) - h(0,x) \geq -d(0,x)\norm{g}_\infty,
\]
by \eqref{tri_ineq}, the definition of $h$ and the fact that $h(0,a) \geq 0$ whenever $a \in A$. Hence $f$ is well-defined. Next, given $x \in A$, $h(x,x) = 0 \leq h(x,a)$ for all $a \in A$, so $f(x)=0$. Finally, let $(x,y) \in \wt{M}$. Given $\ep>0$, pick $a \in A$ such that $f(y) > h(y,a)-\ep d(x,y)$. Again by \eqref{tri_ineq},
\[
f(x)-f(y) < f(x) - h(y,a) + \ep d(x,y) \leq h(x,a)-h(y,a)+\ep d(x,y) \leq h(x,y) + \ep d(x,y),
\]
giving $\Phi f(x,y) < g(x,y)+\ep$. As this holds for all $(x,y) \in \wt{M}$ and $\ep>0$, and $\Phi f$ and $g$ are continuous on $\bwt{M}$, we obtain $\Phi f \leq g$ (of course, since $g$ is bounded, this also implies $f$ is Lipschitz, so it belongs to $\Lip_0(M)$).

Finally, let $\mu \in \meas{\bwt{M}}^+$ have the property that $\duality{h,\Phi^*\mu}=0$ whenever $h \in \Lip_0(M)$ vanishes on $A$. Then $\duality{g,\mu} \geq \duality{\Phi f,\mu}=\duality{f,\Phi^*\mu}=0$ as required.
\end{proof}

Armed with this, we can give our result on the shadows of minimal measures.

\begin{theorem}\label{th:minimal_shadow} Let $\mu \in \meas{\bwt{M}}^+$ be minimal and concentrated on $d^{-1}[0,\infty)$. Then $\pp_s(\supp(\mu)) \subset \esupp{\Phi^*\mu} \cup \set{0}$.
\end{theorem}

\begin{proof}
Set $\psi=\Phi^*\mu \in \Lip_0(M)^*$ and $K=\esupp{\psi} \cup \set{0} \subset \ucomp{M}$. Let $\mathcal{A}$ denote the set of all $A \subset M$ satisfying $0 \in A$ and having the property that $\duality{f,\psi}=0$ whenever $f \in \Lip_0(M)$ vanishes on $A$. Define the open set
\[
U_A = \bwt{M}\setminus\pp^{-1}\pare{\ucl{A} \times \ucl{A}}.
\]
We claim that $\mu(U_A)=0$. Before we prove this claim, we show how this yields the desired result. Given that $\mu(U_A)=0$ for all $A \in \mathcal{A}$, by inner regularity of $\mu$ and a simple compactness argument we obtain $\mu(\bigcup_{A \in \mathcal{A}} U_A)=0$ (cf.~\cite[Proposition 7.2.2 (i)]{Bogachev}). Therefore
\[
\supp(\mu) \subset \bwt{M}\setminus\pare{\bigcup_{A \in \mathcal{A}} U_A} = \bigcap_{A \in \mathcal{A}} \pp^{-1}\pare{\ucl{A} \times \ucl{A}} = \pp^{-1}(K \times K).
\]
The final equality above follows from Definition \ref{defn_extended_support}. From this we conclude that $\pp_s(\supp(\mu)) \subset K$, as required.

To finish, we show that $\mu(U_A)=0$ for all $A \in \mathcal{A}$. Let $A \in \mathcal{A}$ and $\ep>0$. Define $k=-\mathbf{1}_{U_A}$. We would like to apply Theorem \ref{th:minimality_test} \ref{pr:minimality_test_4} to $k$. To do this, we recall from Proposition \ref{pr:d_factor} the map $d^G$ and define the open set $V=(d^G)^{-1}[0,\infty) \subset \wt{M}^G$. Then $d^{-1}[0,\infty) = \gq^{-1}(V)$, and by Proposition \ref{co:g_fibres} \ref{co:g_fibres_1}, $k$ is constant on every fibre of $\gq$ included in $d^{-1}[0,\infty)$. Since $\mu$ is assumed to be concentrated on $d^{-1}[0,\infty)$, we can apply Theorem \ref{th:minimality_test} \ref{pr:minimality_test_4} to $k$ to obtain $g \in G$ satisfying $k \leq g$ and 
\[
\duality{g,\mu} < \int_{\bwt{M}} k \,d\mu + \ep = -\mu(U_A) + \ep.
\]
As $g(x,y) \geq k(x,y)=0$ whenever $(x,y) \in \wt{A}$, we conclude from Proposition \ref{pr:G-support} that $\duality{g,\mu} \geq 0$, whence $\mu(U_A) < \ep$. Since $\ep>0$ was arbitrary, we have $\mu(U_A)=0$ as claimed.
\end{proof}

We remark that, in general, it is not possible to remove the base point from Theorem \ref{th:minimal_shadow} so that the inclusion becomes $\pp_s(\supp(\mu)) \subset \esupp{\Phi^*\mu}$; see \cite[Proposition 5.5]{APS24b}.

\subsection{Optimality versus minimality}\label{subsec:optimality_minimality}

Optimality and minimality turn out to be incomparable properties. That optimality fails to imply minimality is easily seen by considering again the measure $\nu$ in Example \ref{ex:Choquet_motivation}. Here, $\nu$ is an optimal but non-minimal representation of $m_{10}$. Finding an example of a minimal but non-optimal representation takes a little more work, but one can be found using a 4-point metric space. Theorem \ref{th:minimality_test} is used here to demonstrate that the measure exhibited below is minimal.

\begin{example}\label{ex:minimal_non-optimal}
Consider $M=\set{0,a,b,c}$ with metric $d$ given by
\[
 d(x,y) = \begin{cases} 0 & \text{if }x=y \\ \frac{1}{2} & \text{if $(x,y)=(a,b)$ or $(x,y)=(b,a)$} \\ 1 & \text{otherwise.}\end{cases}
\]
(The triangle inequality is trivially satisfied because $\frac{1}{2} \leq d(x,y) \leq 1$ whenever $x \neq y$.)

Now set $\mu=\delta_{(0,a)}+\delta_{(b,c)}$. Then $\norm{\mu}=2$, $\Phi^*\mu=m_{0a}+m_{bc}$ and $\duality{f,\Phi^*\mu}=f(b)-f(a)-f(c)$, $f \in \Lip_0(M)$. If we set $\lambda=\frac{1}{2}\delta_{(b,a)}+\delta_{(0,c)}$, then $\norm{\lambda}=\frac{3}{2}$ and we have
\[
 \duality{f,\Phi^*\lambda} = \tfrac{1}{2}\cdot2(f(b)-f(a)) - f(c) = \duality{f,\Phi^*\mu}, \quad f \in \Lip_0(M),
\]
so $\Phi^*\lambda=\Phi^*\mu$, which implies that $\mu$ is not optimal (and by considering $f \in B_{\Lip_0(M)}$ given by $f(0)=f(b)=0$, $f(a)=-\frac{1}{2}$ and $f(b)=-1$, we see that $\lambda$ is optimal).

It remains to show that $\mu$ is minimal. For this we will use Theorem \ref{th:minimality_test} \ref{pr:minimality_test_2}. Specifically, given $f:\wt{M}\to\RR$, we show that there exists $g \in G$ such that $f \leq g$, $g(0,a)=f(0,a)$ and $g(b,c)=f(b,c)$. This will demonstrate that $\mu$ is minimal. The map $g$ will take the form $g=\max\set{\Phi k_1,\Phi k_2,\Phi k_3}$, where $k_1,k_2,k_3:M\to\RR$ all satisfy $k(0)=0$ (and are automatically Lipschitz).

To this end, let $f:\wt{M}\to\RR$ and set $K=\max f$, $\alpha=f(0,a)$ and $\beta=f(b,c)$. Given a parameter $t \in \RR$, define $k:M\to \RR$ by $k(0)=0$, $k(a)=-\alpha$, $k(b)=t$ and $k(c)=t-\beta$. Then it is clear that $\Phi k(0,a)=\alpha=f(0,a)$ and $\Phi k(b,c)=\beta=f(b,c)$, regardless of the value of $t$. Moreover, we calculate
\[
 \Phi k(0,b) = -t, \quad \Phi k(0,c) = \beta-t, \quad \Phi k(a,b) = -2(t+\alpha) \quad\text{and}\quad \Phi k(a,c) = \beta-\alpha-t.
\]
Hence by considering $L>0$ large enough and defining $k_1=k$ where $t=L$, and $k_2=k$ where $t=-L$, we can ensure that $\max\set{\Phi k_1,\Phi k_2}$ agrees with $f$ on $(0,a)$ and $(b,c)$, and is greater than or equal to $K$ on all other $(x,y) \in \wt{M}$ besides $(a,0)$ and $(c,b)$. Finally, if we define $k_3:M\to \RR$ by $k_3(0)=k_3(b)=0$, $k_3(a)=\max\set{K,-\alpha}$ and $k_3(c)=\max\set{K,-\beta}$, we can verify that
\[
 \Phi k_3(0,a) \leq \alpha = f(0,a), \quad \Phi k_3(a,0) \geq K, \quad \Phi k_3(b,c) \leq \beta = f(b,c) \quad \text{and} \quad \Phi k_3(c,b) \geq K.
\]
Therefore $g:=\max\set{\Phi k_1,\Phi k_2,\Phi k_3}$ agrees with $f$ on $(0,a)$ and $(b,c)$, and is greater than or equal to $K$ (and hence $f$) on all other $(x,y) \in \wt{M}$, as required.
\end{example}

\subsection{Elimination of shared coordinates}\label{subsec:eliminating_coordinates}

Corollary \ref{co:minimality_coordinates} below is our third application of Theorem \ref{th:minimality_test}. It demonstrates that minimal elements concentrated on $d^{-1}(0,\infty)$ are further concentrated on a Borel subset of $d^{-1}(0,\infty)$ that doesn't admit pairs of points that share coordinates. This will be important when it comes to proving results about mutually singular marginals in Section \ref{subsec:singular_marginals}.

To make sense of what this means, we recall the brief discussion about eliminating shared coordinates after Example \ref{ex:Choquet_motivation}. Here, we consider the idea more generally. Let $x,u,y \in M$ be distinct and define $\zeta_1 = (x,u),\zeta_2 = (u,y) \in \wt{M}$. Obviously, these elements share a coordinate: $\pp_2(\zeta_1)=u=\pp_1(\zeta_2)$. Moreover, there exists a third element of $\wt{M}$, namely $\zeta=(x,y)$, satisfying $\pp(\zeta)=(\pp_1(\zeta_1),\pp_2(\zeta_2))$, $d(\zeta) \leq d(\zeta_1)+d(\zeta_2)$, and $d(\zeta)\delta_\zeta \preccurlyeq d(\zeta_1)\delta_{\zeta_1} + d(\zeta_2)\delta_{\zeta_2}$, because
\[
 \duality{g,d(\zeta)\delta_\zeta} = d(\zeta)g(\zeta) \leq d(\zeta_1)g(\zeta_1)+d(\zeta_2)g(\zeta_2) = \duality{g,d(\zeta_1)\delta_{\zeta_1} + d(\zeta_2)\delta_{\zeta_2}}, \quad g \in G. 
\]
The element $\zeta$ is the result of eliminating the shared ``middle'' coordinate from the pair $(\zeta_1,\zeta_2)$, and this elimination process is compatible with $\preccurlyeq$ in the sense that we obtain a new measure $d(\zeta)\delta_\zeta$ that is less than $d(\zeta_1)\delta_{\zeta_1} + d(\zeta_2)\delta_{\zeta_2}$, with respect to $\preccurlyeq$. It turns out that this elimination process can be generalised widely to elements $\zeta_1,\zeta_2 \in d^{-1}(0,\infty)$. 

\begin{lemma}\label{lm:shared_coordinate}
Let $\zeta_1,\zeta_2 \in \bwt{M}$ satisfy $\pp_2(\zeta_1)=\pp_1(\zeta_2)$ and $\zeta_1,\zeta_2 \in d^{-1}(0,\infty)$. Then either
\begin{equation}\label{lm:shared_coordinate_1}
0 \preccurlyeq d(\zeta_1)\delta_{\zeta_1} + d(\zeta_2)\delta_{\zeta_2} \quad\text{and}\quad \pp_1(\zeta_1) = \pp_2(\zeta_2),
\end{equation}
or there exists $\zeta \in \bwt{M}$ such that 
\begin{equation}\label{lm:shared_coordinate_2}
d(\zeta)\delta_\zeta \preccurlyeq d(\zeta_1)\delta_{\zeta_1} + d(\zeta_2)\delta_{\zeta_2},\;\; 0 < d(\zeta) \leq d(\zeta_1)+d(\zeta_2)\;\;\text{and}\;\; \pp(\zeta)=(\pp_1(\zeta_1),\pp_2(\zeta_2)).
\end{equation}
\end{lemma}

\begin{proof}
Let $I$ be the set of all triples $i=(U_1,U_2,n)$, where $U_1 \ni \zeta_1$, $U_2 \ni \zeta_2$ are open in $\bwt{M}$ and $n \in \NN$. We direct $I$ in the natural way by writing $(U_1,U_2,n) \leq (U_1',U_2',n')$ if and only if $U_1' \subset U_1$, $U_2' \subset U_2$ and $n \leq n'$. Given $i=(U_1,U_2,n) \in I$, define 
\[
A_i = \set{((x,u),(v,y)) \in (U_1 \cap \wt{M}) \times (U_2 \cap \wt{M}) \,:\, d(u,v) < \textstyle\frac{1}{n}}.
\]
First we claim that $A_i \neq \varnothing$ for all $i \in I$. The following argument has echoes of the one in the proof of Lemma \ref{lm:mixed_coords}. Suppose $i=(U_1,U_2,n) \in I$ satisfies $A_i=\varnothing$. Then $d(u,v) \geq \frac{1}{n}$ whenever $u \in \pp_2(U_1 \cap \wt{M})$ and $v \in \pp_1(U_2 \cap \wt{M})$. By Proposition \ref{pr:woodsuniform} \ref{pr:woodsuniform_2}, the closures of $\pp_2(U_1 \cap \wt{M})$ and $\pp_1(U_2 \cap \wt{M})$ in $\ucomp{M}$ are disjoint. Then, as $U_1$ is open and $\wt{M}$ is dense in $\bwt{M}$, we have $U_1 \subset \cl{U_1} = \cl{U_1 \cap \wt{M}}$. By continuity
\[
\pp_2(\zeta_1) \in \pp_2(U_1) \subset \pp_2\pare{\cl{U_1 \cap \wt{M}}} \subset \ucl{\pp_2(U_1 \cap \wt{M})},
\]
and likewise
\[
\pp_1(\zeta_2) \in \ucl{\pp_1(U_2 \cap \wt{M})},
\]
from which we deduce $\pp_2(\zeta_1) \neq \pp_1(\zeta_2)$. Since this is not so, we conclude that $A_i \neq \varnothing$ for all $i \in I$.

Next, we show that there exists $i=(U_1,U_2,n) \in I$ such that $(x,v),(u,y) \in \wt{M}$ whenever $((x,u),(v,y)) \in A_i$. Set $t=\min\set{d(\zeta_1),d(\zeta_2)}>0$ and let $n \in \NN$ such that $\frac{2}{n}< t$. By continuity, there exist open $U_1 \ni \zeta_1$, $U_2 \ni\zeta_2$ such that $|d(\zeta_1)-d(x,u)|,|d(\zeta_2)-d(v,y)| < \frac{1}{2}t$ whenever $(x,u) \in U_1 \cap \wt{M}$ and $(v,y) \in U_2 \cap \wt{M}$. Set $i=(U_1,U_2,n)$. Given $((x,u),(v,y)) \in A_i$, we see that
\[
d(x,v) \geq d(x,u)-d(u,v) > \textstyle d(\zeta_1) - \frac{1}{2}t - \frac{1}{n} > d(\zeta_1)-t \geq 0,
\]
and likewise $d(u,y) > 0$.

For $j \geq i$, consider
\[
B_j = \set{((x,u),(v,y)) \in A_j \,:\, (x,y) \in \wt{M}}
\]
Either $B_j \neq \varnothing$ for all $j \geq i$ or $B_j=\varnothing$ for some $j \geq i$. We examine both of these cases. Assume first that $B_j \neq \varnothing$ for all $j \geq i$. We consider two subcases by defining the sets
\[
C_j = \set{((x,u),(v,y)) \in B_j \,:\, (u,v) \in \wt{M}}.
\]
Again, either $C_j \neq \varnothing$ for all $j \geq i$ or $C_j=\varnothing$ for some $j \geq i$. If the first case holds, we can pick $((x_j,u_j),(v_j,y_j)) \in A_j$ such that $(x_j,y_j),(u_j,v_j) \in \wt{M}$ for all $j \geq i$. By compactness, and by taking a subnet if necessary, there exists $\zeta \in \bwt{M}$ such that $(x_j,y_j) \to \zeta$. Since $(x_j,u_j) \to \zeta_1$ and $(v_j,y_j) \to \zeta_2$, by continuity and the fact that $\lim_j d(u_j,v_j)=0$, we obtain
\begin{equation}\label{eqn:shared_coordinate_d}
d(\zeta) = \lim_j d(x_j,y_j) \leq \lim_j d(x_j,u_j) + d(u_j,v_j) + d(v_j,y_j) = d(\zeta_1) + d(\zeta_2).
\end{equation}
and
\begin{equation}\label{eqn:shared_coordinate_pp}
\pp(\zeta) = (\lim_j x_j, \lim_j y_j) = (\pp_1(\zeta_1),\pp_2(\zeta_2)).
\end{equation}
We claim that
\begin{equation}\label{eqn:shared_coordinate}
d(\zeta)\delta_\zeta \preccurlyeq d(\zeta_1)\delta_{\zeta_1} + d(\zeta_2)\delta_{\zeta_2}.
\end{equation}
Given $j \geq i$, we observe
\[
d(x_j,y_j)\delta_{(x_j,y_j)} \preccurlyeq d(x_j,u_j)\delta_{(x_j,u_j)} + d(u_j,v_j)\delta_{(u_j,v_j)} + d(v_j,y_j)\delta_{(v_j,y_j)}
\]
(By the above, we have ensured that these measures are well-defined, since the pairs $(x_j,y_j)$, $(x_j,u_j)$, $(u_j,v_j)$ and $(v_j,y_j)$ belong to $\wt{M}$ for all $j\geq i$.) Now \eqref{eqn:shared_coordinate} follows using Proposition \ref{pr:w_star}, continuity of $d$ and the fact that $\lim_j d(u_j,v_j)=0$. 

Now we consider the case where $C_j=\varnothing$ for some $j \geq i$. Then $C_k=\varnothing$ whenever $k \geq j$; as $B_k \neq \varnothing$ for $k \geq j$, we can pick $((x_k,u_k),(u_k,y_k)) \in A_k$ such that $(x_k,y_k) \in \wt{M}$ for all $k \geq j$. We find $\zeta$ and a subnet as above. Evidently \eqref{eqn:shared_coordinate_d} and \eqref{eqn:shared_coordinate_pp} follow as above. This time, given $k \geq j$, we observe
\[
d(x_k,y_k)\delta_{(x_k,y_k)} \preccurlyeq d(x_k,u_k)\delta_{(x_k,u_k)} + d(u_k,y_k)\delta_{(u_k,y_k)},
\]
whence \eqref{eqn:shared_coordinate} follows again by Proposition \ref{pr:w_star} and continuity of $d$.

In summary, we have shown that \eqref{eqn:shared_coordinate_d} -- \eqref{eqn:shared_coordinate} hold if $B_j \neq \varnothing$ for all $j \geq i$. Finally, if $d(\zeta)>0$ then we obtain \eqref{lm:shared_coordinate_2}. If $d(\zeta)=0$, then $\pp_1(\zeta_1) = \pp_1(\zeta) = \pp_2(\zeta) = \pp_2(\zeta_2)$ by \eqref{eqn:shared_coordinate_pp}, so we obtain \eqref{lm:shared_coordinate_1}.

Now suppose that $B_j=\varnothing$ for some $j \geq i$. In this case we claim that \eqref{lm:shared_coordinate_1} holds. We have $B_k=\varnothing$ for all $k \geq j$. For $k \geq j$, we consider $C_k$ defined as above, and consider two cases. First assume $C_k \neq \varnothing$ for all $k \geq j$.  On this occasion, for $k \geq j$ we can pick $((x_k,u_k),(v_k,x_k)) \in A_k$ such that $(u_k,v_k) \in \wt{M}$, thus
\[
0 \preccurlyeq d(x_k,u_k)\delta_{(x_k,u_k)} + d(u_k,v_k)\delta_{(u_k,v_k)} + d(v_k,x_k)\delta_{(v_k,x_k)},
\]
by Proposition \ref{pr:scalar_mult} \ref{pr:scalar_mult_2}. Hence again by Proposition \ref{pr:w_star}, continuity of $d$ and as $\lim_k d(u_k,v_k)=0$,
\begin{equation}\label{eqn:shared_coordinate_1}
0 \preccurlyeq d(\zeta_1)\delta_{\zeta_1} + d(\zeta_2)\delta_{\zeta_2}.
\end{equation}
Furthermore, by continuity $\pp_1(\zeta_1)=\lim_k x_k = \pp_2(\zeta_2)$, giving \eqref{lm:shared_coordinate_1} as claimed. Second, assume that $C_k=\varnothing$ for some $k \geq j$. Then $C_\ell=\varnothing$ for all $\ell \geq k$. Given such $\ell$, pick $((x_\ell,u_\ell),(u_\ell,x_\ell)) \in A_\ell$. Again by Proposition \ref{pr:scalar_mult} \ref{pr:scalar_mult_2},
\[
0 \preccurlyeq d(x_\ell,u_\ell)\delta_{(x_\ell,u_\ell)} + d(u_\ell,x_\ell)\delta_{(u_\ell,x_\ell)},
\]
whence \eqref{eqn:shared_coordinate_1} holds as before. Because $\pp_1(\zeta_1)=\lim_\ell x_\ell = \pp_2(\zeta_2)$, again we have \eqref{lm:shared_coordinate_1}. This completes the proof.
\end{proof}

We require one more lemma before giving the final result of the section. 

\begin{lemma}\label{lm:approx_g_functions}
Let $\mu \in \meas{\bwt{M}}^+$ be minimal, let $E \subset \bwt{M}$ be Borel and let $f:\bwt{M}\to\RR$ satisfy $f(\zeta)=\lim_n f_n(\zeta)$ for all $\zeta \in E$, where the functions $f_n:\bwt{M}\to\RR$, $n \in \NN$, are upper semicontinuous, constant on all fibres of $\gq$ and $\norm{f_n}_{L_1(\mu)}<\infty$. Then there exists a sequence $(g_n)$ of elements of $G$ and a Borel subset $F \subset E$ that satisfy $\mu(E\setminus F)=0$, $f(\zeta) \leq \liminf_n g_n(\zeta)$ for all $\zeta \in E$ and $f(\zeta) = \lim_n g_n(\zeta)$ for all $\zeta \in F$.
\end{lemma}

We remark that the sequence $(g_n)$ need not be uniformly bounded in norm in general, so $\liminf_n g_n(\zeta)=\infty$ for all $\zeta$ in a possibly non-empty (though $\mu$-null) subset of $E$.

\begin{proof}
Given $n \in \NN$, by Theorem \ref{th:minimality_test} \ref{pr:minimality_test_3}, as $\norm{f_n}_{L_1(\mu)}<\infty$ there exists $g_n \in G$, $f_n \leq g_n$, such that 
\[
 \int_{\bwt{M}} g_n - f_n \,d\mu < 4^{-n}.
\]
First, the fact that $f_n \leq g_n$ for all $n$ and $f_n \to f$ pointwise on $E$ implies $f(\zeta) \leq \liminf_n g_n(\zeta)$ for all $\zeta \in E$.
Second, the positive function $g_n-f_n$ is lower semicontinuous, so $A_n:=(g_n-f_n)^{-1}[0,2^{-n}]$ is closed and $U_n := \bwt{M}\setminus A_n = (g_n-f_n)^{-1}(2^{-n},\infty)$ is open. Let $A=\liminf_n A_n$. Since $2^{-n}\cdot\mathbf{1}_{U_n} \leq g_n-f_n$, we obtain
\[
 2^{-n}\mu(U_n) \leq \int_{\bwt{M}} g_n - f_n \,d\mu < 4^{-n},
\]
whence $\mu(U_n) < 2^{-n}$. It follows that
\[
 \mu\pare{\bigcup_{p=n} U_p} \leq \sum_{p=n}^\infty \mu(U_p) < 2^{1-n},
\]
giving $\mu(\limsup_n U_n)=0$ and $\mu(A)=\norm{\mu}$. By construction $\mu(E\setminus F)=0$, where $F:=A \cap E$, and $\lim_n g_n(\zeta)=f(\zeta)$ for all $\zeta \in F$.
\end{proof}

Using the two lemmas above, we can present our promised result about the sets on which certain minimal elements are concentrated.

\begin{corollary}\label{co:minimality_coordinates}
Let $\mu \in \meas{\bwt{M}}^+$ be minimal and concentrated on $d^{-1}(0,\infty)$. Then $\mu$ is concentrated on a Borel set $A\subset d^{-1}(0,\infty)$, such that if $\zeta_1,\zeta_2 \in A$ then $\pp_2(\zeta_1) \neq \pp_1(\zeta_2)$.
\end{corollary}

\begin{proof}
Consider the maps $f:\bwt{M}\to\RR$ given by $f(\zeta)=-1/d(\zeta)$, $d(\zeta)>0$ and $f(\zeta)=0$ otherwise, and $f_n \in C(\bwt{M})$ given by $f_n(\zeta)=-\min\set{n,\frac{1}{d}}$. Then $-f_n \in G$ by Corollary \ref{co:overd} (so the $f_n$ are constant on fibres of $\gq$) and $f_n \to f$ pointwise on $d^{-1}(0,\infty)$. By Lemma \ref{lm:approx_g_functions}, there exists a sequence $(g_n)$ of elements of $G$ that satisfies $f(\zeta) \leq \liminf_n g_n(\zeta)$ for all $\zeta \in d^{-1}(0,\infty)$ and converges to $f$ everywhere on a Borel set $A \subset d^{-1}(0,\infty)$, where $\mu(d^{-1}(0,\infty)\setminus A)=0$. Let $\zeta_1,\zeta_2 \in A$ and suppose for a contradiction that $\pp_2(\zeta_1) = \pp_1(\zeta_2)$. By Lemma \ref{lm:shared_coordinate}, either \eqref{lm:shared_coordinate_1} holds, or there exists $\zeta \in \bwt{M}$ such that \eqref{lm:shared_coordinate_2} holds. Both cases are impossible: if \eqref{lm:shared_coordinate_1} holds then
\[
 0 \leq d(\zeta_1)g_n(\zeta_1) + d(\zeta_2)g_n(\zeta_2) \to -2,
\]
and if \eqref{lm:shared_coordinate_2} holds then $0 < d(\zeta) \leq d(\zeta_1)+d(\zeta_2) < \infty$ and
\[
 d(\zeta)g_n(\zeta) \leq d(\zeta_1)g_n(\zeta_1) + d(\zeta_2)g_n(\zeta_2),
\]
and taking the limit inferior of both sides yields $-1 \leq -2$.
\end{proof}

\subsection{Mutually singular marginals}\label{subsec:singular_marginals}

Let $\mu \in \meas{\bwt{M}}^+$. We define the \emph{marginals} of $\mu$ to be the pushforward measures $(\pp_i)_\sharp\mu$ of $\mu$ under the coordinate maps $\pp_i$, $i=1,2$. These measures belong to $\meas{\ucomp{M}}$ in general, though they are concentrated on $\rcomp{M}$ or $M$ if $\mu$ is concentrated on $\bwtf$ or $\pp^{-1}(M \times M)$, respectively. We will say that $\mu$ has \emph{mutually singular marginals} if $(\pp_1)_\sharp\mu \wedge (\pp_2)_\sharp\mu=0$, that is,
\[
 (\pp_1)_\sharp\mu \perp (\pp_2)_\sharp\mu.
\]
Let $\lambda$ be a Borel measure on $\ucomp{M}$. Assuming $\lambda$ is sufficiently regular, it is possible to define an element $\widehat{\lambda} \in \Lip_0(M)$ via integration:
\[
 \duality{f,\widehat{\lambda}} = \int_{\ucomp{M}} \ucomp{f}\,d\lambda.
\]
Assuming further regularity of $\lambda$, we find that $\widehat{\lambda} \in \lipfree{M}$. Such ``measure-induced functionals'' are studied intensively in \cite{AP_measures} (where $\widehat{\lambda}$ is denoted $\mathcal{L}\lambda$); we mention in particular \cite[Propositions 4.3 and 4.4 and Theorem 4.11]{AP_measures}, which state precisely the regularity conditions alluded to above.

It happens that there are strong and natural links between measure-induced functionals and optimal De Leeuw representations satisfying certain finiteness conditions. These links were obtained by taking advantage of the idea of optimal couplings and the famous Kantorovich-Rubinstein duality theorem from optimal transport theory. The fact that every element of $\lipfree{M}$ is a convex integral of molecules whenever $M$ is uniformly discrete is one of the applications of this endeavour. We refer the reader to \cite[Section 3]{APS24a} and \cite[Section 4.4]{APS24b} for more details.

In particular, it was found that the representations $\mu \in \meas{\bwt{M}}$ of $\widehat{\lambda}$ obtained in e.g.~\cite[Theorem 3.1]{APS24a} and \cite[Theorems 4.19, 4.20 and Corollary 4.21]{APS24b} can be chosen to have mutually singular marginals. This is essentially due to the Jordan decomposition $\lambda = \lambda^+-\lambda^-$ and the fact that the representations $\mu$ can be chosen to satisfy $(\pp_1)_\sharp\mu\ll\lambda^+$ and $(\pp_2)_\sharp\mu\ll\lambda^-$ (or failing that a slightly weaker condition).

Having representations with mutually singular marginals is a desirable property because it helps to understand and simplify calculations involving the corresponding free space or bidual elements, and it is of interest to see whether the marginals (or parts thereof) of representations can be made to be mutually singular in general, for free space or bidual elements that cannot be measure-induced (or, more generally, \emph{majorisable}; see above Theorem \ref{th:bidual_majorisable} for the definition of this term). This is the aim of the section. As well as being a desirable property in its own right, having (partially) mutually singular marginals is an essential component of the proof of Theorem \ref{th:opt_conc}, which in turn is needed to prove Theorem \ref{th:extreme}. Once again, it is minimality that gives us the desired results.

The main results of this section are Theorems \ref{th:mutually_singular_upgrade} and \ref{th:singular_marginals}. We start with some preliminary work. Given a set $L$, we let
\[
 \wt{L} = \set{(x,y) \in L \times L \,:\, x \neq y}.
\]
While the next lemma is stated in terms of locally compact spaces, by inspection of the proof it can be seen to apply more widely.

\begin{lemma}\label{lm:rectangle_marginals_2}
Let $L$ be locally compact, let $\lambda \in \meas{L \times L}^+$ and let $\ep>0$. Then there exists a finite partition $\mathscr{P}$ of $L$ into Borel sets, such that
\begin{equation}\label{eqn:rectangle_marginals_2}
 \sum_{F \in \mathscr{P}} \lambda(\wt{F}) < \ep.
\end{equation}
\end{lemma}

\begin{proof}
By inner regularity, there exists a compact set $K \subset \wt{L}$ such that $\lambda(\wt{L}\setminus K) < \ep$. By compactness, there exist open sets $U_1,\ldots,U_n,V_1,\ldots,V_n \subset L$ such that $U_i \cap V_i = \varnothing$, $i=1,\ldots,n$, and $K \subset \bigcup_{i=1}^n U_i \times V_i \subset \wt{L}$. Next, given $I \subset \set{1,\ldots,n}$, define the Borel sets
 \[
  A_I = \begin{cases}
         \displaystyle \bigcap_{i \in I} U_i \setminus \bigcup_{i \notin I} U_i & \text{if }I \neq \varnothing, \\ \displaystyle  L\setminus\bigcup_{i=1}^n U_i & \text{if }I = \varnothing
        \end{cases}
        \qquad\text{and}\qquad
  B_I = \begin{cases}
         \displaystyle \bigcap_{i \in I} V_i \setminus \bigcup_{i \notin I} V_i & \text{if }I \neq \varnothing, \\ \displaystyle L\setminus\bigcup_{i=1}^n V_i & \text{if }I = \varnothing.
        \end{cases}      
 \]
Let $\Gamma$ be the set of all pairs $(I,J)$, where $I,J \subset \set{1,\ldots,n}$ and $F_{(I,J)}:=A_I \cap B_J$ is non-empty. Then $\set{F_{(I,J)} \,:\, (I,J) \in \Gamma}$ is a finite partition of $L$ consisting of Borel sets. We claim that $(F_{(I,J)} \times F_{(I,J)}) \cap K = \varnothing$ whenever $(I,J) \in \Gamma$, which will establish \eqref{eqn:rectangle_marginals_2}. Suppose that, for some such $(I,J)$, there exists $(x,y) \in (F_{(I,J)} \times F_{(I,J)}) \cap K$. Then $(x,y) \in U_i \times V_i$ for some $i \leq n$. This implies $x \in A_I \cap U_i$ so $i \in I$ and $A_I \subset U_i$; likewise $B_J \subset V_i$. However, this forces $A_I \cap B_J \subset U_i \cap V_i = \varnothing$, which isn't the case. This completes the proof.
\end{proof}

At this point we need to introduce the set $\bwtfp$ defined originally in \cite[(4.6)]{APS24b}:
\[
\bwtfp = \set{\zeta\in\bwtf \,:\, \pp_1(\zeta)\neq\pp_2(\zeta)}.
\]
Equivalently, $\zeta \in \bwtfp$ if $\zeta \in \bwtf$ and $\bar{d}(\pp(\zeta))>0$. It will be helpful to compare $\bwtfp$ and $\bwtf\setminus d^{-1}(0)$. To do this, recall Proposition \ref{pr:De_Leeuw_relation} and the brief discussion on optimal points and vanishing points at the end of Section \ref{subsec:coords}. We note that $\bwtfp \subset \bwtf\setminus d^{-1}(0)$ by Proposition \ref{pr:De_Leeuw_relation} \ref{pr:De_Leeuw_relation_3}, and the inclusion is strict in general because of the existence of points $\zeta$ satisfying $d(\zeta) > \bar{d}(\pp(\zeta))=0$.

Despite the difference between $\bwtfp$ and $\bwtf\setminus d^{-1}(0)$, with respect to optimal representations we can regard them as being equal, because for $\mu \in \opr{\bwt{M}}$ we have $\mu\restrict_{\bwtfp} = \mu\restrict_{\bwtf\setminus d^{-1}(0)}$. Indeed, $\bwtfp \subset \bwtf\setminus d^{-1}(0)$ implies $\mu\restrict_{\bwtfp} \leq \mu\restrict_{\bwtf\setminus d^{-1}(0)}$. Conversely, if $\mu$ is optimal then so is $\mu\restrict_{\bwtf\setminus d^{-1}(0)}$ by \cite[Proposition 2.1 (d)]{APS24a}, so $\mu\restrict_{\bwtf\setminus d^{-1}(0)}$ is concentrated on $(\bwtf \cap \opt)\setminus d^{-1}(0)$, which is a subset of $\bwtfp$ again by Proposition \ref{pr:De_Leeuw_relation} \ref{pr:De_Leeuw_relation_3}.

\begin{theorem}\label{th:mutually_singular_upgrade}
Let $\mu \in \meas{\bwt{M}}^+$ be minimal and concentrated on $\bwtfp$. Then $(\pp_1)_\sharp\mu \perp (\pp_2)_\sharp\mu$.
\end{theorem}

\begin{proof}
Let $\mu \in \meas{\bwt{M}}^+$ be minimal and concentrated on $\bwtfp$. Then $\pp_\sharp \mu \in \meas{\rcomp{M} \times \rcomp{M}}$ is concentrated on $\wt{\rcomp{M}}$. Fix $\ep>0$. Using Lemma \ref{lm:rectangle_marginals_2}, we fix a partition $\set{F_1,\ldots,F_n}$ of Borel subsets of $\rcomp{M}$ that satisfies 
\begin{equation}\label{eqn:mutually_singular_upgrade_1}
 \sum_{k=1}^n \pp_\sharp\mu(F_k \times F_k) < \ep.
\end{equation}
Next, we make the following claim.

\begin{claim}
Let $E \subset \rcomp{M}$ be Borel. Then there exist disjoint Borel sets $P,Q \subset E$ such that 
\begin{equation}\label{eqn:mutually_singular_upgrade_2}
 (\pp_2)_\sharp\mu(E\setminus P),\,(\pp_1)_\sharp\mu(E\setminus Q) \leq \pp_\sharp\mu(E \times E).
\end{equation}
\end{claim}

Indeed, using Corollary \ref{co:minimality_coordinates}, $\mu$ is concentrated on a set $A\subset \bwtfp$, such that if $\zeta_1,\zeta_2 \in A$ then $\pp_2(\zeta_1) \neq \pp_1(\zeta_2)$. Let $E \subset \rcomp{M}$ be Borel. Using inner regularity, pick $K_\sigma$ subsets $P' \subset A \cap \pp^{-1}((\rcomp{M}\setminus E)  \times E), Q' \subset A \cap \pp^{-1}(E \times (\rcomp{M}\setminus E))$ satisfying
\[
 \mu(\pp^{-1}((\rcomp{M}\setminus E) \times E)\setminus P') = \mu(\pp^{-1}(E \times (\rcomp{M}\setminus E))\setminus Q') = 0.
\]
Set $P=\pp_2(P'), Q=\pp_1(Q') \subset E$, which are again $K_\sigma$ sets. We observe that $P \cap Q = \varnothing$. Indeed, given $\xi \in P \cap Q$, there would exist $\zeta_1 \in P'$, $\zeta_2 \in Q'$ such that $\pp_2(\zeta_1)=\xi=\pp_1(\zeta_2)$, but as $P',Q' \subset A$, this isn't possible. Finally, we estimate
\begin{align*}
 (\pp_2)_\sharp\mu(E\setminus P) = \mu(\pp_2^{-1}(E)\setminus\pp_2^{-1}(\pp_2(P'))) &\leq \mu(\pp_2^{-1}(E)\setminus P')\\
 &= \mu(\pp^{-1}(E \times E)) + \mu(\pp^{-1}((\rcomp{M}\setminus E)  \times E)\setminus P')\\
 &= \mu(\pp^{-1}(E \times E)) = \pp_\sharp\mu(E \times E).
\end{align*}
Likewise we get the same upper estimate for $(\pp_1)_\sharp\mu(E\setminus Q)$, which establishes \eqref{eqn:mutually_singular_upgrade_2} and completes the proof of the claim.

To finish, define $\omega=(\pp_1)_\sharp\mu \wedge (\pp_2)_\sharp\mu$. Fix a Borel set $E \subset \rcomp{M}$ and write $E_k=E \cap F_k$, $k=1,\ldots,n$. We apply the claim to each $E_k$ to obtain disjoint Borel sets $P_k,Q_k \subset E_k$ satisfying 
\[
(\pp_2)_\sharp\mu(E_k\setminus P_k),\,(\pp_1)_\sharp\mu(E_k\setminus Q_k) \leq \pp_\sharp\mu(E_k \times E_k), \quad k=1,\ldots,n.
\]
Define the disjoint Borel sets $P=\bigcup_{k=1}^n P_k, Q = \bigcup_{k=1}^n Q_k \subset E$. We estimate
\begin{align*}
 \omega(E) = \omega(E\setminus P) + \omega(P) &\leq (\pp_2)_\sharp\mu(E\setminus P) + (\pp_1)_\sharp\mu(E\setminus Q)\\
 &= \sum_{k=1}^n (\pp_2)_\sharp\mu(E_k\setminus P_k) + (\pp_1)_\sharp\mu(E_k\setminus Q_k)\\
 &\leq 2\sum_{k=1}^n \pp_\sharp\mu(E_k \times E_k) < 2\ep.
\end{align*}
Since $\ep>0$ and $E \subset \rcomp{M}$ were arbitrary, we obtain $\omega=0$.
\end{proof}

It will be helpful to have to hand a slight generalisation of the above result.

\begin{corollary}\label{co:mutually_singular_upgrade}
Let $A \subset \rcomp{M}$ be Borel, and let $\mu \in \meas{\bwt{M}}^+$ be minimal and concentrated on $\bwtf\setminus\set{\zeta \in \bwtf \,:\, \pp_1(\zeta)=\pp_2(\zeta) \in A}$. Then the restrictions of the marginals of $\mu$ to $A$ are mutually singular:
\[
((\pp_1)_\sharp\mu)\restrict_A \perp ((\pp_2)_\sharp\mu)\restrict_A.
\]
\end{corollary}

\begin{proof}
Define $B=\rcomp{M}\setminus A$ and $\omega=\mu - \mu\restrict_{\pp^{-1}(B \times B)}$, which is minimal by Proposition \ref{pr:min_basic} \ref{pr:min_basic_3}. We observe that $\omega$ is concentrated on $\bwtfp$. By Theorem \ref{th:mutually_singular_upgrade}, $\omega$ has mutually singular marginals. Given a Borel set $E \subset \rcomp{M}$, 
\begin{align*}
 ((\pp_1)_\sharp\mu)\restrict_A(E) = \mu(\pp_1^{-1}(A \cap E)) &= \omega(\pp_1^{-1}(A \cap E)) + \mu\restrict_{\pp^{-1}(B \times B)}(\pp_1^{-1}(A \cap E))\\
 &= \omega(\pp_1^{-1}(A \cap E)) = ((\pp_1)_\sharp\omega)\restrict_A(E),
\end{align*}
whence $((\pp_1)_\sharp\mu)\restrict_A=((\pp_1)_\sharp\omega)\restrict_A$. Likewise $((\pp_2)_\sharp\mu)\restrict_A=((\pp_2)_\sharp\omega)\restrict_A$. The result follows.
\end{proof}

We can combine several previous results to obtain the following theorem, which is our general statement on (partially) mutually singular marginals. 

\begin{theorem}\label{th:singular_marginals}
Let $\nu \in \opr{\bwt{M}}$. Then there exists $\mu \in \opr{\bwt{M}}$ such that 
\begin{enumerate}[label={\upshape{(\roman*)}}]
\item\label{th:singular_marginals_1} $\Phi^*\mu=\Phi^*\nu$;
\item\label{th:singular_marginals_2} if $\nu$ is concentrated on an element of $\mathfrak{C}$ (including any set in Corollary \ref{co:wt} or Proposition \ref{pr:delta}), then so is $\mu$;
\item\label{th:singular_marginals_3} if $\nu$ is concentrated on $\bwtf\setminus\Delta(A)$, where $A \subset \rcomp{M}$ is Borel, then 
\[
((\pp_1)_\sharp\mu)\restrict_A \perp ((\pp_2)_\sharp\mu)\restrict_A.
\]
\end{enumerate}
\end{theorem}

\begin{proof} As $\nu$ is optimal it is positive. Using Proposition \ref{pr:minimal_exist}, there exists minimal $\mu \in \meas{\bwt{M}}^+$ such that $\mu \preccurlyeq \nu$. By Proposition \ref{pr:basic} \ref{pr:basic_4}, $\mu$ is optimal.
\begin{enumerate}[label={\upshape{(\roman*)}}]
\item This follows by Proposition \ref{pr:basic} \ref{pr:basic_3}.
\item This is immediate by the definition of $\mathfrak{C}$.
\item By \ref{th:singular_marginals_2}, $\mu$ is concentrated on $\bwtf\setminus\Delta(A)$. Similarly to the remarks made just above Theorem \ref{th:mutually_singular_upgrade}, as $\mu$ is optimal, it is concentrated moreover on $\bwtf\setminus\set{\zeta \in \bwtf \,:\, \pp_1(\zeta)=\pp_2(\zeta) \in A}$. Hence we can apply Corollary \ref{co:mutually_singular_upgrade}. \qedhere
\end{enumerate}
\end{proof}

We finish the section with some applications of the previous results. We can combine Theorems \ref{th:minimal_shadow} and \ref{th:singular_marginals} to obtain ``ideal'' representations of convex integrals of molecules.

\begin{corollary}\label{co:nice_cim}
 Let $m \in \lipfree{M}$ be a convex integral of molecules. Then $m$ has a representation $\mu \in \opr{\bwt{M}}$ that has mutually singular marginals and is concentrated on $\wt{A}$, where $A:=\supp(m) \cup \set{0}$.
\end{corollary}

\begin{proof}
 Let $\nu \in \opr{\bwt{M}}$ represent $m$ and be concentrated on $\wt{M}$. By Theorem \ref{th:singular_marginals}, there exists $\mu \in \opr{\bwt{M}}$ also representing $m$, concentrated on $\wt{M}$ and having mutually singular marginals. By Theorem \ref{th:minimal_shadow}, $\supp(\mu) \subset \pp^{-1}(K \times K)$, where $K:=\esupp{m} \cup \set{0}$. By \cite[Corollary 3.7]{AP_measures}, $A=K \cap M$. From these facts it follows easily that $\mu$ is concentrated on $\supp(\mu) \cap \wt{M} \subset \wt{A}$.
\end{proof}

Recall that for some metric spaces $M$, every element of $\lipfree{M}$ is a convex integral of molecules. For example, this is true if $M$ is uniformly discrete \cite[Corollary 3.7]{APS24a} or if $M$ is proper and purely 1-unrectifiable \cite[Corollary 6.6]{APS24b}. It turns out that this is a hereditary property, as was observed by R. Aliaga.

\begin{corollary}[R. Aliaga]\label{co:hereditary_cim}Let $M$ have the property that every element of $\lipfree{M}$ is a convex integral of molecules, and let $A \subset M$ be closed and contain $0$. Then every element of $\lipfree{A}$ is a convex integral of molecules (with respect to $\wt{A}$).
\end{corollary}

\begin{proof}
Let $m \in \lipfree{A}$. Then $m$ can be regarded as an element of $\lipfree{M}$, and so it is a convex integral of molecules (with respect to $\wt{M}$). By Corollary \ref{co:nice_cim}, $m$ has a representation $\mu \in \opr{\bwt{M}}$ that is concentrated on $\wt{A}$.
\end{proof}

Finally, we can use Lemma \ref{lm:G_distance_functions} and Theorem \ref{th:singular_marginals} to improve \cite[Theorem 4.22]{APS24b}. We state the following ``finiteness condition'' for a measure $\mu \in \meas{\bwt{M}}^+$ \cite[(4.10)]{APS24b}:
\begin{equation}\label{eq:bwt_marginals_finite_first_moment}
\int_{\bwtfp}\frac{\bar{d}(\pp_1(\zeta),0)}{d(\zeta)}\,d\mu(\zeta) < \infty.
\end{equation}

\begin{proposition}\label{pr:imp_4.20}
Let $\psi \in \Lip_0(M)^*$ have an optimal representation $\nu$ concentrated on $\bwtfp$ and satisfying \eqref{eq:bwt_marginals_finite_first_moment}. Then $\psi$ has an optimal representation $\mu$ concentrated on $\bwtfp$, having mutually singular marginals, and satisfying \eqref{eq:bwt_marginals_finite_first_moment}.
\end{proposition}

\begin{proof}
By Theorem \ref{th:singular_marginals}, $\psi$ has an optimal representation $\mu \preccurlyeq \nu$ concentrated on $\bwtfp$ and having mutually singular marginals. We just need to check that $\mu$ satisfies \eqref{eq:bwt_marginals_finite_first_moment}. Since $\bar{d}(\cdot,0)=\rcomp{d}(\cdot,0)$ by Proposition \ref{pr:bar-d_cont} \ref{pr:bar-d_cont_2}, according to Lemma \ref{lm:G_distance_functions} there exists an increasing sequence of positive functions $g_n \in G \cap nB_{C(\bwt{M})}$ satisfying
\[
 g_n(\zeta) = \min\set{n,\frac{\bar{d}(\pp_1(\zeta),0)}{d(\zeta)}}, \quad \zeta \in \pp_1^{-1}(\rcomp{M})\setminus (\pp_1^{-1}(0) \cap d^{-1}(0)).
\]
Therefore, by the monotone convergence theorem, there exists $g \in \mathfrak{G}$ such that
\[
g(\zeta) = \frac{\bar{d}(\pp_1(\zeta),0)}{d(\zeta)}, \quad \zeta \in \pp_1^{-1}(\rcomp{M})\setminus (\pp_1^{-1}(0) \cap d^{-1}(0)).
\]
Since $\mu$ and $\nu$ are concentrated on $\bwtfp$, the result follows.
\end{proof}

The next result is an immediate consequence of Proposition \ref{pr:imp_4.20} and \cite[Theorem 4.22]{APS24b}. We say that $\psi \in \Lip_0(M)^*$ (or $\lipfree{M}$) is \emph{majorisable} if we can write $\psi=\psi^+ - \psi^-$, where $\psi^\pm \in \Lip_0(M)^*$ (or $\lipfree{M}$) are \emph{positive functionals}, that is, $\duality{f,\psi^\pm} \geq 0$ whenever $f \in \Lip_0(M)$ is positive in the usual sense. We highlight that, in general, not every element of $\Lip_0(M)^*$ or $\lipfree{M}$ can be expressed in this way; see \cite[Example 3.4]{APS24a}, \cite[Example 4.17]{AP_measures} and \cite[Example 3.24]{Weaver2}.

\begin{theorem}[cf. {\cite[Theorem 4.22]{APS24b}}]\label{th:bidual_majorisable}
Let $\psi\in\Lip_0(M)^*$. Then the following are equivalent:
\begin{enumerate}[label={\upshape{(\roman*)}}]
\item\label{th:bidual_majorisable-1} $\psi$ has an optimal representation $\mu$ concentrated on $\bwtfp$ satisfying \eqref{eq:bwt_marginals_finite_first_moment};
\item\label{th:bidual_majorisable-2} $\psi$ is majorisable and avoids 0 and infinity.
\end{enumerate}
If the above hold, then $\mu$ can be chosen to satisfy $(\pp_1)_\sharp\mu \perp(\pp_2)_\sharp\mu$.
\end{theorem}

\section{The main theorems and applications to Lipschitz-free spaces}\label{sec:main_theorems}

\subsection{Optimal and minimal representations of free space elements}\label{subsec:opt_conc}

We recall from Proposition \ref{co:free_bwtf} that all optimal representations of free space elements must be concentrated on $\bwtf=\pp^{-1}(\rcomp{M}\times\rcomp{M})$, though by Example \ref{ex:bad_rep}, they are not obliged to be concentrated on $\pp^{-1}(M \times M)$. The point of this section is to prove the next two results; the first of these was prompted by a question of R.~Aliaga.

\begin{theorem}\label{th:opt_conc} Let $m \in \lipfree{M}$. Then there exists $\mu \in \opr{\bwt{M}}$ such that $m=\Phi^*\mu$ and $\mu$ is concentrated on $\pp^{-1}(M\times M)$.
\end{theorem}

By Proposition \ref{pr:basic} \ref{pr:basic_3} and \ref{pr:basic_4} and Proposition \ref{pr:minimal_exist}, Theorem \ref{th:opt_conc} follows immediately from the next result.

\begin{theorem}\label{th:min_conc} Let $\mu \in \opr{\bwt{M}}$ be minimal and satisfy $\Phi^*\mu \in \lipfree{M}$. Then $\mu$ is concentrated on $\pp^{-1}(M\times M)$.
\end{theorem}

In order to do this, we first recall the cone of lower semicontinuous functions $\Glsc$ defined in Section \ref{subsec:conc}, and present a development of Proposition \ref{pr:G-support}. 

\begin{proposition}\label{pr:Glsc_lipfree}
 Suppose $g \in \Glsc$ satisfies $g\restrict_{\wt{M}} \geq 0$, and let $\mu \in \meas{\bwt{M}}^+$ such that $\Phi^*\mu \in \lipfree{M}$. Then $\int_{\bwt{M}} g\,d\mu \geq 0$.
\end{proposition}

\begin{proof}
Pick a non-empty and bounded set $E \subset G$ such that $g=\bigvee E$. Recall that we can assume $E$ is max-stable. Let $L>0$ such that $\norm{g'} \leq L$ for all $g' \in E$. Set $m=\Phi^*\mu$ and let $\ep>0$. Pick $m' \in \lipfree{M}$ such that $\norm{m - m'} < \frac{\ep}{L}$ and $A:=\supp(m') \cup \set{0}$ is finite. Let $m'=\sum_{k=1}^n a_km_{x_ky_k}$ for some scalars $a_k>0$ and points $(x_k,y_k) \in \wt{A}$, $k=1,\ldots,n$. Set
\[
\lambda = \left(\sum_{k=1}^n \frac{a_k}{d(x_k,y_k)} \right)^{-1}\cdot(\ep-L\norm{m-m'}) > 0.
\]
As $g\restrict_{\wt{M}} \geq 0$ and $E$ is max-stable, we can find $g' \in E$ such that $g'(x,y) > -\lambda/d(x,y)$ for all $(x,y) \in \wt{A}$. Let $h'$ be the map associated with $g'$ from Definition \ref{df:assoc_map}. As in the proof of Proposition \ref{pr:G-support}, we can define $f \in \Lip_0(M)$ by
\[
 f(x) = \min_{a \in A} h'(x,a), \quad x \in M.
\]
Following the argument in the proof of that result, $\Phi f \leq g' \leq g$ and $0 = h'(x,x) \geq f(x) > -\lambda$ whenever $x \in A$. Therefore
\begin{align*}
 \int_{\bwt{M}} g \,d\mu \geq \duality{\Phi f,\mu} = \duality{f,m} &= \duality{f,m'} + \duality{f,m-m'} \\
 &\geq \sum_{k=1}^n a_k\frac{f(x_k)-f(y_k)}{d(x_k,y_k)} - L\norm{m-m'}\\
 &> -\lambda\left(\sum_{k=1}^n \frac{a_k}{d(x_k,y_k)}\right) - L\norm{m-m'} = -\ep.
\end{align*}
As this holds for all $\ep>0$, we are done.
\end{proof}

We define the set $\Gamma=\pp^{-1}((\rcomp{M}\setminus M) \times (\rcomp{M}\setminus M))$ of all points in $\bwtf$ that have neither coordinate in $M$. The next step in the proof of Theorem \ref{th:min_conc} is to show that minimal measures that are concentrated on $\bwtf$ and represent free space elements cannot have any content on $\Gamma$.

\begin{theorem}\label{th:kill_Gamma} Let $\mu \in \meas{\bwt{M}}^+$ be minimal and concentrated on $\bwtf$, and suppose that $\Phi^*\mu \in \lipfree{M}$. Then $\mu(\Gamma)=0$.
\end{theorem}

\begin{proof}
Let $\mu \in \meas{\bwt{M}}^+$ be minimal and concentrated on $\bwtf$, and let $\Phi^*\mu \in \lipfree{M}$. Given $n \in \NN$, define
\[
\Gamma_n = \set{\zeta \in \bwtf \,:\, \bar{d}(\pp_1(\zeta),M),\bar{d}(\pp_2(\zeta),M) \geq \tfrac{1}{n}}.
\]
Since $\bar{d}(\cdot,M)$ is upper semicontinuous (see after Proposition \ref{pr:bar-d_cont}), the sets $\Gamma_n$ are closed in $\bwtf$. Evidently $\Gamma=\bigcup_{n=1}^\infty \Gamma_n$, so it is sufficient to prove that $\mu(\Gamma_n)=0$ for all $n \in \NN$. Moreover, by the inner regularity of $\mu$, it is sufficient to prove that $\mu(H)=0$ whenever $H \subset \bwtf$ is compact and $H \subset \Gamma_n$ for some $n \in \NN$. Given such $H$, we can set $K=\pp_s(H)$, which is a compact subset of $\rcomp{M}$ satisfying $H \subset \pp^{-1}(K \times K) \subset \Gamma_n$ and $\bar{d}(M,K) \geq \frac{1}{n}$. Therefore, the proof will be complete if we can show that $\mu(\pp^{-1}(K \times K))=0$ whenever $K \subset \rcomp{M}$ is compact and $\bar{d}(M,K)>0$.

Hereafter, let us fix such a set $K$. As compact subsets of $\rcomp{M}$ are $\bar{d}$-bounded by Proposition \ref{pr:rcomp_observations} \ref{pr:rcomp_observations_3}, we can fix $C>0$ such that $\bar{d}(0,\xi) \leq C$ whenever $\xi \in K$. Set 
$r = \min\set{\tfrac{1}{2},\frac{1}{2C}\bar{d}(M,K)}>0$. By Lemma \ref{lm:G_distance_functions_closed}, for $i=1,2$, there exist $g_i \in \Glsc$ such that
\begin{equation}\label{eqn:maybe_2}
g_i(\zeta) = \begin{cases}\displaystyle \min\set{1,\frac{\bar{d}(\pp_i(\zeta),K)}{d(\zeta)}} & \text{if } \zeta \in \pp_i^{-1}(\rcomp{M}\setminus K),\\
0 & \text{if } \zeta \in \pp_i^{-1}(K).
\end{cases}
\end{equation}
Then $g_1+g_2 \in \Glsc$. Let $\zeta \in \pp^{-1}(M \times M)$. We claim that $(g_1+g_2)(\zeta) \geq r$. If $g_i(\zeta)=1$ for some $i=1,2$, then we are done, so we assume otherwise. If $d(\zeta) \leq 4C$ then
\[
g_i(\zeta) = \frac{\bar{d}(\pp_i(\zeta),K)}{d(\zeta)} \geq \frac{\bar{d}(M,K)}{4C}, \quad i=1,2,
\]
giving $(g_1+g_2)(\zeta) \geq r$. Otherwise $d(\zeta) > 4C$, and by recalling \eqref{eq:R_finite_d} we obtain
\[
\bar{d}(\pp_1(\zeta),K)+\bar{d}(\pp_2(\zeta),K) \geq \bar{d}(\pp_1(\zeta),0)+\bar{d}(\pp_2(\zeta),0) - 2C\\  \geq d(\zeta) - 2C \geq \tfrac{1}{2}d(\zeta),
\]
giving $(g_1+g_2)(\zeta) \geq \frac{1}{2} \geq r$. 

Set $f=\min\{-r\cdot\mathbf{1}_{\bwt{M}},-(g_1+g_2)\}$. As $f$ is the minimum of two upper semicontinuous functions that are constant on fibres of $\gq$, it enjoys these properties as well. Fix $k=-r\cdot\mathbf{1}_{\pp^{-1}(K \times K)}$. We claim that $f+g_1+g_2 \leq k$. Indeed, $f+g_1+g_2 \leq 0$ by definition, and if $\zeta \in \pp^{-1}(K \times K)$ then $g_1(\zeta)=g_2(\zeta)=0$, so $f(\zeta)=-r=k(\zeta)$.  We further claim that if $\zeta \in \pp^{-1}(M \times M)$ then $(f+g_1+g_2)(\zeta)= 0$. Given such $\zeta$, as $(g_1+g_2)(\zeta) \geq r$, we have $f(\zeta)=-(g_1+g_2)(\zeta)$, thus $(f+g_1+g_2)(\zeta)= 0$.

Now let $\ep>0$.  By Theorem \ref{th:minimality_test} \ref{pr:minimality_test_3} and the minimality of $\mu$, there exists $g_0 \in G$ such that $f \leq g_0$ and
\[
\duality{g_0,\mu} < \int_{\bwt{M}} f\,d\mu + r\ep. 
\]
Since $g_0+g_1+g_2 \in \Glsc$ and $(g_0 + g_1+g_2)(\zeta) \geq (f + g_1+g_2)(\zeta)=0$ whenever $\zeta \in \pp^{-1}(M \times M)$, by Proposition \ref{pr:Glsc_lipfree} and the above we obtain
\[
0 \leq \int_{\bwt{M}} g_0+g_1+g_2 \,d\mu < \int_{\bwt{M}} f+g_1+g_2 \,d\mu + r\ep \leq \int_{\bwt{M}} k \,d\mu + r\ep = -r\mu(\pp^{-1}(K \times K)) + r\ep.
\]
It follows that $\mu(\pp^{-1}(K \times K)) < \ep$. As this holds for all $\ep>0$, the proof is complete.
\end{proof}

Some more preliminary work is needed before we can prove Theorem \ref{th:min_conc}. The next step is to find another necessary condition for an element to belong to $\lipfree{M}$. This time, the condition will be expressed in terms of sufficiently well-behaved elements of $\Lip_0(\rcomp{M},\bar{d})$. 

We start by defining the linear map $\Psi:\Lip_0(\rcomp{M},\bar{d})\to\ell_\infty(\bwt{M})$ by
\[
 (\Psi k)(\zeta) = \begin{cases} \dfrac{k(\pp_1(\zeta))-k(\pp_2(\zeta))}{d(\zeta)} & \text{if }\zeta \in \bwtf\setminus d^{-1}(0) \\ 0 & \text{otherwise.} \end{cases}
\]
We show first that $\norm{\Psi k}_\infty = \lipnorm{k}$. Given $\zeta \in \bwtf\setminus d^{-1}(0)$, we have $\bar{d}(\pp(\zeta)) \leq d(\zeta)$ by Proposition \ref{pr:De_Leeuw_relation} \ref{pr:De_Leeuw_relation_3} and thus
\[
 |\Psi k(\zeta)| \leq \frac{\bar{d}(\pp(\zeta))\lipnorm{k}}{d(\zeta)} \leq \lipnorm{k}.
\]
Conversely, given distinct $\xi,\eta \in \rcomp{M}$, using Proposition \ref{pr:optimals_everywhere} we find $\zeta \in \opt \cap \pp^{-1}(\xi,\eta)$, so that
\[
\norm{\Psi k}_\infty \geq \left| \frac{k(\pp_1(\zeta))-k(\pp_2(\zeta))}{d(\zeta)}\right| = \frac{|k(\xi)-k(\eta)|}{\bar{d}(\xi,\eta)},
\]
again by Proposition \ref{pr:De_Leeuw_relation} \ref{pr:De_Leeuw_relation_3}. Therefore $\norm{\Psi k}_\infty = \lipnorm{k}$.

We are interested specifically in those functions belonging to
\[
\mathcal{X} :=\set{k \in \Lip_0(\rcomp{M},\bar{d}) \,:\, \text{$k$ is Borel}}.
\]
It is clear that $\Psi k$ is Borel whenever $k \in \mathcal{X}$, and thus $\Psi k$ can be viewed naturally as an element of $C(\bwt{M})^{**}$ via integration.

We want to focus on some particular subspaces of $\mathcal{X}$. Given $n \in \NN$, define
\begin{equation}\label{eqn:M_r}
M_n = \set{\xi \in \rcomp{M} \,:\, \bar{d}(\xi,M) \leq \tfrac{1}{n}}.
\end{equation}
The $M_n$ are $G_\delta$ sets (and hence Borel) because $\bar{d}(\cdot,M)$ is upper semicontinuous on $\rcomp{M}$. Define the subspaces
\[
 \mathcal{X}_n = \set{k \in \mathcal{X} \,:\, \text{$k(\xi)=0$ whenever $\xi \in M_n$}}, \qquad n \in \NN.
\]

The next lemma is another of the key tools needed to prove Theorem \ref{th:min_conc}. It has its roots in \cite[Theorem 4.11]{AP_measures}. It is stated in greater generality than is needed for its immediate purpose, but this generalisation only requires a marginally longer proof, and it may have future applications. Recall the definition of $\Delta(A)$ given just before Proposition \ref{pr:delta}.

\begin{lemma}\label{lm:Psi_norming}
Let $\mu \in \meas{\bwt{M}}$ be concentrated on $\bwtf\setminus \Delta(\rcomp{M}\setminus M)$. Then
 \begin{equation}\label{eqn:Psi_norming}
  \sup\set{\int_{\bwt{M}} \Psi k \,d\mu : k \in \bigcup_{n=1}^\infty B_{\mathcal{X}_n}} \leq \inf\set{\norm{\Phi^*\mu - m} \,:\, m \in \lipfree{M}}.
 \end{equation}
\end{lemma}

\begin{proof}
To ease notation, set $\Delta = \Delta(\rcomp{M}\setminus M)$. Let $n \in \NN$, $k \in B_{\mathcal{X}_n}$, $m \in \lipfree{M}$ and $\ep>0$. Using \eqref{eq:positive_rep}, let $\nu \in \meas{\bwt{M}}^+$ be a (not necessarily optimal) representation of $m$ concentrated on $\wt{M}$, and set $\lambda=\mu-\nu$, so that $\Phi^*\lambda = \Phi^*\mu - m$. We have $\duality{\mu,\Psi k}=\duality{\lambda,\Psi k}$ because $k$ vanishes on $M$ and $\nu$ is concentrated on $\wt{M}$. 

Given $\ep>0$, by Lusin's theorem \cite[Theorem 7.1.13]{Bogachev} applied to both marginals of $|\lambda|$ (which are concentrated on $\rcomp{M}$), there exists a compact set $K \subset \rcomp{M}$ such that $k\restrict_K$ is continuous and $|\lambda|(H) < \frac{1}{2}\ep$, where
\[
 H:= \pp_1^{-1}(\rcomp{M}\setminus K) \cup \pp_2^{-1}(\rcomp{M}\setminus K).
\]
Since $k\restrict_{K \cup \set{0}}$ is also continuous, we can and do suppose that $0 \in K$. Furthermore, by inner regularity of the marginals of $|\lambda|$, we can assume that $K \cap M$ and $K\setminus M$ are also compact. 

In the next step, we enlarge $K \cap M$ by defining the set
\[
 L=\set{\xi \in \rcomp{M} \,:\, \bar{d}(\xi,K \cap M) \leq s},
\]
for a suitably chosen $s>0$, in such a way that $L \subset M_n$ and $k\restrict_{L \cup (K\setminus M)}=k\restrict_{L \cup K}$ is continuous as well. By Proposition \ref{pr:bar-d_cont} \ref{pr:bar-d_cont_1}, $L$ is closed (in fact it is easily seen to be compact). Using Proposition \ref{pr:bar-d_cont} \ref{pr:bar-d_cont_1} and compactness, we have $r:=\bar{d}(K \cap M,K\setminus M)>0$. Thus if we set $s=\min\set{\frac{1}{n},\frac{r}{2}}$, then $L \subset M_n$ and, by a straightforward argument, we have $\bar{d}(L,K\setminus M) \geq \frac{r}{2}$ and thus $L \cap K\setminus M = \varnothing$. Since $k\restrict_L$ vanishes and is hence continuous, $k\restrict_{K\setminus M}$ is continuous, and $L$ and $K\setminus M$ are closed and disjoint, it follows that $k\restrict_{L \cup (K\setminus M)}=k\restrict_{L \cup K}$ is continuous.

Since $k$ vanishes on $L$ and $k\restrict_{K\setminus M}$ is compact, $k$ is bounded on $L \cup K$ (alternatively, just consider the compactness of $L$). By Proposition \ref{pr:Mat_app}, there exists a bounded function $f \in B_{\Lip_0(M)}$ such that $\rcomp{f}\restrict_{L \cup K}=k\restrict_{L \cup K}$. We claim that 
\begin{equation}\label{lm:Psi_norming_1}
 \Phi f\restrict_{\bwtf\setminus (\Delta \cup H)} = \Psi k\restrict_{\bwtf\setminus (\Delta \cup H)}.
\end{equation}
To prove this claim, let $\zeta \in \bwtf\setminus (\Delta \cup H)$. First assume that $d(\zeta)>0$. Then because $\zeta \in \bwtf\setminus H=\pp^{-1}(K\times K)$, we have by Proposition \ref{pr:De_Leeuw_relation} \ref{pr:De_Leeuw_relation_1}
\[
 \Phi f(\zeta) = \frac{\rcomp{f}(\pp_1(\zeta))-\rcomp{f}(\pp_2(\zeta))}{d(\zeta)} = \frac{k(\pp_1(\zeta))-k(\pp_2(\zeta))}{d(\zeta)} = \Psi k(\zeta).
\]
Now suppose that $d(\zeta)=0$. Then $\Psi k(\zeta)=0$, so we need to show that $\Phi f(\zeta)=0$ as well. As $\pp_1(\zeta)=\pp_2(\zeta)$ (by Proposition \ref{pr:De_Leeuw_relation} \ref{pr:De_Leeuw_relation_3}) and $\zeta \notin \Delta$, it follows that $\zeta \in \pp^{-1}(M \times M)$, thus there exists $x \in K \cap M$ such that $\pp_1(\zeta)=\pp_2(\zeta)=x$. Let $(x_i,y_i)$ be a net in $\wt{M}$ converging to $\zeta$. Then $d(x_i,x),d(y_i,x) \leq s$ for all large enough $i$, which implies $x_i,y_i \in L \cap M$ and $f(x_i)=f(y_i)=0$ for these $i$. Therefore $\Phi f(\zeta)=0$ as required, and we have established \eqref{lm:Psi_norming_1}.

Since \eqref{lm:Psi_norming_1} holds and $\lambda$ is concentrated on $R\setminus \Delta$, we can estimate
\begin{align*}
 \duality{\mu,\Psi k}=\duality{\lambda,\Psi k} &\leq \norm{\Phi^*\lambda} + |\duality{\lambda,\Psi k} - \duality{\Phi f,\lambda}|\\
 &= \norm{\Phi^*\lambda} + \bigg|\int_H \Psi k - \Phi f \,d\lambda \bigg|\\
 &\leq \norm{\Phi^*\mu - m} + 2|\lambda|(H) < \norm{\Phi^*\mu - m} + \ep.
\end{align*}
This holds for all $n \in \NN$, $k \in B_{\mathcal{X}_n}$, $m \in \lipfree{M}$ and $\ep>0$, hence result.
\end{proof}

Finally, we can prove Theorem \ref{th:min_conc}.

\begin{proof}[Proof of Theorem \ref{th:min_conc}]
Let $\mu \in \opr{\bwt{M}}$ be minimal and satisfy $\Phi^*\mu \in \lipfree{M}$. By Proposition \ref{co:free_bwtf}, $\mu$ is concentrated on $\bwtf$. By Theorem \ref{th:kill_Gamma}, $\mu(\Gamma)=0$. Thus, to finish the proof, it suffices to show
\[
\mu(\pp^{-1}((\rcomp{M}\setminus M) \times M))=\mu(\pp^{-1}(M \times (\rcomp{M}\setminus M)))=0.
\]
In other words, the set of $\zeta \in \bwtf$ having exactly one coordinate in $M$ is $\mu$-null. By symmetry, it will be enough to show that $\mu(\pp^{-1}((\rcomp{M}\setminus M) \times M))=0$.

To prove this we will use Corollary \ref{co:mutually_singular_upgrade} and Lemma \ref{lm:Psi_norming}. The fact that $\mu(\Gamma)=0$ implies $\mu(\Delta(\rcomp{M}\setminus M))=0$. Hence, by Corollary \ref{co:mutually_singular_upgrade},
\[
((\pp_1)_\sharp\mu)\restrict_{\rcomp{M}\setminus M} \perp ((\pp_2)_\sharp\mu)\restrict_{\rcomp{M}\setminus M}.
\]
This means there exists a Borel set $A \subset \rcomp{M}\setminus M$ satisfying $\mu(\pp_2^{-1}(A))=0$ and $\mu(\pp_1^{-1}(E))=\mu(\pp_1^{-1}(A \cap E))$ whenever $E \subset \rcomp{M}\setminus M$ is Borel. Thus we need only to show that $\mu(\pp^{-1}(A \times M))=0$. Using a measure-theoretic argument similar to the one given at the start of the proof of Theorem \ref{th:kill_Gamma}, it will be sufficient to show that $\mu(\pp^{-1}(K \times M))=0$ whenever $K \subset A$ is compact and $\bar{d}(M,K)>0$.

Let $K$ be such a set and fix $N \in \NN$ such that $\frac{2}{N}\leq \bar{d}(M,K)$. The function $\bar{d}(\cdot,K)$ is lower semicontinuous by Proposition \ref{pr:bar-d_cont} \ref{pr:bar-d_cont_1}. Given $n \in \NN$, define the $\bar{d}$-Lipschitz and upper semicontinuous (and thus Borel) function $k_n:\rcomp{M}\to[0,1]$ by $k_n(\xi)=\max\set{0,1-n\bar{d}(\xi,K)}$. We observe that
\[
W_n := \set{\xi \in \rcomp{M} \,:\, \bar{d}(\xi,K) < \tfrac{1}{n}} = k_n^{-1}(0,1].
\]
Let $n \geq N$ and $\xi \in W_n$. Then $\bar{d}(\xi,\eta)<\frac{1}{n}$ for some $\eta \in K$, whence
\[
 \bar{d}(\xi,M) \geq \bar{d}(\eta,M) - \bar{d}(\xi,\eta) \geq \bar{d}(M,K) - \bar{d}(\xi,\eta) > \tfrac{2}{N}-\tfrac{1}{n} \geq \tfrac{1}{N}.
\]
Hence $W_n \cap M_N = \varnothing$, which implies $k_n \in \mathcal{X}_N$ whenever $n \geq N$. According to Lemma \ref{lm:Psi_norming}, because $\mu$ is concentrated on $\bwtf\setminus \Delta(\rcomp{M}\setminus M)$ and $\Phi^*\mu \in \lipfree{M}$, the integrals $\int_{\bwt{M}} \Psi k_n \,d\mu$ vanish whenever $n \geq N$.

Define the Borel function $\psi:\bwt{M} \to\RR$ by
\[
 \psi(\zeta) = \begin{cases} \frac{1}{d(\zeta)} & \text{if } \zeta \in \pp^{-1}(K \times (\rcomp{M}\setminus K)),\\
 -\frac{1}{d(\zeta)} & \text{if } \zeta \in \pp^{-1}((\rcomp{M}\setminus K) \times K),\\
 0 & \text{otherwise.}\end{cases}
\]
We claim that the sequence $(\Psi k_n)$ converges pointwise to $\psi$ on $\bwt{M}$. Let $\zeta \in \bwt{M}$. If $\zeta \notin \bwtf$ then $\Psi k_n(\zeta)=0=\psi(\zeta)$ for all $n$, so assume hereafter that $\zeta \in \bwtf$. It is clear that if $\xi \in K$ then $k_n(\xi)=1$  for all $n$, else $k_n(\xi)=0$ for all large enough $n$, so from these facts we quickly deduce that $\Psi k_n(\zeta) \to \psi(\zeta)$ whenever $\zeta \in \bwtf$.

We want to use the dominated convergence theorem, so we need to dominate $\psi$ and the $\Psi k_n$, $n \geq N$, by a $\mu$-integrable function. Define $\theta:\bwt{M} \to\RR$ by
\[
 \theta(\zeta) = \begin{cases} \frac{1}{d(\zeta)} & \text{if } \zeta \in E:=\bwtf\setminus d^{-1}(0) \cap (\pp_1^{-1}(W_N) \cup \pp_2^{-1}(W_N)),\\
 0 & \text{otherwise.}\end{cases}
\]
Evidently $|\psi| \leq \theta$ and $|\Psi k_n| \leq \theta$ whenever $n \geq N$. Let us verify that $\theta$ is $\mu$-integrable. Let $\zeta \in E\setminus \Gamma$. If $\pp_1(\zeta) \in W_N$ then $\pp_1(\zeta) \notin M_N$. As $\zeta \in \bwtf\setminus \Gamma$, we must have $\pp_2(\zeta) \in M$, and thus by Proposition \ref{pr:De_Leeuw_relation} \ref{pr:De_Leeuw_relation_3} $d(\zeta) \geq \bar{d}(\pp(\zeta)) \geq \frac{1}{N}$, giving $\theta(\zeta) \leq N$. We reach the same conclusion if $\pp_2(\zeta) \in W_N$. Given this and the fact that $\mu(\Gamma)=0$,
\[
\int_{\bwt{M}} \theta \,d\mu = \int_\Gamma \theta \,d\mu + \int_{E\setminus\Gamma} \theta \,d\mu \leq N\norm{\mu} < \infty.
\]
Therefore we can apply the dominated convergence theorem to obtain
\[
\int_{\bwt{M}} \psi \,d\mu = \lim_n \int_{\bwt{M}} \Psi k_n \,d\mu = 0.
\]
Again given that $\mu(\Gamma)=0$ and recalling that $\mu(\pp_2^{-1}(K))=\mu(\pp_2^{-1}(A))=0$, the above reduces to
\[
\int_{\pp^{-1}(K \times M)} \frac{d\mu}{d} = \int_{\bwt{M}} \psi \,d\mu = 0.
\]
Because $\frac{1}{d}$ is strictly positive on $\pp^{-1}(K \times M)$, we deduce that $\mu(\pp^{-1}(K \times M))=0$. This ends the proof.
\end{proof}

We can combine Theorem \ref{th:min_conc}, Proposition \ref{pr:basic} \ref{pr:basic_3} and \ref{pr:basic_4} and Proposition \ref{pr:minimal_exist}, together with part of the argument in the proof of Corollary \ref{co:nice_cim}, to obtain the final result of the section. Bearing in mind that not every free space element is a convex integral of molecules, it yields representations of such elements that are are good as one can hope for.

\begin{corollary}\label{co:nice_rep}
Let $m \in \lipfree{M}$ and set $A=\supp(m) \cup \set{0}$. Then $\mu$ is concentrated on $\pp^{-1}(A\times A)$ whenever $\mu \in \opr{\bwt{M}}$ is minimal and $\Phi^*\mu=m$. Consequently, $m$ has an optimal representation concentrated on $\pp^{-1}(A\times A)$.
\end{corollary}

As is written in Section \ref{subsec:background}, Corollary \ref{co:nice_rep} is the starting point for the proof of Theorem \ref{th:extreme} and a number of other applications to Lipschitz-free spaces; we refer the reader to \cite{APS24c} for a full treatment of them. 

\section*{Acknowledgements}

I would like to express my gratitude to R.~Aliaga and E.~Perneck\'a, with whom I have had countless interesting and fruitful mathematical discussions in recent years, and from whom I have learned much about Lipschitz-free spaces. Particular thanks are due to R.~Aliaga, who read a number of iterations of this work and whose valuable comments led to many improvements. Thanks go to M.~C\'uth as well, for supplying further comments and corrections.

\end{document}